\documentclass[10pt]{article}
\usepackage{tikz-cd}
\usepackage{quiver}
\usepackage{graphicx}
\usepackage{pgfkeys}
\usetikzlibrary{arrows}
\usepackage[english]{babel}
\usepackage[utf8]{inputenc}
\usetikzlibrary{calc,babel}
\usepackage[all,cmtip]{xy}
\usepackage{extarrows}
\usepackage{amsfonts}
\usepackage{amsmath}
\usepackage{yhmath}
\usepackage[
backend=biber,
style=alphabetic,
]{biblatex}
\usepackage{stmaryrd}
\usepackage{mathtools}
\usepackage{amssymb}
\usepackage{amsthm}
\usepackage[shortlabels]{enumitem}

\usepackage[top=1.5in, bottom=1in, left=1in, right=1in]{geometry}
\usepackage{setspace}
\usepackage[hidelinks]{hyperref} 
\usepackage[capitalize]{cleveref}
\usepackage{thmtools}
\hypersetup{hypertexnames=false}

\numberwithin{equation}{section}

\newtheorem{thm}[equation]{Theorem}
\newtheorem{lem}[thm]{Lemma}
\newtheorem*{clm}{Claim}
\newtheorem{cor}[thm]{Corollary}
\newtheorem{prop}[thm]{Proposition}

\newtheorem{que}[thm]{Question}
\newtheorem{notn}[thm]{Notation}

\theoremstyle{definition}
\newtheorem{defi}[thm]{Definition}
\newtheorem{example}[thm]{Example}

\theoremstyle{remark}
\newtheorem{rem}[thm]{Remark}

\newcommand{\N}{\mathbb N}

\newcommand{\Z}{\mathbb Z}

\newcommand{\E}{\mathcal{E}}

\newcommand{\K}{\mathbb{K}}

\newcommand{\s}{\mathcal{S}}

\newcommand{\res}{\operatorname{res}}
\newcommand{\ind}{\operatorname{ind}}

\newcommand{\kg}{{\mathcal{K}_G}}
\newcommand{\hatkg}{\hat{\mathcal{K}}_G}
\newcommand{\R}{\mathbb R}

\newcommand{\KKRGRAnon}{\wideKK_{!,\infty}}

\newcommand{\CAlg}{\mathsf C^{*}\!\mathsf{Alg}}
\newcommand{\CAlgGr}{\widehat{\mathsf C}^{*}\!\mathsf{Alg}}
\newcommand{\sep}{\mathrm{sep}}
\newcommand{\ex}{\mathrm{ex}}
\newcommand{\se}{\mathrm{se}}
\newcommand{\Kas}{\mathrm{Kas}}
\newcommand{\KK}{\mathrm{KK}}
\newcommand{\wideKK}{\widehat{\mathrm{KK}}}
\newcommand{\kk}{\mathrm{kk}}
\newcommand{\widekk}{\widehat{\mathrm{kk}}}
\newcommand{\ETheory}{\mathrm{E}}
\newcommand{\wideETheory}{\widehat{\mathrm{E}}}
\newcommand{\eTheory}{\mathrm{e}}
\newcommand{\wideeTheory}{\widehat{\mathrm{e}}}
\DeclareMathOperator{\Asy}{Asy}

\newcommand{\RALG}{\mathsf G\CAlg}
\newcommand{\RGraded}{\mathsf G\CAlgGr}
\newcommand{\RGradedsep}{\mathsf G\CAlgGr^{\sep}}

\newcommand{\GTWOALG}{\mathsf G_2\CAlg}
\newcommand{\RALGsep}{\mathsf G\CAlg^{\sep}}
\newcommand{\RALGsepH}{\mathsf H\CAlg^{\sep}}
\newcommand{\RALGsTWO}{\mathsf G_2\CAlg^{\sep}}
\newcommand{\RALGsTWOh}{\mathsf G_2\CAlg_{h}^{\sep}}
\newcommand{\RALGseph}{\mathsf G\CAlg_{h}^{\sep}}
\newcommand{\SKGRALGseph}{\mathrm S_{\kg}\mathsf G\CAlg_{h}^{\sep}}
\newcommand{\SKGexcl}{\mathrm S_{\kg}\mathsf G\CAlg_{h,!,\oplus}^{\sep}}
\newcommand{\hoKKse}{\KK^{G,\sep}_{\mathrm{se},0}}
\newcommand{\hokkse}{\kk^{G,\sep}_{\mathrm{se},0}}
\newcommand{\hoKKex}{\KK^{G,\sep}_{\ex,0}}
\newcommand{\ASYG}{\Asy^{G}}
\newcommand{\SKGASYG}{\mathrm S_{\kg}\Asy^{G}}

\newcommand{\KKANA}{\wideKK^{G}_{!,0}}
\newcommand{\kkANA}{\widekk^{G}_{!,0}}
\newcommand{\KKRGRA}{\wideKK^{G}_{!,\infty}}
\newcommand{\ER}{\KK^{G}_{\ex,\infty}}
\newcommand{\kkRGRA}{\widekk^{G}_{!,\infty}}
\newcommand{\KKANAnon}{\wideKK_{!,0}}
\newcommand{\coind}{\mathsf{coind}}
\newcommand{\Egrclass}{\wideETheory^{G}_{\mathrm{class}}}
\newcommand{\egrclass}{\wideeTheory^{G}_{\mathrm{class}}}
\newcommand{\EBGHT}{\ETheory^{G}_{\mathrm{class}}}
\newcommand{\eBGHT}{\eTheory^{G}_{\mathrm{class}}}

\newcommand{\KKKAS}{\KK_{\Kas}^{G,\sep}}
\newcommand{\kkKAS}{\kk_{\Kas}^{G,\sep}}
\newcommand{\KKKASgen}{\KK^{G}_{\Kas}}

\newcommand{\KKKASgr}{\wideKK_{\Kas}^{G,\sep}}

\newcommand{\hoKKsegen}{\KK^{G}_{\mathrm{se},0}}
\newcommand{\KKR}{\KK^{G}_{!,\infty}}
\newcommand{\KKRex}{\KK^{G_2}_{\ex,\infty}}
\newcommand{\KKRexTWO}{\KK^{C_2}_{\ex,\infty}}
\newcommand{\hoKKRex}{\KK^{G_2}_{\ex,0}}
\newcommand{\hoKKRexT}{\KK^{C_2}_{\ex,0}}
\newcommand{\kkR}{\kk^{G}_{!,\infty}}
\newcommand{\hoKKR}{\KK^{G}_{!,0}}

\newcommand{\KKRCTWO}{\KK^{C_2}_{!,\infty}}
\newcommand{\KKRGTWO}{\KK^{G_2}_{!,\infty}}
\newcommand{\KKRGTWOsep}{\KK_{!,\infty}^{G_2,\sep}}
\newcommand{\kkRGTWO}{\kk^{G_2}_{!,\infty}}
\newcommand{\hokkRsep}{\kk_{!,0}^{G,\sep}}
\newcommand{\hoKKRsep}{\KK_{!,0}^{G,\sep}}
\newcommand{\hoEgra}{\wideKK_{\ex,0}^{G,\sep}}
\newcommand{\Egra}{\wideKK_{\ex,\infty}^{G,\sep}}
\newcommand{\KKgr}{\wideKK_{\se,\infty}^{G,\sep}}
\newcommand{\hoKKgra}{\wideKK_{\mathrm{se},0}^{G,\sep}}

\newcommand{\hoKKCTWO}{\KK^{C_2}_{!,0}}
\newcommand{\hoKKGTWO}{\KK^{G_2}_{!,0}}

\newcommand{\KKRsep}{\KK_{!,\infty}^{G,\sep}}
\newcommand{\kkRsep}{\kk_{!,\infty}^{G,\sep}}

\newcommand{\KKse}{\KK^{G}_{\mathrm{se},\infty}}
\newcommand{\kkse}{\kk^{G}_{\mathrm{se},\infty}}

\newcommand{\C}{\mathbb{C}}

\DeclarePairedDelimiterX\braket[2]{\langle}{\rangle}{#1 \delimsize\vert #2}
\DeclarePairedDelimiterX\inr[2]{\langle}{\rangle}{#1 \delimsize, #2}

\makeatletter
\newcommand{\colim@}[2]{%
  \vtop{\m@th\ialign{##\cr
    \hfil$#1\operator@font colim$\hfil\cr
    \noalign{\nointerlineskip\kern1.5\ex@}#2\cr
    \noalign{\nointerlineskip\kern-\ex@}\cr}}%
}
\newcommand{\colim}{%
  \mathop{\mathpalette\colim@{\rightarrowfill@\textstyle}}\nmlimits@
}
\makeatother

\title{A higher categorical approach to graded bivariant $K$-theories for Real $C^*$-algebras}
\author{\scalebox{1.1}Anupam Datta\thanks{Universität Göttingen, \href{mailto:anupam.datta@mathematik.uni-goettingen.de}{anupam.datta@mathematik.uni-goettingen.de}}}
\date{}
\begin{document}
\maketitle
\begin{abstract}
    We present an $\infty$-categorical refinement of classical $KK$-theory and $E$-theory for graded, $G$-equivariant for a locally compact second countable group $G$, Real $C^*$-algebras.   As a consequence, we obtain a couple of fiber sequences connecting the different $K$-theory spectra, and explicitly analyse the induced endomorphisms on $K_0$-groups. We also answer precisely some folklore questions in the area. Among other things, we show that graded $KK$- and $E$-categories are not a Dwyer-Kan localization of graded $C^*$-algebras.  
\end{abstract}
\tableofcontents
\begingroup

\renewcommand{\C}{\mathbb C}
\renewcommand{\hatkg}{\hat{\mathcal{K}}_G}
\section{Introduction}
This article is devoted to studying bivariant $K$-theories for $C^*$-algebras found in an index theorist's toolkit using methods from higher category theory. Recent works of Bunke and coauthors (\cite{BEL}, \cite{KKandE}, \cite{EisCASS}) have refined the classical $KK$- and $E$-categories of $C^*$-algebras to stable $\infty$-categories, which then enjoy an $\infty$-categorical version of the classical ($1$-categorical) universal properties as mentioned in \cite{HIGFIRST} and \cite{CH}. This $\infty$-categorical enhancement is not just for the satisfaction of the modern homotopy theorist, but also helps segregate the homotopical and analytical aspects of the subject.

In \cite{KKandE}, Bunke constructs small stable $\infty$-categories $KK^{\operatorname{sep}}_{!,\infty}$ (where $! \in \{\operatorname{semi-split (se),exact(ex)}\}$ denotes the kind of excision we want to have) by universally enforcing certain properties Kasparov's $KK$- and Connes-Higson's $E$-theory are known to satisfy, namely, homotopy invariance, compact stability, $!$-excision, and Bott periodicity. The same construction can also be done equivariantly in the presence of the action of locally compact second countable group, as mentioned in \cite{BUNKEcrossed}. That the homotopy category of $KK^{\operatorname{sep}}_{\operatorname{se},\infty}$ (resp. $KK^{\operatorname{sep}}_{\operatorname{ex},\infty}$) is equivalent to Kasparov's $KK$- (resp. Connes-Higson's $E$-)categories can be seen by comparing with the classical universal properties, as is worked out in \cite[Section 13]{KKandE}. The advantage of \emph{defining} bivariant $K$-categories as a Dwyer-Kan localization is that one can write functors out of the $KK_{!,\infty}^{\operatorname{sep}}$-categories by verifying that a functor defined on separable $C^*$-algebras satisfies certain properties.

However, \cite{BEL}, \cite{KKandE}, \cite{EisCASS} have been written exclusively with \emph{ungraded} $C^*$-algebras. In contrast, most of classical bivariant $K$-theories have been applied to problems in geometry and topology in their \emph{graded} avatar. For example, consider the proof of the Baum-Connes conjecture for a-T-menable groups in Higson-Guentner's article \cite{HG}, where a myriad of the exotic functorial properties of the graded $E$-category were used.
Nevertheless, it is the {ungraded $KK_!$-theories}\footnote{we adopt Bunke's notation and denote by $KK_!$ either $KK$- or $E$-theory, even in the classical contexts.} whose universal properties are classically known. Given the importance of the graded theory and the effectiveness of realizing $KK_!$ as Dwyer-Kan localizations as mentioned earlier, one can ask the same for the graded $KK_!$-category. It turns out that the answer is in the negative, as shown in \Cref{prop: graded kk not dk loc}. The author does not know if this fact was known earlier. However, not all is lost, and higher categorical refinements of graded $KK_!$-categories can still be provided. Following ideas in an unpublished note of Bunke-Duenzinger (see also \cite[Section 2.4]{BUNKELUDEWIG}), we define an $\infty$-categorical refinement of graded, $G$-equivariant $KK_!$-category as \emph{a certain category of comodules over a distinguished cocommutative coalgebra} in $KK^{G_2}_{!,\infty}$, the $G \times C_2$-equivariant, ungraded bivariant $K$-category. It turns out that this picture captures the formal properties of the classical graded $KK_!$-categories. The details of this higher categorical refinement of graded bivariant $K$-categories is one of the cornerstones of this article, and is developed in Section 4. 

The other aspect missing from the writings of Bunke on higher categorical approaches to $KK_!$-theories is a discussion of Real $C^*$-algebras. Apart from being of independent interest, Real $K(K)$-theory encodes finer information than its complex counterpart, as for example is clear from the seminal works of Hitchin (\cite{HITCHIN}) and Rosenberg (\cite{ROSENBERG}). Luckily, the construction of Real $KK_{!,\infty}$-categories for separable $C^*$-algebras follow the same path as in \cite{KKandE} (see also \cite[Section 3.4]{EisCASS}), with only one difference: we would have to enforce the well known $(1,1)$-periodicity at the final step of the construction. The construction of stable $\infty$-categories refining classical Real $KK_!$-categories is the content of the first few subsections of Section 3.\footnote{In our treatment of graded $KK_!$-categories, we have worked with Real $C^*$-algebras equipped with the action of an external locally compact group. We did not mention this in the last paragraph because the passage from ungraded to graded is much more drastic than the one from complex to Real. And equivariant matters were already considered in works of Bunke.} \emph{Unless otherwise mentioned, in this article we work with Real $C^*$-algebras by default.}

We now mention some applications of the abstract theory developed in Sections 3 and 4 of this article. In order to streamline the presentation, we package all of them in the next theorem. To simplify the exposition, we also ignore the external (compact) group action here. The (more or less straightforward) generalization to the equivariant case is present in Section 5.
\begin{thm}\label{thm: intro endo}
    Denote by $\K(-),\K^{C_2}(-),\hat{\K}(-)$ the ungraded (classical), $C_2$-equivariant, and graded spectrum valued $K$-theory functors on the corresponding categories of (Real) $C^*$-algebras. Let $A$ be a $C^*$-algebra with a $C_2$-action $\alpha_A$. Then we have the following\footnote{the labels $\mathbf{1},\mathbf{2},\mathbf{3},\mathbf{4}$ are just notation and should not be confused with the multiplication with some constant maps!}.
    \begin{enumerate}
        \item There is a natural fiber sequence of spectra
        \[
        \K(A_u) \xrightarrow{\mathbf{1}} \K^{C_2}(A) \xrightarrow{\mathbf{2}} \hat{\K}(A)
        \tag{ABS}
        \]
         Here $(-)_u$ denotes the underlying $C^*$-algebra functor. Furthermore, the fiber sequence splits if $A$ is inner graded.
        \item There is a natural fiber sequence of spectra
        \[
        \hat{\K}(A) \xrightarrow{\mathbf{3}}\K^{C_2}(A) \xrightarrow{\mathbf{4}} \K((A \hat\otimes Cl^{0,2})_u) \tag{van Daele}
        \]
        Here $A \hat\otimes Cl^{0,2}$ denotes the graded tensor product of $A$ and $Cl^{0,2}$, the Clifford algebra on two generators which square to $+1$. If $A$ is inner graded, the fiber sequence splits.
        \item There are endomorphisms of the different $K$-theory spectra of $A$ arising from the above two fiber sequences. At the level of $\pi_0$, their effect can be concretely described when the algebra $A$ is unital:
        \begin{enumerate}
            \item $\mathbf{3} \circ \mathbf{2} \in \operatorname{End}(K_0^{C_2})(A)$ is given by mapping the class of a $C_2$-equivariant projective $A$-module $[M,\alpha_M] \in K_0^{C_2}(A)$ to $[M,\alpha_M] - [M,-\alpha_M]$.
            \item $\mathbf{2} \circ \mathbf{3} \in \operatorname{End}(\hat{K}_0(A))$ is given by mapping the class of a Fredholm cycle\footnote{see \Cref{defi: fred cycle}.} $(\E,F) \in K_0^{\operatorname{Fr}}(A)$ to $(\E \oplus \E,F \oplus  -F) \in K_0^{\operatorname{Fr}}(A)$.
            \item When $A$ is inner graded \emph{or} $A$ is complex, then $\K((A \hat\otimes Cl^{0,2})_u) \simeq \K(A_u)$, and the map $\mathbf{4} \circ \mathbf{1} \in \operatorname{End}(K_0(A_u))$ is given by mapping the class of a finitely generated projective $A$-module $[M] \in K_0(A)$ to $[M \oplus M_\alpha] \in K_0(A)$, where $M_\alpha$ has the same underlying abelian group $\underline{M}$ as $M$, and the $A$-action is given by $\underline{m}.a:=\underline{m\alpha_A(a)}.$
            \item Under the same assumption on $A$ as in the last point, the map $\mathbf{1} \circ \mathbf{4} \in \operatorname{End}(K_0^{C_2}(A))$ is given by mapping the class of a $C_2$-equivariant finitely generated projective $A$-module $[M,\alpha_M] \in K_0^{C_2}(A)$ to $[M \oplus M, (m,n) \mapsto (\alpha_M(n),\alpha_M(m))] \in K_0^{C_2}(A)$.
        \end{enumerate}
    \end{enumerate}
\end{thm}
Before describing the organization of this article, we mention the relation between our fiber sequences and the \emph{Clifford module quotient}\footnote{the author learnt this sleek term from a talk of Luuk Stehouwer at the Max Planck Institute for Mathematics in 2025.} picture of $K$-theory of a point as described by Atiyah, Bott and Shapiro in \cite{ABS}. Roughly speaking, one of the main results of loc.\! cit.\! (see also \cite[Section I.9]{LM}) is the following (\emph{recall that $Cl^{p,q}$ denotes the Clifford algebra on $p+q$ generators, where $p$ generators square to $-1$, and $q$ generators square to $+1$.})
\begin{thm}{(Atiyah-Bott-Shapiro)}\label{thm: Atiyah-Bott-Shapiro}
    For every non-negative integer $n$, there is a canonical isomorphism of abelian groups
    \[
    K^{-n}(*) \cong \frac{\{\operatorname{finitely~generated~graded~Cl^{n,0}-modules}\}^{\operatorname{gp}}}{\{\operatorname{finitely~generated~graded~Cl^{n+1,0}-modules}\}^{\operatorname{gp}}}
    \]
\end{thm}
It can be seen (\Cref{rem: abs}) that \Cref{thm: Atiyah-Bott-Shapiro} is asserting that the long exact sequence of homotopy groups of the fiber sequence labeled ``ABS" splits off into short exact sequences when $A$ is a Clifford algebra. Our results therefore clarifies the situation for a general graded, Real ($G$-equivariant for a compact group $G$) $C^*$-algebra: \emph{replace the process of taking cokernels in abelian groups by taking cofiber in spectra}. This also answers a folklore question pertaining to whether higher $K$-groups of a $C^*$-algebra can be given a Clifford module quotient description analogous to \Cref{thm: Atiyah-Bott-Shapiro}; see \Cref{rem: abs} for details.

As the reader may guess, the other fiber sequence labeled ``van Daele" has been tagged so because it recovers the classical definition of van Daele $K$-theory (\cite{vDone},\cite{vDtwo}). After a bit of massaging\footnote{by working with the (graded) opposite $C^*$-algebra}, this sequence can be seen to recover Karoubi's definition of graded $K$-theory (\cite[Chapitre II.2.1]{KAR}) as well. We do not delve into these comparisons (and the inter-comparisons) in this article, and point the curious reader to the work of David Aretz and Luuk Stehouwer (\cite{AS}), where such considerations have been discussed in detail. Note that loc.\! cit.\! is written in the generality of Banach algebras, and our fiber sequence ``ABS" is more or less their \emph{definition} of the graded $K$-theory spectrum (see \cite[Definitions 3.42 \& 3.55]{AS}).

Finally, we mention that Ulrich Haag has works (\cite{HAAG2},\cite{HAAG1}) in this area. Therein, Haag obtains a couple of long exact sequences involving the different flavors of $K$-groups, but only for \emph{complex} $C^*$-algebras. It can be seen that exactly one of Haag's sequences survives the Real case, and the other one does not. The sequence that survives can be readily identified with the long exact sequence of homotopy groups of our sequence ``ABS". In the author’s view, some of the proofs presented in his papers would benefit from a more detailed and transparent exposition. Hence the methods of higher categories used in this paper both generalize and clarify Haag's works.
\subsection{Organization of the article}
We now discuss the organization of this article.

Section 2 contains well-known results from equivariant and graded $K$-theory. The key highlights of this section are \Cref{thm: spectral model of g eq k theory} and \Cref{prop: omnibus graded k-theory}. The non-equivariant case of \Cref{thm: spectral model of g eq k theory} is the content of \cite{D}, and the equivariant generalization is more or less straightforward. \Cref{prop: omnibus graded k-theory} contains some miscellaneous results on the graded $K$-theory of a distinguished graded $C^*$-algebra $\s$, and the results therein are going to be really crucial in the later sections.

In Section 3, we emulate the analogous chapters of \cite{KKandE}, and construct a stable $\infty$-categorical refinement of the $G$-equivariant, Real $KK$- and $E$-categories. We actually combine the approaches of \cite{KKandE} and \cite{EisCASS}; Section 3.3 is taken from the relevant section of \cite{EisCASS}. In Section 3.4, we (have to) deviate from the writings of Bunke and work with a Real variant of the Toeplitz extension so as to enforce $(1,1)$-periodicity. The category so obtained is denoted by $\KKRsep$, where $!\in \{\operatorname{se,ex}\}$ denotes whether we enforced excision for semi-split exact or all exact sequences. Section 3.5 discusses the well-known change of group functors: restriction, induction and coinduction. These come in handy while studying the graded theory, although in the simple case with the groups involved being $G$ and $G_2:=G \times C_2$. Section 3.6 makes connections with the models of equivariant and graded $K$-theory as studied in Section 2 with certain corepresentable functors out of $\KKRsep$ (see \Cref{thm: Recovering classical K-theories from our stable inf cat} for a precise statement). Finally in Section 3.7, the extension of $\KKRsep$ to the large category of all $G$-$C^*$-algebras is considered.

Section 4 contains the higher categorical approach to defining the graded bivariant $K$-categories, which we denote by $\KKRGRA$. \emph{Without apriori knowing that the homotopy category of $\KKRGRA$ recovers classical graded $KK$- and $E$-theory according as $!=\operatorname{se}$ or $!=\operatorname{ex}$}, we show that the formal consequences of our definition, together with a calculation in classical graded $K$-theory yields the familiar equivalence between $C_0(\R^{0,1})$ and $Cl^{1,0}$ in $\KKRGRA$. We also connect the spectral model of $K$-theory as discussed in Section 2 with the functor corepresented by $\C$ in $\KKRGRA$.

Section 5 is devoted to deducing \Cref{thm: intro endo} as a consequence of the machinery developed in Sections 3 and 4. As advertised before, in this section we also discuss the folklore question of representing graded $K$-groups of a (graded), (Real) unital $C^*$-algebra in terms of Clifford module quotients.

Finally in Section 6, we compare the homotopy categories of the ungraded and graded $KK^G_{!,\infty}$-categories we developed with the classical ungraded and graded $KK^G$- and $E^G$-categories (upon restricting to separable $C^*$-algebras; the general comparison is discussed in \Cref{cor: KK large same}). We give more details in the context of $E$-theory: the well-known universal property of ungraded $E$-theory lacks a proof in the foundational writings \cite{CH} and \cite{HGT}, and a detailed proof (following the work of Thomsen in \cite{THOASY}) is provided in Section 6.2. In Section 6.3, we discuss an equivariant generalization of a technical result of Dadarlat and Loring, which is an intermediary in making the comparison of the homotopy category of $\Egra$ (denoted $\hoEgra$) and $\Egrclass$. We mention that the comparison of the homotopy category of $\KKgr$ (denoted $\hoKKgra$) and $\KKKASgr$ is not directly possible using the techniques we used for their $E$-theoretic counterparts, but instead can be deduced from results of Meyer (\cite{MEYER}). \\\\
\textit{Acknowledgments: The author is deeply grateful to Johannes Ebert for suggesting this project, his generous guidance throughout, and his thoughtful comments on earlier drafts. Many thanks to David Aretz and Luuk Stehouwer for having several pleasant discussions on topics related to this article, and also for freely sharing some ideas from their recent (and related) project. The author would also like to thank  Ulrich Bunke, Benjamin Duenzinger, Michael Joachim, Matthias Lesch, Yuezhao Li and Maxime Ramzi for discussions at several points which clarified some mathematical issues.}

\textit{
This paper forms part of the PhD thesis of the author at the University of Bonn. Throughout the duration of his PhD, the author was supported by the Deutsche Forschungsgemeinschaft (DFG, German
Research Foundation) under Germany's Excellence Strategy – GZ 2047/1 -
390685813, Hausdorff Center for Mathematics.
}
\section{Background on $K$-theory spaces, groups, and some calculations}
In this section, we discuss the necessary background from classical / equivariant / graded $K$-theory that will be needed in the later sections. We emphasize that (morally) $K$-theory is a significantly simpler concept than $KK$-theory, and we use our knowledge of $K$-theory as discussed in this section later while studying the bivariant $K$-categories.

Before proceeding further, we have to introduce the main objects of the entire project, namely, \emph{graded, Real, equivariant} $C^*$-algebras. 
\begin{defi}
    \hfill
    \begin{enumerate}
        \item  A graded complex $C^*$-algebra $(A,\varepsilon)$ is a complex $C^*$-algebra $A$ equipped with a $*$-automorphism $\varepsilon$ such that $\varepsilon^2=1$. $\varepsilon$ is called the grading on $A$. The elements of the $+1$ and $-1$ eigenspaces of $A$ are called even and odd, respectively.
        \item For $G$ a locally compact second countable group, a complex $G$-$C^*$-algebra $(A,\alpha)$ is a $C^*$-algebra $A$ together with a continuous homomorphism $\alpha: G \rightarrow \operatorname{Aut}(A)$, where $\operatorname{Aut(A)}$ is equipped with the point-norm topology.
        \item A Real $C^*$-algebra $(A,\tau)$ is a complex $C^*$-algebra $A$ equipped with a $\C$-anti-linear $*$-automorphism $\tau$ such that $\tau^2=1$. $\tau$ is called the Real structure on $A$. \footnote{Sometimes, the Real structure will simply be denoted by $a \mapsto \bar{a}$.} The real $C^*$-algebra associated to $(A,\tau)$ is $A_{\tau}:=\{ a \in A \mid a = \tau(a)\}$, and it can be shown that this induces an equivalence of categories between real and Real $C^*$-algebras.
        \item For any two (or all three) of the adjectives defined before, we combine them in the obvious way by requiring the corresponding structure maps to commute.
    \end{enumerate}
\end{defi}
\begin{notn}\label{notn: prelim}
    We follow the following conventions throughout this article.
    \begin{itemize}
        \item Unless otherwise mentioned, $G$ will denote a locally compact second countable group. 
        \item As already mentioned in the introduction, we will by default be working be working with Real $C^*$-algebras. Results in the complex case can be deduced by simply ignoring the Real structure. We \emph{do not} mention the word ``Real" unless there is an explicit need to do so.
        \item We denote the category of (ungraded) Real $G$-$C^*$-algebras (with the obvious morphisms) by $\RALG$. This category is enriched over topological spaces using the point-norm topology on the hom sets (see \cite[Section 2.2]{D} for details), and the mapping space is denoted by $\underline{\operatorname{Hom}}_G(A,B)$. The set of path componentns of $\underline{\operatorname{Hom}}_G(A,B)$ is denoted by $[A,B]_G$. If it is clear from the context, we suppress the $G$ from the notation.
        \item The symmetric monoidal structures on $\RALG$ will be denoted by $-\otimes_? -$, where $? \in \{\min,\max\}$ corresponds to the minimal and maximal tensor products. If an algebra under consideration is nuclear, we often omit the decoration from the tensor product.
        \item We denote the category of graded, Real $G$-$C^*$-algebras (with the obvious morphisms) by $\RGraded$. Its symmetric monoidal structures will be denoted by $-\hat\otimes_? -$, where $? \in \{\min,\max\}$ corresponds to the graded minimal and maximal tensor products. If an algebra under consideration is nuclear, we often omit the decoration from the tensor product (see also \cite[Corollary E.2.19]{WY}). $\RGraded$ is also topologically enriched in a way analogous to the ungraded case mentioned above; the only difference being that the desired equivariance now is with respect to $G \times C_2$.
        \item Finally, we remark that in exceptional cases where we do have to work with the category of complex $C^*$-algebras (for example, in \Cref{cor: endomorphisms in the complex graded case}), we keep using the same notation as above. This should not create any confusion.
    \end{itemize}
\end{notn}
We next mention some examples of graded $C^*$-algebras which we will often encounter. Further examples and discussions can be found in \cite[]{D}
\begin{example}\label{example: basic examples}
    \hfill
    \begin{enumerate}
        \item The base field $\C$ is trivially graded, and is always equipped with the trivial $G$-action. \emph{In the Real situation, the Real structure on $\C$ is given by pointwise complex conjugation, and is omitted from the notation}.
        \item The Clifford algebras come with a natural grading, and carry the trivial action of the external group $G$.
        \item $\s:= C_0(\R)$ with the even-odd grading, trivial $G$-action, and the Real structure given by pointwise complex conjugation is going to be a very important algebra for us. By choosing an odd, increasing homeomorphism between $(-1,1)$ and $\R$ ({the choice of this homeomorphism does not matter as the space of all such homeomorphisms is contractible}), we obtain an isomorphism between $\s$ and $\tilde{\s}$, the \emph{bounded transform} of $\s$. We mention another perspective on $\tilde{\s}$ below, which comes very handy in concrete situations. Consider the map
        \begin{align*}
    \Phi:\Tilde{\s} &\rightarrow \{f:[0,1) \rightarrow Cl^{0,1} \mid f(0) \in \mathbb{C} \subseteq Cl^{0,1}\} \subseteq C_0([0,1),Cl^{0,1})\\
    f&\mapsto f_{\operatorname{ev}}\mid_{[0,1)} + \varepsilon f_{\operatorname{odd}}\mid_{[0,1)}, 
\end{align*}
where $f_{\operatorname{ev}},f_{\operatorname{odd}}$ denote the even and odd parts of $f$ respectively, and $\varepsilon$ is the non-trivial generator of $Cl^{0,1}$. $\Phi$ can directly be checked to be an isomorphism in $\CAlgGr$.
        \item A grading on a $C^*$-algebra is said to be \emph{inner} if it is given by conjugation with a self-adjoint unitary in the multiplier algebra of $A$. If $A$ is furthermore equivariant and/or Real, the unitary needs to satisfy the required invariance conditions. The algebras of $G$-continuous bounded and compact operators on a graded $G$-Hilbert space are naturally inner graded. The Clifford algebra $Cl^{p,q}$ is inner graded if $p-q \equiv 0 ~(\operatorname{mod}4)$.
    \end{enumerate}
\end{example}
Before proceeding, we set up some shorthand notation that we will be using throughout this article.
\begin{notn}
    We denote by $G_2$ the group $G \times C_2$ (equipped with the product topology).
\end{notn}
\begin{rem}
    Although the categories $\GTWOALG$ and $\RGraded$ are equivalent (by definition!), the symmetric monoidal structures on them are different. See \cite[Section 3.1]{AS} for a categorical perspective on the symmetric monoidal structures on $\RGraded$.
\end{rem}
We now begin our study of $K$-theory. As is customary in the development of equivariant $K$-theory, we obtain interesting results when the action is by a {compact} group. So \emph{throughout the remainder of this section, $G$ will denote a compact group.}

\subsection{Miscellaneous results on equivariant (ungraded) $K$-theory}
In this subsection, we record some well known results in the realm of equivariant $K$-theory following \cite[Chapter 11]{BLAK}. Recall that (by definition) the equivariant $K_0$-group of a unital $G$-$C^*$-algebra is the Grothendieck group of the abelian monoid (under direct sum) of finitely generated, $G$-equivariant projective $A$-modules. We summarize some well-known properties of the equivariant $K_n$-functors in the next remark for the ease of reference later.
\begin{rem}\label{prop: KOC2}
    The following facts about $G$-equivariant $K$-theory are well-known.
    \begin{enumerate}
        \item  $K_0^G(\C)$ is a ring with multiplication obtained from the external tensor product of $G$-vector spaces, and is isomorphic to $R(G)$, the representation ring of $G$. In particular, the ring $K_0^{C_2}(\C)$ is isomorphic to $\frac{\Z[x]}{x^2=1}$, the group ring of $C_2$. The preferred isomorphism maps the trivial representation of $C_2$ on $\C$ to $1 \in \frac{\Z[x]}{x^2=1}$ and the sign representation of $C_2$ on $\C$ to $x \in \frac{\Z[x]}{x^2=1}$. 
        \item More generally, the external product with an element of $K^G_0(\C)$ refines $K^G_n(-)$ to functors \footnote{this should not be confused with the ``$G$-action on $K^G$-groups induced by the action on the algebra", because in general there is no such action! When $G$ is $C_2$, then this action exists, but can be shown to be always trivial.}
        \[
        K_n^G(-): \RALG \rightarrow R(G)\text{-}\operatorname{Mod} \forall~ n \geq 0
        \] 
        \item Of importance to us is the Green-Julg theorem, which asserts that there is a natural isomorphism
        \[
        K^G_n(A) \cong K_n(A \rtimes G) ~\forall~n \geq 0
        \]
        See \cite[Theorem 11.7.1]{BLAK} for a proof. For the case of $G=C_2$, note that we have a natural isomorphism $A \rtimes C_2 \cong (A \hat\otimes Cl^{0,1})_u$ (see \cite[Proposition 14.5.1]{BLAK}) and will sometimes use this identification without mentioning.
    \end{enumerate}
\end{rem}
    For our applications later, we need to compute the $C_2$-equivariant $K$-groups of the graded $C^*$-algebras $Cl^{0,1}$ and $\s$. Though the calculations of these groups are surely well known, we include in the proofs some details which will come handy in Section 5.
\begin{prop}\label{prop: KOC2 CL01}
    $K_n^{C_2}(Cl^{0,1}) \cong K_n(\C) ~\forall~n \in \N_{\geq 0} $. The $C_2$-action on $K_n^{C_2}(Cl^{0,1})$ induced by $K_0^{C_2}(\C)$ (as dictated by \Cref{prop: KOC2}.2) is given by the identity.
\end{prop}
\begin{proof}
    We begin by analyzing the maps
    \[
    K_n^{C_2}(\C) \xrightarrow{\iota_*} K_n^{C_2}(Cl^{0,1}).
    \]
    Consider the commutative diagram (with the unlabeled horizontal arrow us so defined to render the diagram commutative)
\[\begin{tikzcd}
	{K_n^{C_2}(\mathbb{C})} && {K_n^{C_2}(Cl^{0,1})} \\
	{K_n(Cl^{0,1}_u)} && {K_n(Cl^{0,2}_u) \cong K_n(M_2(\mathbb{C}_u))} \\
	{K_n(\mathbb{C}) \oplus K_n(\mathbb{C})} && {K_n(\mathbb{C})}
	\arrow["{\iota_*}", from=1-1, to=1-3]
	\arrow["\cong"', no head, from=1-1, to=2-1]
	\arrow["\cong", no head, from=1-3, to=2-3]
	\arrow["{(\iota \hat\otimes Cl^{0,1})_u}", from=2-1, to=2-3]
	\arrow["\cong", from=3-1, to=2-1]
	\arrow[from=3-1, to=3-3]
	\arrow["\cong"', from=3-3, to=2-3]
\end{tikzcd}\]
where the top vertical isomorphisms are induced by the Green-Julg theorem, and the lower vertical isomorphisms are induced by the canonical maps. We claim that the bottom horizontal map is given by addition. To that end, one checks that the inclusion $Cl^{0,1} \rightarrow Cl^{0,2}$ corresponds to the homomorphism
\begin{align*}
    \C \oplus \C &\rightarrow M_2(\C)\\
    (x,y) &\mapsto \frac{1}{2}\begin{pmatrix}
        x+y & x-y \\
        x-y & x+y
    \end{pmatrix}.
\end{align*}
under the standard identifications. Precomposing with the two inclusions of $\C$ into $(\C \oplus \C)_u$ therefore induce the homomorphisms $\C \rightarrow M_2(\C)_u$ given by 
\[
 \lambda \mapsto \frac{1}{2}\begin{pmatrix}
        \lambda & \lambda \\ \lambda & \lambda
    \end{pmatrix}  ~~\text{and}~~\mu \mapsto \frac{1}{2}\begin{pmatrix}
        \mu & -\mu \\ -\mu & \mu
    \end{pmatrix}.
\]
The two homomorphisms can be homotoped to the left-upper-corner inclusion and/or the right-lower-corner inclusion. The claim now follows from \cite[Lemma 2.7.6]{WY}.

To understand the $C_2$-action on $K_n^{C_2}(Cl^{0,1})$ induced by $K_0^{C_2}(\C)$, one notes that the flip action on $K_0^{C_2}(\C)$ is invariant under the addition map, and that $\iota_*$ is an $\Z[C_2]$-module homomorphism.
\end{proof}
\begin{cor}\label{prop: K0C2S}
    $K_0^{C_2}(\mathcal{S}) \cong \Z$. The $C_2$-action on $K_0^{C_2}(\s)$ induced by $K_0^{C_2}(\C)$ is given by $n \mapsto -n$.
\end{cor}
\begin{proof}\footnote{We probably are not giving the cleanest of this result! Instead, we mention in the proof some ideas which we will use later. We are also implicitly identifying $\s$ with the bounded transform $\tilde{\s}$ of $\s$ in this proof.}
Consider the diagram in $\RALG$
\begin{equation}\label{eqn: mapping cone for s}
    \begin{tikzcd}
	{\ker(\operatorname{ev}_1)} & {\{f:[0,1] \rightarrow Cl^{0,1} \mid f(0) \in \mathbb{C} \subseteq Cl^{0,1}\}} & {C([0,1],Cl^{0,1})} \\
	{\tilde{\mathcal{S}}} & {\mathbb{C}} & {Cl^{0,1}}
	\arrow[hook, from=1-1, to=1-2]
	\arrow["{\Phi, \cong}"', from=1-1, to=2-1]
	\arrow[hook, from=1-2, to=1-3]
	\arrow["{\operatorname{ev_0}, \simeq}"', two heads, from=1-2, to=2-2]
	\arrow["{\operatorname{ev}_1}"{description}, dashed, two heads, from=1-2, to=2-3]
	\arrow["{\operatorname{ev_0}, \simeq}"', two heads, from=1-3, to=2-3]
	\arrow["{\operatorname{ev_0}}"', from=2-1, to=2-2]
	\arrow["\iota"', from=2-2, to=2-3]
\end{tikzcd}
\end{equation}
where $\Phi$ is as in \Cref{example: basic examples}.3. Note that the right square is a pullback square (it is the defining square of the mapping cylinder construction; see also \Cref{example: mapping cone sequence is Schochet}.2), and the left square as well as the two triangles involving the dashed edge commute up to homotopy. The rest follows by applying the $K_0^{C_2}(-)$ functor to \Cref{eqn: mapping cone for s}, using the homotopy invariance and exicision properties of equivariant $K$-theory and the proof of \Cref{prop: KOC2 CL01}. Observe that the kernel of the addition map $\Z \oplus \Z \rightarrow \Z$ is
    $
    \{(n,-n) \mid n \in \Z\},
    $
    which is evidently isomorphic to $\Z$ with the action $n \mapsto -n$.
\end{proof}
We now want to present a generator of $K_0^{C_2}(\s) \cong \Z$. To that end, we need a concrete model for the $C_2$-equivariant $K_0$-group of a non-unital $C^*$-algebra.  This picture is going to be convenient for later use, especially in Section 3.6.

We begin with the relevant definitions.
\begin{defi}
    Let $A$ be a separable $G$-$C^*$-algebra. Set $QA:= A *A$, the coproduct in the category of separable $G$-$C^*$-algebras. Denote by $qA$ the kernel of the fold map $QA \rightarrow A$. $qA$ gets a natural $G$-action from $A$.
\end{defi}
\begin{rem}\label{rem: qA properties}
    The following remarks about $qA$ are in order.
    \begin{enumerate}
        \item $qA$ fits in a short exact sequence
        \[
        0 \rightarrow qA \rightarrow QA \rightarrow A \rightarrow 0.
        \]
        The sequence splits in two canonical ways, depending on embedding $A$ into one of the components of $QA$.
        \item The $q$ construction is functorial, and there are two natural maps $qA \rightarrow A$ given by restricting the homomorphisms
        \[
        A * A \xrightarrow{(1_A,0)} A, ~~~~ A * A \xrightarrow{(0,1_A)} A
        \]
        to $qA$.
        \item Given $A,B \in \RALG$, a prequasi-homomorphism from $A$ to $B$ consists of a $G$-$C^*$-algebra $E$, two $G$-equivariant $*$-homomorphisms\footnote{$\overline{\alpha}$ should not be confused with any sort of complex conjugation!} $\alpha,\overline{\alpha}:A \rightarrow E$, an ideal $J \triangleleft E$, and a $G$-equivariant $*$-homomorphism $\mu:J\rightarrow B$ such that $\alpha(x)-\overline{\alpha}(x) \in J~\forall~x\in A$. We write it as
    \[
\begin{tikzcd}
	{A } & {E \triangleright J} & B
	\arrow["{\overline{\alpha}}"', from=1-1, to=1-2]
	\arrow["\alpha", shift left=3, from=1-1, to=1-2]
	\arrow["\mu", from=1-2, to=1-3]
\end{tikzcd}
    \]
    By universal properties, a prequasi-homomorphism $A \rightarrow B$ yields a map $\mu \circ q(\alpha,\overline{\alpha}):qA \rightarrow B$, which we simply denote by $\alpha-\overline{\alpha}$.
        \item When $A=\C$, there is a concrete description of $q\C$:
        \[
        q\C \cong \{f\in C_0((0,1],M_2(\C)) \mid f(1) ~\text{is diagonal}\}.
        \]
        This is stated in \cite[Exercise 10.11.13.g]{BLAK}, and a proof with a lot of details and other background is available in \cite{SAARBRUCKEN}. Technically, loc.\! cit.\! gives a concrete function algebra description of the universal unital $C^*$-algebra on two projections, from which our result can be directly deduced in light of the defining short exact sequence of $q\C$ as mentioned in the first remark above. See also \cite[Section 6]{CUNSURVEY}.
        \item The above isomorphism directly shows that, in light of \Cref{example: basic examples}.3, \[\s \rtimes C_2 \cong (\s \hat\otimes Cl^{0,1})_u \cong q\C.\] This bridges the gap between the spectral model (studied in the next subsection) and the Cuntz picture. In particular, the proof of the next proposition can also be deduced from \Cref{thm: spectral model of g eq k theory} in light of a restriction-coinduction adjunction at the level of the compact-stabilized homotopy categories. Similar ideas are used in the proof of \Cref{lem: qC special}.
        \item An important technical feature of $q\mathbb{C}$ is \emph{semi-projectivity}, which implies that for all $A \in \RALG$, the canonical map
        \[
        [q\mathbb{C},A]_G \xrightarrow{\cong} [[q\mathbb{C},A]]_G
        \]
        from the $G$-equivariant homotopy classes of maps to the $G$-equivariant asymptotic homotopy classes of maps from $q\mathbb{C}$ to $A$ is a bijection of sets. See \cite[Exercise 4.7.1, Proposition 25.1.7]{BLAK} and \cite{Blackadar1985Shape} for general results on semi-projective $C^*$-algebras. The proof of $q\mathbb{C}$ being semi-projective is the content of \cite[Theorem 7]{Loring2008Projective}.
    \end{enumerate}
\end{rem}
\begin{prop}\label{prop: Cuntz picture k-theory}
    Let $G$ be a compact group and $A$ a $G$-$C^*$-algebra. Then, there is a natural isomorphism of abelian groups
    \[
    [q\C,A \otimes \mathcal{K}(L^2G_\N)]_G \cong K_0^{G}(A).
    \]
\end{prop}
\begin{proof}
    The result is well known, see for example \cite[Section 17.6]{BLAK}, \cite[Remark 1]{CUNTZ2}. The compact group action (and the Real structure) do not cause any extra difficulty as $G$-equivariant\footnote{$G$-equivariant adjointable operators are precisely the $G$-fixed points of the $G$-$C^*$-algebra of adjointable operators on a $G$-Hilbert module.}, $G$-continuous Fredholm operators on $G$-Hilbert modules classify $G$-equivariant $K_0$ when $G$ is compact; see \Cref{thm: spectral model of g eq k theory}.5 for example.
\end{proof}
\begin{rem}\label{rem: the map in the cuntz picture iso}
     We mention that the isomorphism \begin{align}\label{eqn: cuntz picture dia}
        K_0^G(A) &\xrightarrow{\cong} [q\C,A \otimes \mathcal{K}(L^2G_\N)]_G
    \end{align}
    is given by precomposition with the first map $q\C \rightarrow \C$ as described in \Cref{rem: qA properties}.2, see \cite[Exercise 10.11.13.f]{BLAK}. By taking $n$-fold loops on both sides of \Cref{eqn: cuntz picture dia}, we obtain 
    \[
    K_n^G(A) \cong \pi_n(\underline{\operatorname{Hom}}_G(q\C,A \otimes \mathcal{K}(L^2G_\N)))
    \]
    for all $n \in \N$.
\end{rem}
We are now in a position to supply a concrete generator of $K_0^{C_2}(\s)$. To that end, we consider two distinguised $C_2$-invariant projections in $C([-1,1],M_2)$. Note that $C[-1,1]$ with the even-odd grading is the multiplier algebra of $\s$ \footnote{we implicitly identified $\s$ with its bounded version} and $M_2$ has the diagonal-off diagonal grading, thereby inducing the tensor product action on $C([-1,1],M_2)$.
Set
\begin{align*}
    p_1:[-1,&1] \rightarrow M_2(\C) \\
    &t \mapsto \begin{pmatrix}
        \cos^{2}(\frac{\pi}{2}t) & \sin(\frac{\pi}{2}t)\cos(\frac{\pi}{2}t) \\
        \sin(\frac{\pi}{2}t)\cos(\frac{\pi}{2}t) & \sin^2(\frac{\pi}{2}t)
    \end{pmatrix}
\end{align*}
and 
\begin{align*}
    p_2:[-1,&1] \rightarrow M_2(\C) \\
    &t \mapsto \begin{pmatrix}
        0 & 0\\
        0 & 1
    \end{pmatrix}
\end{align*}
{It is clear that $p_1,p_2$ are $C_2$-invariant projections in $C([-1,1],M_2)$, and that they induce a homomorphism 
\[
p_1-p_2: q\C \rightarrow M_2(\s)
\]}
in the notation of \Cref{rem: qA properties}.3.
\begin{prop}\label{prop: KOC2S to KOC2C commutes}
    The homomorphism $\Z \rightarrow K_0^{C_2}(\s)$ given by $1 \mapsto p_1-p_2$ renders the diagram 
\[\begin{tikzcd}
	{K_0^{C_2}(\mathcal{S})} && {K_0^{C_2}(\mathbb{C})} \\
	\Z && {\Z \oplus \Z}
	\arrow["{{\operatorname{ev}_0}_*}", from=1-1, to=1-3]
	\arrow["\cong", from=1-3, to=2-3]
	\arrow[from=2-1, to=1-1]
	\arrow["{n \mapsto (n,-n)}"', from=2-1, to=2-3]
\end{tikzcd}\]
commutative, where the right isomorphism is the isomorphism described in \Cref{prop: KOC2}.1.
\end{prop}
\begin{proof}
    From the descriptions of $p_1$ and $p_2$ and \Cref{rem: the map in the cuntz picture iso}, it is clear that $\operatorname{ev}_0(p_1-p_2)$ corresponds to the map $q\C \rightarrow M_2$ given by the difference of the $C_2$-equivariant projections $\begin{pmatrix}
        1 & 0 \\
        0 & 0
    \end{pmatrix}$ and $\begin{pmatrix}
        0 & 0\\
        0 & 1
    \end{pmatrix}$ in $M_2$. Now, $M_2$ with the diagonal-off diagonal grading is nothing but the algebra of operators on the $C_2$-vector space $\C^{\operatorname{triv}} \oplus \C^{\operatorname{sign}}$. The rest follows directly.
\end{proof}
\begin{cor}
    The left vertical arrow above is in fact an isomorphism. We declare $p:=p_1-p_2$ to be the preferred generator of $K_0^{C_2}(\s) \cong \Z$. 
\end{cor}
\begin{proof}
    Since we know that $K_0^{C_2}(\s)$ is abstractly isomorphic to $\Z$, this follows from the elementary observation that the map \[
    \Z \xrightarrow{n \mapsto (n,-n)} \Z \oplus \Z
    \]
    can only factor through $\Z$ (in the category of abelian groups) through an isomorphism $\Z \rightarrow \Z$.
\end{proof}
\subsection{The spectral model of graded, $G$-equivariant $K$-theory}
Recall that throughtout this subsection, $G$ is a \emph{compact} group. We want to develop a $G$-equivariant version of the \emph{spectral model of $K$-theory}, originally due to Higson and Guentner in the complex case (\cite[Section 1.5]{HG}), and discussed with details in the Real case in \cite{D}. The spectral model predates any $\infty$-categorical refinements of $K(K)$-theory, and in the present work, we would need to have formal properties of it setup \emph{a priori} in order to get mileage out of the stable $\infty$-categories for graded bivariant $K$-theories constructed in Section 4 (see especially the proof of \Cref{thm: dirac-dual Dirac abstract}).

We begin with the relevant definitions.
\begin{defi}
    Let $A \in \RGraded$.  Define the $G$-equivariant spectral $K$-theory space of $A$, denoted $\mathbb{K}^{G,\operatorname{sp}}(A)$, to be 
    \[
    \mathbb{K}^{G,\operatorname{sp}}(A):=\underline{\operatorname{Hom}}_{G_2}(\s,A \hat\otimes \hatkg),
    \]
    the $G \times C_2$-equivariant $*$-homomorphisms from $\s$ to $A \hat\otimes \hatkg$ equipped with the point-norm topology.
   We define the spectral $K$-theory groups of $A$ to be the homotopy groups of $\mathbb{K}^{G,\operatorname{sp}}(A)$,
    \[
    K_n^{G,\operatorname{sp}}(A):=\pi_n(\mathbb{K}^{G,\operatorname{sp}}(A),\underline{0})
    \]
\end{defi}
Before studying the spectral model, we introduce the so-called \emph{Fredholm model}, which formalizes the notion of $G$-equivariant Fredholm operators classifying $G$-equivariant $K$-theory. It is also a crucial bridge between showing that the spectral model recovers the classical projective module picture of $K$-theory when the algebra is inner graded.
\begin{defi}\label{defi: fred cycle}
    Let $A \in \RGraded$. 
    \begin{itemize}
        \item A $G$-Fredholm cycle over $A$ is a tuple $(\E,F)$, where $\E$ is a graded, countably generated Hilbert $A,G$-module \footnote{see \cite[Section 2.1]{MEYER} for the definition}, and $F$ is a $G$-continuous, $G$-invariant, odd, self-adjoint Fredholm operator on it.
        \item Two $G$-Fredholm cycles $(\E_0,F_0)$ and $(\E_1,F_1)$ over $A$ are said to be \emph{concordant} if there exists a $G$-Fredholm cycle $(\E,F)$ over $A \hat\otimes C[0,1]\cong C([0,1],A)$ such that $(\E \hat\otimes_{\operatorname{ev}_i}A, \operatorname{ev}_i(F)) \cong (\E_i,F_i)$ for $i=0,1$, where $\operatorname{ev}_t:C([0,1],A) \rightarrow A$ is the homomorphism given by $\operatorname{ev}_t(f)=f(t)$.
        \item Denote by $K_0^{G,\operatorname{Fr}}(A)$ the abelian monoid of concordance classes of $G$-Fredholm cycles over $A$.
    \end{itemize}
\end{defi}
\begin{thm}\label{thm: spectral model of g eq k theory}
    Some key features of the spectral model and the Fredholm model of graded, $G$-equivariant $K$-theory are mentioned below. In what follows, $A,B,...$ denote arbitrary graded $G$-$C^*$-algebras, and all the morphisms are in the category $\RGraded$.
    \begin{enumerate}
        \item $\mathbb{K}^{G,\operatorname{sp}}(A)$ is an $H$-group. Consequently, all the equivariant spectral $K$-groups are in fact abelian groups.
        \item A short exact sequence $0 \rightarrow A \rightarrow B \rightarrow C \rightarrow 0$ in $\RGraded$ yields a Serre fibration on the associated spectral $K$-theory spaces.
        \item The functors $K_n^{G,\operatorname{sp}}(-): \RGraded \rightarrow \operatorname{Ab}$ preserve filtered colimits for all $n \in \N_{\geq 0}$.
        \item The functors 
        \[
        K_0^{G,\operatorname{sp}}(-): (\RGraded,\hat\otimes_!) \rightarrow (\operatorname{Ab},\otimes)
        \]
        are lax monoidal.
        \item We have a natural isomorphism of abelian groups \[\pi_0(\mathbb{K}^{G,\operatorname{sp}}(A)) \xrightarrow{\cong} K_0^{G,\operatorname{Fr}}(A).\]
        \item When $A$ is inner graded, we have a natural isomorphism of abelian groups
        \[
        \pi_n(\mathbb{K}^{G,\operatorname{sp}}(A)) \cong K_n^{G}(A) ~\forall~n \geq 0.
        \]
    \end{enumerate}
\end{thm}
\begin{proof}[Proof sketch]
    Assertion 1 follows from the more general statement that given a graded morphism $\varphi : \s \rightarrow A$, the induced morphism
    \[
    \s \rightarrow A \hat\otimes M_2(\C) = M_2(A), ~~~ f \mapsto \begin{pmatrix}
        \varphi(f) & 0\\
        0 & \alpha(\varphi(f))
    \end{pmatrix}
    \]
    is nullhomotopic in $\underline{\operatorname{Hom}}_G(\s,A \hat\otimes M_2(\C))$ (here $\alpha$ is the grading on $A$), where $M_2(\C)$ has the grading given by $\operatorname{Ad}(\tau)$, where $\tau=\begin{pmatrix}
        1 & 0 \\
        0 & -1
    \end{pmatrix}$. See \cite[Remark 1.7]{HG}.\\\\
    Assertion 2 follows from the non-equivariant assertion proven in \cite[Section 3.2]{D} by taking $G$-fixed points. Note that taking $G$-fixed points is an exact functor because $G$ is compact.\\\\
    By similar reductions as done before, to show Assertion 3, it suffices to prove that for a filtered diagram $(A_i)_{i \in I}$ of non-equivariant graded $C^*$-algebras with colimit $A$, one has a weak homotopy equivalence of spaces
    \[
    \operatorname{colim}_{i \in I} \underline{\operatorname{Hom}}(\s,A_i) \rightarrow \underline{\operatorname{Hom}}(\s,A).
    \]
    Upon passing to the Cayley transform picture (see \cite[Lemma 3.1]{D}), the result follows from \cite[Corollary to Theorem 12]{PALAIS}.\\\\
    The key to demonstrating Assertion 4 is the comultiplication map on $\s$ as described in \cite[Section 4.2]{D}. The details are in \cite[Section 6]{D}, and the same proofs go through equivariantly. See also \Cref{cor: s cocommutative coalgebra}.\\\\
    We now discuss Assertion 5. A natural map $\pi_0(\mathbb{K}^{G,\operatorname{sp}}(A)) \xrightarrow{\cong} K_0^{G,\operatorname{Fr}}(A)$ can be written by directly emulating the non-equivariant construction as in \cite[Section 8.5]{D}. We mention some aspects of it for ease of reference later.

    Firstly, we implicitly replace $\s$ with $\tilde{\s}$, and the Fredholm cycles with essentially unitary ones. Recall that \[A \hat\otimes \hatkg \cong \mathcal{K}_A(A \otimes \widehat{L^2G_{\N}}),\footnote{we do not decorate the tensor product of two graded Hilbert modules, since no Koszul signs are involved as there are no multiplicative structures.}\] the $G$-equivariant graded $C^*$-algebra of $A$-compact operators on the standard graded $G$-equivariant Hilbert $A$-module. Given $\varphi: \s \rightarrow A \hat\otimes \hatkg \cong \mathcal{K}_A(A \otimes \widehat{L^2G_{\N}})$, set \[\mathcal{H}_{\varphi}:=\overline{\varphi(\s)(A \otimes \widehat{L^2G_{\N}})},\] and note that $\varphi$ restricts to a non-degenerate homomorphism
    \[
    \varphi^{\operatorname{nd}}: \s \rightarrow \mathcal{K}_A(\mathcal{H}_\varphi).
    \]
    The crucial point here is that $\varphi^{\operatorname{nd}}$ indeed maps to the $A$-compact operators on the (not necessarily complementable) submodule $\mathcal{H}_\varphi \subseteq A \otimes \widehat{L^2G_{\N}}$, as is guaranteed by \cite[Lemma 4.1]{MEYER}. Non-degeneracy of $\varphi^{\operatorname{nd}}$ implies that there exists a unique $G$-equivariant, graded homomorphism $\tilde{\varphi}: C[-1,1] \rightarrow \mathcal{B}(\mathcal{H}_{\varphi})$ extending $\varphi^{\operatorname{nd}}$. The natural map is now given by 
    \begin{align*}
       \pi_0(\mathbb{K}^{G,\operatorname{sp}}(A)) \xrightarrow{}& K_0^{G,\operatorname{Fr}}(A)  \\
       \varphi \longmapsto& (\mathcal{H}_\varphi,\tilde{\varphi}(x))
    \end{align*}
    The proof that this map is bijective now follows with the obvious modifications from \cite[Section 8.5]{D}.\\\\
    Finally, we come to Assertion 6. We just discuss the case of $\pi_0$ when $A$ is unital and ungraded, the general case follows from the techniques described in \cite[Section 2.9]{WY}. One proof strategy in the unital, ungraded case would be to show that the natural map $K_0^G(A) \rightarrow K_0^{G,\operatorname{Fr}}(A)$ has an inverse given by taking the $G$-equivariant index of $G$-invariant Fredholm operators (as done in the non-equivariant case in \cite[Section 8.4]{D}), and then conclude by Assertion 5. However, granted the results of loc.\! cit.\! and the Green-Julg theorem, we can quickly obtain the conclusion by leveraging the fact that $A \rtimes G \cong (A \otimes \mathcal{K}(L^2G))^G$, as shown in \cite[Proposition 11.2]{HGT}. See also the proof of \cite[Proposition 3.1]{DEKM}.
\end{proof}
\begin{notn}
    Having established natural isomorphisms between different models of graded $K$-theory, we henceforth denote the $G$-equivariant graded $K$-groups as $\hat{K}^G_*(-)$.
\end{notn}
Next, we analyse the graded $K_0$-group of $\s$ (ignoring the external group $G$). We summarize our results in the following proposition. Note that all the algebras $C_0(\R^n)$ in the below proof (and in general) carry the trivial grading. The author learnt most of these results from Johannes Ebert.
\begin{prop} \label{prop: omnibus graded k-theory}
\hfill
    \begin{enumerate}
        \item There is a short exact sequence
        \[
        0 \xrightarrow{} C_0(\R) \hat\otimes Cl^{0,1} \xrightarrow{\psi} \s \xrightarrow{\operatorname{ev}_0} \C \rightarrow 0
        \]
        in $\CAlgGr$.
        \item For every $A \in \CAlgGr$, the map
        \[
        \underline{\operatorname{Hom}}(\s,\s \hat\otimes A) \xrightarrow{(\operatorname{ev}_0 \hat\otimes 1_A)_*} \underline{\operatorname{Hom}}(\s,A)
        \]
        have natural continuous sections $\sigma_A$. Furthermore, $\sigma_{\C}(\operatorname{ev}_0)=\operatorname{1}_{\s}$.
        \item For every $A \in \CAlgGr$, we have naturally split exact sequences
        \[
        0\rightarrow \hat{K}_*(C_0(\R) \hat\otimes Cl^{0,1} \hat\otimes A) \xrightarrow{\psi_*} \hat{K}_*(\s \hat\otimes A) \xrightarrow{\operatorname{{ev}_0}_*} \hat{K}_*(A) \rightarrow 0.
        \]
        \item For each $n \in \N$, multiplication by the Bott class $[\beta_\C^n] \in \hat{K}(C_0(\R^n) \hat\otimes Cl^{0,n})$ yields natural isomorphisms 
        \[
        \hat{K}_*(A) \rightarrow \hat{K}_*(C_0(\R^n) \hat\otimes Cl^{0,n} \hat\otimes A).
        \] In particular, $\hat{K}_0(C_0(\R^n) \hat\otimes Cl^{0,n}) \cong \Z$, and the isomorphism identifies $[\beta_{\C}^n]$ and $1$. The grading $\alpha$ on $C_0(\R^n) \hat\otimes Cl^{0,n}$ induces the map $(-1)^n: \Z \rightarrow \Z$. 
        \item $\hat{K}_0(\s) \cong \Z \oplus \Z$ as abelian groups, with the preferred isomorphism mapping $\psi_*[\beta_\C]=:u$ to $(0,1)$ and $[\operatorname{1_\s}]=:e$ to $(1,0)$.
        \item The composition
        \[
        C_0(\R) \hat\otimes Cl^{0,1} \xrightarrow{\psi} \s \xrightarrow{\beta_\C} C_0(\R) \hat\otimes Cl^{0,1} 
        \]
        induces multiplication by $2$ on $\hat{K}_0(C_0(\R) \hat\otimes Cl^{0,1} ) \cong \Z$.
        \item In the context of Assertion 5, we have $\alpha_*(u)=-u$ and $\alpha_*(e)=e-u$, where $\alpha$ denotes the grading on $\s$.
        \item There exists an isomorphism of $C_2$-abelian groups between $\hat{K_0}(\s)$ and $\Z \oplus \Z$ with the flip action, with the preferred isomorphism sending $[1_\s]$ to $(1,0)$ and $[\alpha_\s]$ to $(0,1)$.
    \end{enumerate}
\end{prop}
\begin{proof}
    We begin with Assertion 1. Consider 
    \[
     0 \xrightarrow{} J\xrightarrow{} \s \xrightarrow{\operatorname{ev}_0} \C \rightarrow 0
    \]
    where $J:=\ker(\operatorname{ev}_0)$. Then, we have 
    \begin{align}\label{eqn: J and C0 0 inf}
        J \xrightarrow{\cong} C_0((0,\infty)) \hat\otimes Cl^{0,1} ~~~~~~~~~~~~~f \mapsto f_{\operatorname{ev}}\mid_{(0,\infty)}\hat\otimes 1+f_{\operatorname{odd}}\mid_{(0,\infty)}\hat\otimes \varepsilon,
    \end{align}
    where $f_{\operatorname{ev}}(x)=\frac{f(x)+f(-x)}{2}, f_{{\operatorname{odd}}}(x)=\frac{f(x)-f(-x)}{2}$, and $\varepsilon$ is the canonical generator of $Cl^{0,1}$. Take a homeomorphism $h:(0,\infty) \rightarrow \R$ such that \[\lim_{t \to 0} h(t)=-\infty, \lim_{t \to \infty} h(t)=\infty.\]
    For concreteness, take $h(x)=\ln(x)$. We can define our sought after map $\psi:C_0(\R) \hat\otimes Cl^{0,1} \rightarrow \s$ using $h$ and the isomorphism above.\\\\
    For Assertion 2, note that a direct calculation using the generators $e^{-x^2}$ and $xe^{-x^2}$ of $\s$ shows that the composite
    \[
    \operatorname{Hom}(\s,A) \xrightarrow{1_{\s} \hat\otimes -} \operatorname{Hom}(\s \hat\otimes \s,\s \hat\otimes A) \xrightarrow{\Delta^*} \operatorname{Hom}(\s,\s \hat\otimes A) \xrightarrow{(\operatorname{ev}_0 \hat\otimes 1_A)_*} \operatorname{Hom}(\s,A)
    \]
    is the identity, and has the required effect on $\operatorname{ev}_0 \in \operatorname{Hom}(\s,\C)$ too.\\\\
    Assertion 3 follows immediately from Assertion 2.\\\\
    We just prove the last part of Assertion 4, since the first two parts are the content of \cite[Section 9]{D}. Recall that the Bott class is the class in graded $K$-theory represented by the Bott map, which is the functional calculus of the embedding of $\R^{n}$ in ${Cl^{0,n}}$ as a homomorphism $\s \rightarrow C_0(\R^{n},Cl^{0,n})$.\footnote{strictly speaking, we mean $\R^{0,n}$ here}
    
    We begin with the case $n=1$. To simplify the notation, we denote $\beta_\C^1$ by $\beta_\C$. We have to show that $\alpha_*([\beta_{\C}])=-[\beta_{\C}]$.  By definition of $\beta_{\C}$, we have
    \begin{align*}
    \beta_{\C}: \s &\rightarrow C_0(\R) \hat\otimes Cl^{0,1} \cong C_0(\R,Cl^{0,1})~~~~~
    f \mapsto t \mapsto f(t\varepsilon).
    \end{align*}
    This is same as the map $f \mapsto f^{\operatorname{ev}}+f^{\operatorname{odd}}\varepsilon$, as can be readily checked on the generators $u$ and $v$ of $\s$. Therefore, we have \begin{align*}
        \alpha \circ \beta_\mathbb{C} : \s &\rightarrow C_0(\R,Cl^{0,1}),~~~~~f \mapsto f^{\operatorname{ev}}-f^{\operatorname{odd}}\varepsilon.
    \end{align*}
    We now identify 
    \[
    \operatorname{Hom}(\s,C_0(\R) \hat\otimes Cl^{0,1} ) \cong \operatorname{Map}_*(\R^{+},\operatorname{Hom}(\s,Cl^{0,1})).
    \]
    Furthermore, under the Cayley transform isomorphism, we have a homeomorphism\footnote{in what follows, we denote by $\R^+$ the one-point compactification of $\R$. $\infty \in \R^+$ is the point at infinity.}
    \begin{align} \label{eqn: iso Rplus and hom s cl01}
        \R^+ &\rightarrow \operatorname{Hom}(\s,Cl^{0,1}) \\
        \infty &\mapsto 0 \\
        t &\longmapsto f \mapsto f^{\operatorname{ev}}(t)+f^{\operatorname{odd}}(t)\varepsilon
    \end{align}
    Whence, we have a homeomorphism of spaces
    \[
     \operatorname{Hom}(\s,C_0(\R) \hat\otimes Cl^{0,1} ) \cong \operatorname{Map}_*(\R^+,\R^+),
    \]
    under which $\beta_{\C}$ corresponds to the identity, and $\alpha \circ \beta$ corresponds to the negative of identity. Since degree of the negative identity as a map from $\R^+$ to itself is $-1$, the conclusion follows.
    
    For the general case, we consider the commutative diagram of abelian groups (see \cite[Section 7.1]{D} for the details)
\[
\begin{array}{ccccccc}
\hat{K}_0\!\left(C_0(\R)\hat\otimes Cl^{0,1}\right)
& \otimes &
\hat{K}_0\!\left(C_0(\R)\hat\otimes Cl^{0,1}\right)
& \otimes\cdots\otimes &
\hat{K}_0\!\left(C_0(\R)\hat\otimes Cl^{0,1}\right)
& \longrightarrow &
\hat{K}_0\!\left(C_0(\R^n)\hat\otimes Cl^{0,n}\right)
\\[1ex]
\downarrow\!{\scriptstyle\alpha_*}
&&
\downarrow\!{\scriptstyle\alpha_*}
&&
\downarrow\!{\scriptstyle\alpha_*}
&&
\downarrow\!{\scriptstyle\alpha_*}
\\[1ex]
\hat{K}_0\!\left(C_0(\R)\hat\otimes Cl^{0,1}\right)
& \otimes &
\hat{K}_0\!\left(C_0(\R)\hat\otimes Cl^{0,1}\right)
& \otimes\cdots\otimes &
\hat{K}_0\!\left(C_0(\R)\hat\otimes Cl^{0,1}\right)
& \longrightarrow &
\hat{K}_0\!\left(C_0(\R^n)\hat\otimes Cl^{0,n}\right).
\end{array}
\]
    From \cite[Lemma 9.3.4]{WY} on the one hand we get \[[\beta_\C]\times [\beta_\C]\times... \times [\beta_\C] \mapsto [\beta_\C^n]\in \hat{K}_0(C_0(\R^n) \hat\otimes Cl^{0,n}).\]
    Since $\alpha_*([\beta_\C])=-[\beta_\C]$ for the case dealt with above, from the commutativity of the above diagram, we get on the other hand
    \[
    (-[\beta_\C])\times (-[\beta_\C])\times... \times (-[\beta_\C]) \mapsto \alpha_*[\beta_\C^n]\in \hat{K}_0(C_0(\R^n) \hat\otimes Cl^{0,n}).
    \]
    This proves Assertion 4.
    \\\\
    Assertion 5 follows immediately from  Assertions 3 and 4.\\\\
    We now show Assertion 6. Choose a homeomorphism $h:(0,\infty) \rightarrow \R$ as in the proof of Assertion 1. By definition, the composition $\beta_\C \circ \psi$ is
    \[
    C_0(\R) \hat\otimes Cl^{0,1} \xrightarrow{h^*} C_0((0,\infty)) \hat\otimes Cl^{0,1} \xrightarrow{} \s \xrightarrow{\beta_\C} C_0(\R) \hat\otimes Cl^{0,1},
    \]
    where the unlabelled arrow sends 
    \[
    f \hat\otimes 1 \longmapsto x \mapsto f(|x|), ~~~~~~~~f \hat\otimes \varepsilon \longmapsto x \mapsto s(x)f(|x|),
    \]
    where $s(x): \R \setminus \{0\} \rightarrow \{-1,1\}$ is the sign function\footnote{strictly speaking these are the formulae when $x \neq 0$. At $x=0$ we by default assign the value $0$. We also do this at certain places later without explicitly mentioning.} (this function is nothing but the inverse to the isomorphism in \Cref{eqn: J and C0 0 inf} postcomposed with the inclusion of $J$ into $\s$). The whole composite therefore sends
    \[
    f \hat\otimes 1 \longmapsto f \circ h \circ |\cdot| \hat\otimes 1,~~~~~~~~ f \hat\otimes \varepsilon \longmapsto (s\cdot f\circ h \circ |\cdot|) \hat\otimes \varepsilon.
    \]
    From the description of the Bott map as in the proof of Assertion 4, the induced map
    \begin{align*}
        \operatorname{Hom}(\s,C_0(\R) \hat\otimes Cl^{0,1}) &\xrightarrow{(\beta_{\mathbb{C}}\circ \psi)_*}  \operatorname{Hom}(\s,C_0(\R) \hat\otimes Cl^{0,1})
    \end{align*}
    sends the Bott map to the homomorphism \[
    (f \longmapsto f^{\operatorname{ev}} \circ h \circ |\cdot|)\hat\otimes 1 + (s\cdot f^{\operatorname{odd}}\circ h \circ |\cdot|)\hat\otimes \varepsilon.
    \]
    Upon identifying 
    \[
     \operatorname{Hom}(\s,C_0(\R) \hat\otimes Cl^{0,1} ) \cong \operatorname{Map}_*(\R^{+},\operatorname{Hom}(\s,Cl^{0,1})),
    \]
    the Bott map therefore gets mapped to 
    \[
    x \longmapsto f\mapsto (f^{\operatorname{ev}}(h(|x|))\cdot 1+s(x)f^{\operatorname{odd}}(h(|x|))\cdot\varepsilon).
    \]
    But this is same as the function
    \begin{align*}
        \R^+ &\rightarrow \operatorname{Hom}(\s,Cl^{0,1}) \\
        x &\longmapsto f \mapsto \begin{cases}
            f^{\operatorname{ev}}(h(x))\cdot1+f^{\operatorname{odd}}(h(x))\cdot\varepsilon~~~~~~~~~\text{if}~x>0\\
            f^{\operatorname{ev}}(h(-x))\cdot1-f^{\operatorname{odd}}(h(-x))\cdot\varepsilon ~~~~ \text{if}~x<0 \\
            0 ~~~~~~~~~~~~~~~~~~~~~~~~~~~~~~~~~~~~~~~~~~~~~~~~ \text{if}~ x=0,\infty
        \end{cases}
    \end{align*}
    Under the identification 
    \[
    \R^+ \cong \operatorname{Hom}(\s,Cl^{0,1})
    \]
    from \Cref{eqn: iso Rplus and hom s cl01},  this corresponds to the function
    \begin{align*}
        h':\R^+ &\rightarrow \R^+ \\
        x &\mapsto \begin{cases}
            h(x)~\text{if}~x > 0\\
            -h(-x)~\text{if}~x < 0 \\
            \infty~~~~~~\text{if} ~x=0,\infty
        \end{cases}
    \end{align*}
    The {local degree formula} yields that the degree of $h'$ is $2$. The proof is now completed by the following argument: consider the diagram
\[\begin{tikzcd}
	{\operatorname{Hom}(\mathcal{S},C_0(\mathbb{R},Cl^{0,1}))} && {\operatorname{Hom}(\mathcal{S},C_0(\mathbb{R},Cl^{0,1}))} \\
	{\operatorname{Map}_*(\mathbb{R}^+,\operatorname{Hom}(\mathcal{S},Cl^{0,1}))} && {\operatorname{Map}_*(\mathbb{R}^+,\operatorname{Hom}(\mathcal{S},Cl^{0,1}))} \\
	{\operatorname{Map}_*(\mathbb{R}^+,\mathbb{R}^+)} && {\operatorname{Map}_*(\mathbb{R}^+,\mathbb{R}^+)}
	\arrow["{(\beta_{\mathbb{C}} \circ \psi)_*}", from=1-1, to=1-3]
	\arrow["\cong", no head, from=1-1, to=2-1]
	\arrow["\cong", no head, from=1-3, to=2-3]
	\arrow["\cong", no head, from=2-1, to=3-1]
	\arrow["\cong", no head, from=2-3, to=3-3]
	\arrow[from=3-1, to=3-3],
\end{tikzcd}\]
where the bottom horizontal arrow is defined so as to make the diagram commute. Therefore the bottom horizontal arrow sends 
\[
\operatorname{id}: \R^+ \rightarrow \R^+ 
\]
to
\[
h':\R^+ \rightarrow \R^+,
\]
since we showed above that the right-down itinerary maps $\beta_\C$ to $h'$, and the left vertical composite identifies $\beta_\C$ and $\operatorname{id}$, as was observed in the proof of Assertion 4. Thus, the map induced by $\beta_\C \circ \psi$ on $\hat{K}_0(C_0(\R,Cl^{0,1}))$ is multiplication by the degree of $h'$, which is equal to $2$.
\\\\
For Assertion 7, we consider the commutative diagram (we denote by $\alpha$ the grading on the respective algebras)
\[\begin{tikzcd}
	0 && 0 \\
	{\hat{K}_0(C_0(\R)\hat\otimes Cl^{0,1})} && {\hat{K}_0(C_0(\R)\hat\otimes Cl^{0,1})} \\
	{\hat{K}_0(\mathcal{S})} && {\hat{K}_0(\mathcal{S})} \\
	{\hat{K_0}(\mathbb{C})} && {\hat{K}_0(\mathbb{C})} \\
	0 && 0
	\arrow[from=1-1, to=2-1]
	\arrow[from=1-3, to=2-3]
	\arrow["{\alpha_*=-1.}", from=2-1, to=2-3]
	\arrow["{\psi_*}", from=2-1, to=3-1]
	\arrow["{\psi_*}", from=2-3, to=3-3]
	\arrow["{\alpha_*}", from=3-1, to=3-3]
	\arrow["{{\operatorname{ev}_0}_*}", from=3-1, to=4-1]
	\arrow["{{\operatorname{ev}_0}_*}", from=3-3, to=4-3]
	\arrow["{\alpha_*=1.}", from=4-1, to=4-3]
	\arrow[from=4-1, to=5-1]
	\arrow[from=4-3, to=5-3]
\end{tikzcd}\]
where the the top horizontal map is negation from Assertion 4, and the bottom horizontal map is the identity since $\C$ has the trivial grading. The commutativity of the top rectangle yields that \[\alpha_*(u)=-u.\]
Understanding $\alpha_*(e)$ takes a bit more work. From Assertion 5 and the commutativity of the lower rectangle in the above diagram, we can say that 
\[
\alpha_*(e)=e+nu
\]
for some $n \in \Z$. To compute the value of $n$, we consider the commutative diagram 
\[\begin{tikzcd}
	{\hat{K}_0(\mathcal{S})} && {\hat{K}_0(\mathcal{S})} \\
	{\hat{K}_0(C_0(\mathbb{R}) \hat\otimes Cl^{0,1})} && {\hat{K}_0(C_0(\mathbb{R}) \hat\otimes Cl^{0,1})}
	\arrow["{\alpha_*}", from=1-1, to=1-3]
	\arrow["{\beta_*}"', from=1-1, to=2-1]
	\arrow["{\beta_*}", from=1-3, to=2-3]
	\arrow["{\alpha_*=-1.}"', from=2-1, to=2-3]
\end{tikzcd}\]
From Assertion 4 and Assertion 6, we can say that $\beta_*(e)=[\beta_\C]$ and $\beta_*(u)=2[\beta_\C]$. Tracing the fate of $e$ under the two itineraries from $\hat{K}_0(\s)$ to $\hat{K}_0(C_0(\R) \hat\otimes Cl^{0,1})$, we obtain the equation
\[
[\beta_\C] + 2n\cdot[\beta_\C]=-[\beta_\C],
\]
which shows \[n=-1.\] This finishes the proof of Assertion 7.
\\\\
To prove Assertion 8, we note that Assertion 7 shows (in light of Assertion 5) that $(e,\alpha_*e)$ is also a $\Z$-basis of $\hat{K}_0(\s)$, and the grading $\alpha_\s$ evidently swaps the basis elements.
\end{proof}
\section{Equivariant (ungraded) Real $KK_!$-theory}
\emph{Throughout this section, we work with Real $C^*$-algebras as always. Recall that $! \in \{\operatorname{se,ex}\}$ denotes the type of excision we have in our bivariant $K$-theory, where $\operatorname{se}$ stands for semi-split and $\operatorname{ex}$ stands for exact. In the remainder of this article, we use the superscript ``sep" to denote separable versions of the categories described in \Cref{notn: prelim}.}

In this section, we will first define a small stable $\infty$-category $\KKRsep$ as a series of Dwyer-Kan localizations of $\RALGsep$. Then, we take ind-completions to get a presentable, stable $\infty$-category $\KKR$ together with a universal functor $\RALG \rightarrow \KKR$. We devote the next few subsections to analyzing each step of the localization. Before proceeding, we remind the readers that $\RALGsep$ is an essentially small category.

\subsection{Inverting homotopy equivalences}
In the first step, we perform the Dwyer-Kan localization of $\RALGsep$ at the class of ($G$-equivariant) homotopy equivalences. We begin by recalling the relevant definitions.
\begin{defi}
\hfill
\begin{enumerate}
    \item Two maps $f_0,f_1:A \rightarrow B$ in $\RALGsep$ are said to be homotopic if there exists a morphism $F:A \rightarrow C([0,1],B)$ such that $\operatorname{ev}_i \circ F=f_i$ for $i=0,1$.
    \item  Let $f:A \rightarrow B$ be a morphism in $\RALGsep$. $f$ is said to be a homotopy equivalence if there exists a morphism $g: B \rightarrow A$ in $\RALGsep$ such that the two composites are homotopic to the respective identity maps.
\end{enumerate}
We denote by $\mathcal{W}_h$ the (essentially small) class of homotopy equivalences in $\RALGsep$. 
\end{defi}
\begin{defi}
    Denote by ${\RALGseph}$ the Dwyer-Kan localization of $\RALGsep$ at $\mathcal{W}_h$, and denote by \[
    L_h: \RALGsep \rightarrow \RALGseph
    \]
    the localization functor.
\end{defi}
By definition, this means for any $\infty$-category $\mathcal{D}$, we have an equivalence of functor categories
\begin{equation}\label{eqn: uni prop ralgseph}
    \mathbf{Fun}(\RALGseph,\mathcal{D}) \xrightarrow[\simeq]{L_h^*} \mathbf{Fun}^{h}(\RALGsep,\mathcal{D}),
\end{equation}
where the right-hand side denotes the full subcategory of $\mathbf{Fun}(\RALGsep,\mathcal{D})$ spanned by functors which sends homotopy equivalences in $\RALGsep$ to equivalences in $\mathcal{D}$.

To phrase the relevant properties of the $\infty$-category $\RALGseph$, we need the notion of Schochet fibrations, defined first in \cite{SCHOCHETFIB}. Recall that the category $\RALGsep$ is topological enriched, as mentioned in \Cref{notn: prelim}. Although it is not necessary for us, the basic definitions and properties of Schochet fibrations also work verbatim in the category $\RALG$.
\begin{defi}
\hfill
\begin{enumerate}
    \item A morphism $f: A \rightarrow B$ in $\RALGsep$ is said to be a \emph{Schochet fibration} if for all $D \in \RALGsep$, the induced map of topological spaces
    \[
    \underline{\operatorname{Hom}}(D,A) \xrightarrow{f_*} \underline{\operatorname{Hom}}(D,B)
    \]
    is a Serre fibration.
    \item A cartesian square in $\RALGsep$ 
\[\begin{tikzcd}
	A & B \\
	D & C
	\arrow[from=1-1, to=1-2]
	\arrow[from=1-1, to=2-1]
	\arrow["f", from=1-2, to=2-2]
	\arrow[from=2-1, to=2-2]
\end{tikzcd}\]
is called \emph{Schochet fibrant} if $f$ is a Schochet fibration.
\item For an $\infty$-category $\mathcal{D}$, a functor $\RALGsep \rightarrow \mathcal{D}$ will be called \emph{Schochet exact} if it maps Schochet fibrant cartesian squares to cartesian squares.  We denote by 
\[
\mathbf{Fun}^{Sch}(\RALGsep,\mathcal{D})
\]
the full subcategory of $\mathbf{Fun}(\RALGsep,\mathcal{D})$ spanned by functors which are Schochet exact.
\end{enumerate}
\begin{example}\label{example: mapping cone sequence is Schochet}
We present a few important examples of Schochet fibrations which will be required later on. For proofs see the relevant sections in \cite{SCHOCHETFIB} and \cite{KKandE}.
\begin{enumerate}
    \item For any $C^*$-algebra $A$, the canonical map $C([0,1],A) \xrightarrow{\operatorname{ev}_0} A$ is a Schochet fibration (and a homotopy equivalence). $C([0,1],A)$ is also called the \emph{path algebra} of $A$, and denoted $PA$.
    \item Pullback of a Schochet fibration along any $*$-homomorphism is a Schochet fibration. Given a $*$-homomorphism $f:B \rightarrow
     C$, denote by $Z(f)$ the pullback $B \times_C PC$. $Z(f)$ is called the \emph{mapping cylinder} of $f$, and the canonical map $Z(f) \rightarrow B$ is a Schochet fibration (and a homotopy equivalence).

     Concretely, we have 
     \[
     Z(f)=\{(b,\gamma) \mid b \in B, \gamma \in PC, \gamma(0)=f(b)\}.
     \]
     There is also a canonical map $\operatorname{ev}_1:Z(f) \rightarrow C$ which is also a Schochet fibration. By definition, the \emph{mapping cone} of $f$ fits in the short exact sequence
     \[
     0 \rightarrow C(f) \rightarrow Z(f) \xrightarrow{\operatorname{ev}_1} C \rightarrow 0.
     \]
     This sequence is furthermore semi-split and is called the mapping cone sequence. It will be extremely important for us in Section 3.3.
     
     In the special case that $f$ is an inclusion, $Z(f)$ is isomorphic to all paths in $C$ which start at $B$.
    \item The maximal tensor product with an auxilliary algebra of a Schochet fibrant cartesian square remains a Schochet fibrant cartesian square. The minimal tensor product with an auxilliary algebra of of a mapping cone sequence remains a mapping cone sequence.
\end{enumerate}
\end{example}
\end{defi}
The following proposition discusses properties of $\RALGseph$ and $L_h$.
\begin{prop}\label{prop: RALGseph omnibus}
    \hfill
    \begin{enumerate}
        \item The $\infty$-category $\RALGseph$ is locally small, and is canonically equivalent to the $\infty$-category presented by the Kan-enriched\footnote{we implicitly always view a topologically enriched category as a Kan enriched category} category $\RALGsep$, i.e., by taking its homotopy coherent nerve.
        \item $\RALGseph$ is pointed and $L_h$ is reduced, i.e. maps the zero algebra to the zero object in $\RALGseph$.
        \item $\RALGseph$ has finite limits and countable coproducts. $L_h$ preserves finite products and countable coproducts. 
        \item For every left-exact $\infty$-category $\mathcal{D}$, the equivalence in \Cref{eqn: uni prop ralgseph} restricts to an equivalence 
        \[
        \mathbf{Fun}^{\operatorname{lex}}(\RALGseph,\mathcal{D}) \xrightarrow[\simeq]{L_h^*} \mathbf{Fun}^{h,Sch}(\RALGsep,\mathcal{D}),
        \]
        where the left-hand entity denotes the full subcategory of $\mathbf{Fun}(\RALGseph,\mathcal{D})$ spanned by left-exact functors.
        \item $L_h: \RALGsep \rightarrow \RALGseph$ has a symmetric monoidal refinement for both the minimal and maximal tensor products on $\RALGsep$. Moreover, the functors \[\otimes_?:\RALGseph \times \RALGseph \rightarrow \RALGseph\] are bi-left exact for $? \in \{\operatorname{min,max}\}$.
        \item For every left-exact and symmetric monoidal $\infty$-category $\mathcal{D}$ and $? \in \{\min,\max\}$, the pullback along the symmetric monoidal refinement of $L_h$ induces an equivalence
        \[
        \mathbf{Fun}^{\operatorname{lex}}_{\otimes_?/\operatorname{lax}}(\RALGseph,\mathcal{D}) \xrightarrow[\simeq]{L_h^*} \mathbf{Fun}^{h,Sch}_{\otimes_?/\operatorname{lax}}(\RALGsep,\mathcal{D}),
        \]
        where the subscript $\otimes_? / \operatorname{lax}$ indicates two separate formulas, one for symmetric monoidal functors and the other for lax monoidal functors.
    \end{enumerate}
\end{prop}
\begin{proof}
    The proofs of \cite[Section 3]{KKandE} can be readily adapted to this situation. The Real structure poses no extra difficulty. Note that the key point for the left-exact assertions is that the category of separable $G$-$C^*$-algebras becomes a category of fibrant objects with fibrations as the Schochet fibrations and weak equivalences as the homotopy equivalences.
\end{proof}
\begin{rem}\label{rem: S is loops}
    For later use, we explicitly mention an easy corollary of \Cref{prop: RALGseph omnibus}.4: we have an equivalence of functors
    \[
    L_h \circ (C_0(\R^{0,1},-)) \simeq \Omega \circ L_h : \RALGsep \rightarrow \RALGseph.
    \]
    The details are exactly the same as in \cite[Example 3.19]{KKandE}.

    We also mention at this juncture that all the subsequent localizations leading to $\KKRsep$ will be left-exact Dwyer-Kan localizations. Hence we can take $C_0(\R^{0,1},-)$ (descended to the category under consideration) as a model for the loops functor in each of the subsequent localizations.
\end{rem}
We end this subsection by a casting the idea of the proof of \Cref{prop: K0C2S} using the language of fiber sequences.
\begin{prop}\label{prop: S to C to Cl fib seqn in homotopy cat}
    There exists a fiber sequence in $\RALGsTWOh$ of the form
    \begin{equation}\label{eq: s fib seqn}
        \s \xrightarrow{\operatorname{ev}_0} \mathbb{C} \xrightarrow{\operatorname{\iota}} Cl^{0,1},
    \end{equation}
    where $\iota$ is the canonical inclusion of $\mathbb{C} \cong Cl^{0,0}$ into $Cl^{0,1}$.
\end{prop}
\begin{proof}
This is immediate from \Cref{eqn: mapping cone for s} and \Cref{example: mapping cone sequence is Schochet}.
\end{proof}
\subsection{Enforcing $\mathcal{K}_G$-stability}
At the second step, we want to enforce a equivariant form of stability in $\RALGseph$. The idea is to enforce that the map induced on the $G$-$C^*$-algebras of compact operators by isometric inclusions of $G$-Hilbert spaces become equivalences. We begin with a couple of definitions.
\begin{defi}
\hfill
\begin{itemize}
    \item A morphism $f: A \rightarrow B$ in $\RALGsep$ is said to be a \emph{$\kg$-equivalence} if the induced map
        \[
        f \otimes 1_{\kg}: A \otimes \kg \rightarrow B \otimes \kg
        \]
        is a homotopy equivalence.
    \item For any $\infty$-category $\mathcal{D}$, a functor $F: \RALGsep \rightarrow \mathcal{D}$ is said to be \emph{$G$-stable} if for every isometric embedding of separable $G$-Hilbert spaces $V'\hookrightarrow V$, with $V' \neq 0$, and for all objects $A \in \RALGsep$, the induced map $F(A \otimes \mathcal{K}(V)) \rightarrow F(A \otimes \mathcal{K}(V'))$ is an equivalence.
\end{itemize}
         
\end{defi}
\begin{rem}
    It is easy to see that a homotopy invariant functor on $\RALGsep$ inverts $\kg$-equivalences if and only if is $G$-stable. See \cite[Proposition 2.2]{NAD} for a proof. It so turns out that inverting the $\mathcal{K}_G$-equivalences enables one to control the mapping spaces in the localization, whereas to verify that a homotopy invariant functor indeed inverts $\mathcal{K}_G$-equivalences, one usually just checks $G$-stability. 
\end{rem}
 We denote by $\mathcal{W}_{\kg}$ the (essentially small) class of $\mathcal{K}_G$-equivalences.
\begin{defi}
    Denote by $\SKGRALGseph$ the Dwyer-Kan localization of $\RALGseph$ at $L_h(\mathcal{W}_\kg)$, and denote by
    \[
    L_{\kg}: \RALGseph \rightarrow \SKGRALGseph
    \] the localization functor.
\end{defi}
Note that this means that the canonical functor
\[
L_{h,\kg}:\RALGsep \xrightarrow{L_h} \RALGseph \xrightarrow{L_{\kg}} \SKGRALGseph
\]
is a Dwyer-Kan localization of $\RALGsep$ at $\mathcal{W}_h \cup \mathcal{W}_{\kg}$, whence we have an equivalence of functor categories
\begin{equation}\label{eqn: SKGRALGseph uni prop}
     \mathbf{Fun}(\SKGRALGseph,\mathcal{D}) \xrightarrow[\simeq]{L_{h,\kg}^*} \mathbf{Fun}^{h,\kg}(\RALGsep,\mathcal{D}),
\end{equation}
where the right-hand side denotes the full subcategory of $\mathbf{Fun}(\RALGsep,\mathcal{D})$ spanned by functors which are homotopy invariant and $G$-stable.
\begin{prop}\label{prop: SKGRALKGseph omnibus}
    \hfill
    \begin{enumerate}
        \item The $\infty$-category $\SKGRALGseph$ is canonically equivalent to the full-subcategory of $\RALGseph$ spanned by $\kg$-stable algebras. When $G$ is compact, the functor $L_{\kg}$ is in fact the left-adjoint of a left Bousfield localization.
        \item $\SKGRALGseph$ has finite limits and countable coproducts. $L_{h,\kg}$ preserves finite products and countable coproducts. $L_\kg$ is left-exact.
        \item For every left-exact $\infty$-category $\mathcal{D}$, the equivalence in \Cref{eqn: SKGRALGseph uni prop} restricts to an equivalence
        \begin{equation}\label{eqn: lex skg uni prop}
            \mathbf{Fun}^{\operatorname{lex}}(\SKGRALGseph,\mathcal{D}) \xrightarrow[\simeq]{L_{h,\kg}^*} \mathbf{Fun}^{h,\kg,Sch}(\RALGsep,\mathcal{D}).
        \end{equation}
        \item Countable coproducts in $\SKGRALGseph$ are given by countable direct sums of the respective $C^*$-algebras. The canonical map from a binary (resp. countable) coproduct to a binary (resp. countable) product in $\SKGRALGseph$ is an equivalence, whence $\SKGRALGseph$ is a semi-additive $\infty$-category.
        \item For every left-exact semi-additive $\infty$-category $\mathcal{D}$ with countable coproducts, the equivalence in \Cref{eqn: lex skg uni prop} restricts to an equivalence
        \begin{equation}\label{eqn: skg uni prop for sym mon}
              \mathbf{Fun}^{\operatorname{lex},\coprod_{\omega}}(\SKGRALGseph,\mathcal{D}) \xrightarrow[\simeq]{L_{h,\kg}^*} \mathbf{Fun}^{h,\kg,Sch,\oplus_\omega}(\RALGsep,\mathcal{D}),
        \end{equation}
        where the superscripts $\coprod_\omega, \oplus_\omega$ denote countable coproduct and countable sum preserving functors, respectively.
        \item $L_\kg: \RALGseph \rightarrow \SKGRALGseph$ has a symmetric monoidal refinement for both the tensor products on $\RALGseph$. Consequently the same assertion is also true for $L_{h,\kg}$. Moreover, the functors \[\otimes_?:\SKGRALGseph \times \SKGRALGseph\rightarrow \SKGRALGseph\] are bi-left exact and bi-countable coproduct preserving for $? \in \{\operatorname{min,max}\}$.
        \item For every symmetric monoidal left-exact semi-additive $\infty$-category $\mathcal{D}$ with countable coproducts, the pullback along the symmetric monoidal refinement of $L_{h,\kg}$ induces an equivalence
        \[
        \mathbf{Fun}^{\operatorname{lex},\coprod_\omega}_{\otimes_?/\operatorname{lax}}(\SKGRALGseph,\mathcal{D}) \xrightarrow[\simeq]{L_{h,\kg}^*} \mathbf{Fun}^{h,\kg,Sch,\oplus_\omega}_{\otimes_?/\operatorname{lax}}(\RALGsep,\mathcal{D}).
        \]
    \end{enumerate}
\end{prop}
\begin{proof}
    This follows by combining the approaches in \cite[Section 4]{KKandE}, \cite[Remark 3.34]{EisCASS} and \cite[Lemma 3.18, Corollary 3.19]{DJ}. The Real structure poses no extra difficulty.
\end{proof}
\subsection{Enforcing $!$-excision}
Let
\[\begin{tikzcd}
	0 & A & B & C & 0
	\arrow[from=1-1, to=1-2]
	\arrow[from=1-2, to=1-3]
	\arrow["f"', from=1-3, to=1-4]
	\arrow["s"', shift left=3, curve={height=18pt}, dotted, from=1-4, to=1-3]
	\arrow[from=1-4, to=1-5]
\end{tikzcd}\]
be an extension of algebras in $\RALGsep$, where the cpc-split $s$ exists if we are dealing with the case $!=\operatorname{se}$. From this extension, we form the natural diagram
\begin{equation}\label{eqn: excision diagram}
    \begin{tikzcd}
	& 0 &&& \\
	0 & A & B & C & 0 \\
	0 & {C(f)} & {Z(f)} & C & 0 \\
	& {Q(f)} \\
	& 0
	\arrow[from=1-2, to=2-2]
	\arrow[from=2-1, to=2-2]
	\arrow[from=2-2, to=2-3]
	\arrow["{\iota_f}", from=2-2, to=3-2]
	\arrow["f"', from=2-3, to=2-4]
	\arrow["{h_f}", from=2-3, to=3-3]
	\arrow["s"', shift left=3, curve={height=18pt}, dotted, from=2-4, to=2-3]
	\arrow[from=2-4, to=2-5]
	\arrow[no head, from=2-4, to=3-4]
	\arrow[from=3-1, to=3-2]
	\arrow[from=3-2, to=3-3]
	\arrow["{\pi_f}", from=3-2, to=4-2]
	\arrow["{\tilde{f}}"', from=3-3, to=3-4]
	\arrow["{\tilde{s}}", shift left=3, curve={height=-18pt}, dotted, from=3-4, to=3-3]
	\arrow[from=3-4, to=3-5]
	\arrow["{\hat{s}}", shift left=3, curve={height=-18pt}, dotted, from=4-2, to=3-2]
	\arrow[from=4-2, to=5-2]
\end{tikzcd}
\end{equation}
We now explain the algebras and the morphisms in the above diagram. The middle row is the mapping cone extension of $f$ as described in \Cref{example: mapping cone sequence is Schochet}.2. $h_f$ is the map $b \mapsto (b,\mathrm{const}_{f(b)})$, and $\iota_f$ is the map $a \mapsto (a,0)$ (without loss of generality, we are assuming that $A$ is an ideal in $B$). $\pi_f: C(f) \rightarrow Q(f)$ is the cokernel of $\iota_f$. Note that we have an isomorphism \[Q(f) \cong \{(c,\gamma) \in C \times PC \mid \gamma(0)=c,\gamma(1)=0\} \cong C_0([0,1),C),\]
and therefore $Q(f)$ is contractible.

We now note that if $s$ is a cpc-split of $f$, then it induces natural cpc-splits $\tilde{s}$ and $\hat{s}$ of $\tilde{f}$ and $\pi_f$, given by $\tilde{s}(c)=(s(c),\mathrm{const}_c)$ and $\hat{s}(c,\gamma)=(s(c),\gamma)$.\\\\
We denote by $\hat{W_!}$ the (essentially small) class of morphisms $L_{h,\kg}(\iota_f)$ for all $!$-exact sequences. We then define $W_!$ to be the closure of $\hat{W_!}$ inside $\SKGRALGseph$ under pullbacks, the 2-out-of-3 property and countable sums. Note that ${W_!}$ is also essentially small.
\begin{defi}
    Denote by $\SKGexcl$ the Dwyer-Kan localization of $\SKGRALGseph$ at $W_!$, and denote by
    \[
    L_{!,\oplus_\omega}: \SKGRALGseph \rightarrow \SKGexcl
    \] the localization functor.
\end{defi}
Note that the composite 
\[
    L_{h,\kg,!,\oplus_\omega}:\RALGsep \xrightarrow{L_h} \RALGseph \xrightarrow{L_\kg} \SKGRALGseph\xrightarrow{L_{!,\oplus_\omega}} \SKGexcl
\]
is a Dwyer-Kan localization of $\RALGsep$, as can be seen from the proof of \cite[Proposition 3.40]{EisCASS}. Before describing the key properties of this localization, we mention the motivation for closing $\hat{W}_!$ under countable sums.
\begin{rem}
    It is a non-trivial property of classical $KK$- and $E$-categories for separable $C^*$-algebras that they have countable coproducts, which are furthermore modeled by countable direct sums of the respective $C^*$-algebras. If we do not close up $\hat{W}_!$ under countable sums, then there is no formal reason why our to-be-constructed $\KKRsep$ category will have countable coproducts (see the next paragraph too). The issue therefore arises when we want to extend $\KKRsep$ to the large category $\KKR$ on all of $\RALG$: on the one hand, we can simply consider $\operatorname{Ind}_{\aleph_0}$-completion of $\KKRsep$, but then we would destroy the countable coproducts which already exist in $\KKRsep$ (since \emph{a posteori} the homotopy category of $\KKRsep$ is equivalent to the classical bivariant $K$-categories), and would therefore be an unnatural object to consider. On the other hand, our current approach (as done in Section 3.7) of taking $\operatorname{Ind}_{\aleph_1}(\KKRsep)$ would \emph{a priori} have no reason to have countable coproducts, but would have all uncountable filtered colimits, and would be yet another unnatural object! Of course, if we already know the comparison between the homotopy category of $\KKRsep$ and the classical bivariant $K$-categories, then the $\operatorname{Ind}_{\aleph_1}$-completion of $\KKRsep$ works out fine and becomes a presentable stable $\infty$-category, but we want to postpone this comparison until the very end. Thus, it calls for \emph{enforcing} that our $\KKRsep$ categories have countable coproducts.   

    Having said that, we mention that the step of closing up $\hat{W}_!$ under countable sums is only relevant for non-separable algebras (and even then, more for the {aesthetic} features of our large categories). In particular, if the reader is only interested in separable $C^*$-algebras, then this step can be omitted, and one would still get a stable $\infty$-categorical refinement of the classical bivariant $K$-categories for separable $C^*$-algebras which would a fortiori have countable coproducts.
\end{rem}

\begin{prop}\label{prop: SKGexcl omnibus}
    \hfill
    \begin{enumerate}
        \item The $\infty$-category $\SKGexcl$ is left-exact, semi-additive and posseses countable coproducts.
        \item The localization functor $L_{!,\oplus_\omega}$ is left-exact and countable sum preserving. The functor $L_{h,\kg,!}$ is $!$-exact.
        \item For every left-exact semi-additive $\infty$-category $\mathcal{D}$ admitting countable coproducts, we have equivalences
        \begin{equation}
\label{uni prop excision iwthout symm mon}
\resizebox{0.95\textwidth}{!}{$\displaystyle
\mathbf{Fun}^{\operatorname{lex},\coprod_\omega}
(\SKGexcl,\mathcal{D})
\xrightarrow[\simeq]{L_{!,\oplus}^*}
\mathbf{Fun}^{\operatorname{lex},\coprod_\omega,\hat{W}_!}
(\SKGRALGseph,\mathcal{D})
\xrightarrow[\simeq]{L_{h,\kg}^*}
\mathbf{Fun}^{{h,\kg,!+Sch,\oplus}_\omega}
(\RALGsep,\mathcal{D}),
$}
\end{equation}
where the superscript $\hat{W}_!$  in the middle entity stands for functors on $\SKGRALGseph$ which map morphisms in $\hat{W}_!$ to equivalences, and the superscript $!$ in the last entity stands for functors on $\RALGsep$ which map $!$-exact sequences to fiber sequences.
        \item When $!=\operatorname{se}$, \[L_{\operatorname{se},\oplus_\omega}: \SKGRALGseph \rightarrow \mathrm S_{\kg}\mathsf G\CAlg_{h,\mathrm{se},\oplus_\omega}^{\sep}\] has a symmetric monoidal refinement for both the tensor products on $\SKGRALGseph$. When $!=\operatorname{ex}$, \[L_{\operatorname{ex},\oplus_\omega}: \SKGRALGseph \rightarrow \mathrm S_{\kg}\mathsf G\CAlg_{h,\mathrm{ex},\oplus_\omega}^{\sep}\] has a symmetric monoidal refinement for \emph{only} the maximal tensor product on $\SKGRALGseph$. Consequently, the same assertions are true for $L_{h,\kg,\operatorname{se},\oplus_\omega}$ and $L_{h,\kg,\operatorname{ex},\oplus_\omega}$. 
        
        Moreover, the functors \[\otimes_?:\SKGexcl \times \SKGexcl\rightarrow \SKGexcl\] are bi-left exact and bi-countable coproduct preserving (for every choice of $(?,!)$ except $(\operatorname{min,ex})$).
        \item For every symmetric monoidal left-exact semi-additive $\infty$-category $\mathcal{D}$ admitting countable coproducts, the pullback along the symmetric monoidal refinement of $L_{h,\kg,!,\oplus_\omega}$ (as discussed in the previous assertion) induces an equivalence
        \begin{equation}\label{eqn: uni prop excision}
            \mathbf{Fun}^{\operatorname{lex},\coprod_\omega}_{\otimes_?/\operatorname{lax}}(\SKGexcl,\mathcal{D})  \xrightarrow[\simeq]{L_{h,\kg,!,\oplus_\omega}^*} \mathbf{Fun}^{{h,\kg,!+Sch,\oplus_{\omega}}}_{\otimes_?/\operatorname{lax}}(\RALGsep,\mathcal{D}). 
        \end{equation}   
    \end{enumerate}
\end{prop}
\begin{proof}
    This follows by combining the approaches in \cite[Section 5]{KKandE} and \cite[Remark 3.35]{EisCASS}. Note that closing up $\hat{W}_!$ under countable sums has only been considered in \cite{EisCASS}, and an argument for the localization $L_{!,\oplus_\omega}$ to be countable coproduct preserving is omitted in loc.\! cit.\! This however follows from a dual version of \cite[Theorem 2.6.4]{RSW}\footnote{note that loc.\! cit.\! is written with cofibration categories, whereas our current business is with fibration categories. However, loc.\! cit.\! uses \cite[Proposition 7.7.1]{CIS} in the proof, which is indeed written with fibration categories.}. 
\end{proof}
We end this subsection with an important lemma, which will be useful in formulating the universal property of $\KKRsep$ in the next section. The lemma can be considered as an automatic Schochet-exactness result, which is ubiquitous for applications since Schochet exactness of a functor can be very difficult to verify.
\begin{lem}\label{prop: automatic Schochet exactness}
    Any homotopy invaraint and $!$-exact functor from $\RALGsep$ to a stable $\infty$-category is automatically Schochet exact.
\end{lem}
\begin{proof}
    The proof can be taken directly from \cite[Lemma 5.2]{KKandE}.
\end{proof}
\subsection{Inverting Bott}
This subsection is going to somewhat different than the corresponding step in the complex theory as developed in \cite[Sections 6,7]{KKandE}. Indeed, in the complex case, an analysis of the Toeplitz extensions \[
0 \rightarrow \mathcal{K}(l^2\N) \otimes B \rightarrow \mathcal{T}_0 \otimes B \rightarrow C_0(\R) \otimes B \rightarrow 0
\]\footnote{ That the Toeplitz algebra is nuclear can be deduced from nuclear $C^*$-algebras being closed under extensions, see \cite[Section IV.3.1]{BLAOP} for details.}
yields a boundary map \[C_0(\R^2) \otimes B\xrightarrow{\beta_{!,B}}B\] in (the complex analog of)
$\SKGexcl$, which subsequently are inverted yielding the stable $\infty$-category $\KK^{G,\sep}_{!,\infty,\mathbb{C}}$ (check loc.\! cit.\! for details). This also shows that $\Omega$ is a self-inverse endofunctor on $\KK^{G,\sep}_{!,\infty,\mathbb{C}}$.

In the Real case, one has to work with the Real version of the Toeplitz extension. We discuss it next. Before proceeding, we setup some notation.
\begin{notn}
    Let $\R^{p,q}$ denote the Real space $\R^{p+q}$, with the involution $\kappa_{p,q}$ given by \[\kappa_{p,q}(x_1,x_2,...,x_{p+q})=(-x_1,-x_2,...,-x_p,x_{p+1},x_{p+2},...,x_{p+q}).\]  Denote by $\mathbb{S}^{p,q}$ the unit sphere in $\R^{p,q}$ with the induced involution.
\end{notn}
The following result is going to be our workhorse for the Real Bott periodicity.
\begin{prop}\label{prop: real toep extn}
    There exists a semi-split extension of Real $C^*$-algebras
    \[
    0\rightarrow \mathcal{K}(l^2\N) \rightarrow \mathcal{T}^{R} \rightarrow C(\mathbb{S}^{1,1}),
    \]
    where $\mathcal{T}^R$ is the Real isometry classfier $C^*$-algebra, i.e., there is a distinguised Real isometry $v$ in $\mathcal{T}^{R}$ such that for any unital (non-equivariant) Real $C^*$-algebra $A$, there is a bijection of sets 
    \[
    {\operatorname{Hom}}_u(\mathcal{T}^R,A) \xrightarrow{\cong} \operatorname{Isom}(A),~~~~~ \varphi \mapsto \varphi(v)
    \]
    between the set of unital, Real homomorphisms from $\mathcal{T}^R$ to $A$ and the set of Real isometries in $A$.
\end{prop}
Before proving the above result, we begin with a quick justification that $\mathcal{T}^{R}$ indeed exists, and also set up some notation along the way.
\begin{lem}\label{lem: TR universal}
    The Real $C^*$-algebra $\mathcal{T}^R$ as defined in \Cref{prop: real toep extn} exists.
\end{lem}
\begin{proof}
    Recall that the complex Toeplitz algebra $\mathcal{T^\C}$, given by the $C^*$-subalgebra generated by the unilateral shift (denoted $\boldsymbol{\iota}$ from now on) in $\mathcal{B}(l^2(\N))$, is the universal (complex) isometry classifier (\cite{COB}). It therefore follows that one model for $\mathcal{T}^R$ is $\mathcal{T}^{\C}$ equipped with the unique Real structure such that the unilateral shift is fixed by it. It is also easily seen to be the same as the Real $C^*$-subalgebra of $\mathcal{B}(l^2(\N))$ generated by the unilateral shift, where the Real structure on the Hilbert space $l^2(\N)$ is by complex conjugation (as always).
\end{proof}
We now show that $C(\mathbb{S}^{1,1})$ is also a universal Real $C^*$-algebra.
\begin{lem}
    $C(\mathbb{S}^{1,1})$ is the Real unitary classifier $C^*$-algebra.
\end{lem}
\begin{proof}
    Again as before, the complex version of this result is well-known, see for example \cite[Proposition 2.3.4]{SOBERS}. It follows that the Real structure on the $C^*$-algebra of continuous functions on the circle should be so defined that the identity function on the circle is a Real unitary. Since the Real structure on the codomain $\C$ is complex conjugation, it directly follows from the definitions that the Real structure stemming from the Real space $\mathbb{S}^{1,1}$ is indeed the right one.
\end{proof}
\begin{proof}[Proof (of \Cref{prop: real toep extn})]
    The morphism $\mathcal{T}^R \xrightarrow{} C(\mathbb{S}^{1,1})$ is given by mapping $\boldsymbol{\iota}$ (in the concrete model of $\mathcal{T}^R$ mentioned in \Cref{lem: TR universal}) to the inclusion $\mathbb{S}^{1,1} \hookrightarrow \C$. The kernel of $\varphi$ contains \[\boldsymbol{\iota}^n(1_{l^2\N}-\boldsymbol{\iota}\boldsymbol{\iota}^*){\boldsymbol{\iota}^*}^m~\forall~ n,m \geq 0,\] and therefore also contains the Real $C^*$-subalgebra $\mathcal{K}(l^2\N)$ of $\mathcal{B}(l^2(\N))$. It is well-known that upon forgetting the Real structure, the kernel of $\varphi$ is precisely $\mathcal{K}(l^2\N)$, see for example \cite[Section 2.4]{SOBERS}. This shows the existence of the extension.

    To show that it also admits a Real cpc-split, one directly checks that the cpc-split described in \cite[Remark 6.7]{KKandE} is a Real map.
\end{proof}
\begin{cor}\label{cor: real red toep extn}
    There exists a semi-split extension of Real $C^*$-algebras
    \begin{equation}\label{eqn: red toep eqn}
 0 \rightarrow \mathcal{K}(l^2\N) \rightarrow \mathcal{T}^R_0 \rightarrow C_0(\R^{1,0}) \rightarrow 0.
    \end{equation}
    Given $A \in \RALGsep$ and a homotopy invariant, $\kg$-stable, split exact functor $F:\RALGsep \rightarrow \mathcal{D}$ taking values in the group objects of a semi-additive $\infty$-category $\mathcal{D}$, $F(A \otimes \mathcal{T}^R_0)\simeq 0$.
\end{cor}
\begin{proof}
    The existence of the semi-split extension readily follows from \Cref{prop: real toep extn} and the isomorphism of Real $C^*$-algebras
    \[
    C_0(\R^{1,0}) \cong \ker(C(\mathbb{S}^{1,1})\xrightarrow{\operatorname{ev}_{(0,1)}} \C).
    \]
    The second part of the corollary is proven just as in the complex case, see \cite[Section 6]{KKandE} and \cite[Corollary 1.5.3]{SCHROEDER}.
\end{proof}
Having \Cref{cor: real red toep extn} at our disposal, we can again emulate the treatment of \cite{KKandE} without the Real structure posing any extra difficulty. 

For each $B \in \RALGsep$, \Cref{eqn: red toep eqn} yields a semi-split extension
\[
0 \rightarrow \mathcal{K}(l^2\N) \otimes B \rightarrow \mathcal{T}_0^R \otimes B  \rightarrow C_0(\R^{1,0}) \otimes B \rightarrow 0.
\]
Applying the semi-split exact functor $L_{h,\kg,!,\oplus_\omega}$ to these extensions, one obtains a Puppe sequence
\begin{equation}
\label{eqn: bott extension}
\resizebox{0.98\textwidth}{!}{$\displaystyle
\cdots \rightarrow
L_{h,\kg,!,\oplus_\omega}\!\bigl(C_0(\R^{1,1}) \otimes B\bigr)
\xrightarrow{\beta_{!,B}}
L_{h,\kg,!,\oplus_\omega}\!\bigl(\mathcal{K}(l^2\N) \otimes B\bigr)
\rightarrow
L_{h,\kg,!,\oplus_\omega}\!\bigl(\mathcal{T}_0^R \otimes B\bigr)
\rightarrow
L_{h,\kg,!,\oplus_\omega}\!\bigl(C_0(\R^{1,0}) \otimes B\bigr)
$}
\end{equation}
where $\beta_{!,B}$ are the relevant ``boundary maps" in $\SKGexcl$. Using the stability of $\SKGexcl$, we obtain a natural transformation
\[
\beta_!:=\beta_{!,\mathbb{C}} \otimes -: C_0(\R^{1,1}) \otimes - \implies \operatorname{id}_{\SKGexcl}: \SKGexcl \rightarrow \SKGexcl.
\]
The important feature of this natural transformation is the following.
\begin{cor}\label{cor: loop eq cor}
    If $A$ is a group object in $\SKGexcl$, then the component of $\beta_!$
    \[\beta_{!,A}: C_0(\R^{1,1}) \otimes_! A \rightarrow A\] is an equivalence.
\end{cor}
\begin{proof}
    The proof of \cite[Corollary 6.11]{KKandE} will work with the obvious modifications.
\end{proof}
We are now in a position to define the small stable $\infty$-category $\KKRsep$. We denote by $W_{\operatorname{Bott}}$ the (essentially small) class of maps $\beta_{!,B}$ as in \Cref{eqn: bott extension}. 
\begin{defi}\label{defi: KKRsep}
    Define $\KKRsep$ to be the Dwyer-Kan localization of $\SKGexcl$ at $W_{\operatorname{Bott}}$, and denote by 
    \[
    L_{\operatorname{Bott}}: \SKGexcl \rightarrow \KKRsep
    \]
    the localization functor.
\end{defi}
Note that the composite 
    \begin{equation}\label{eqn: kk functor}
        \kkRsep:\RALGsep \xrightarrow{L_h} \RALGseph \xrightarrow{L_\kg} \SKGRALGseph\xrightarrow{L_{!,\oplus_\omega}} \SKGexcl \xrightarrow{L_{\operatorname{Bott}}} \KKRsep
    \end{equation}
    is a Dwyer-Kan localization of $\RALGsep$, as can be deduced from the proof of \cite[Proposition 3.40]{EisCASS}.
\begin{prop}\label{prop: imp}
    \hfill
    \begin{enumerate}
        \item $\KKRsep$ is left-exact. $L_{\operatorname{Bott}}$ is also left-exact and has a fully faithful left-adjoint. $\KKRsep$ can be realized as the full subcategory of group objects of $\SKGexcl$. 
        \item $\KKRsep$ is a small stable $\infty$-category. 
        \item $\KKRsep$ is countably cocomplete. For a countable collection $(B_n)_{n \in \N}$ of algebras in $\RALGsep$, we have an equivalence
        \[
        \coprod_{n \in \N}\kkRsep(B_n) \xrightarrow{\simeq} \kkRsep(\oplus_{n \in \N} B_n) .
        \]
        When $!=\operatorname{ex}$, $\kkRsep$ even preserves general countable filtered colimits. 
        \item Every left-exact functor from $\SKGexcl$ to a left-exact additive $\infty$-category $\mathcal{D}$ automatically factors (essentially uniquely) through $\KKRsep$. More precisely, for every left-exact $\infty$-category $\mathcal{D}$, we have an equivalence
        \[
        \mathbf{Fun}^{\operatorname{lex}}(\KKRsep,\mathcal{D}) \xrightarrow[\simeq]{L_{\operatorname{Bott}}^*} \mathbf{Fun}^{\operatorname{lex}}(\SKGexcl,\mathcal{D}).
        \]
        \item For every countably cocomplete stable $\infty$-category $\mathcal{D}$, we have an equivalence
        \[
        \mathbf{Fun}^{\operatorname{colim}_{\omega}}(\KKRsep,\mathcal{D}) \xrightarrow[\simeq]{{\kkRsep}^{*}} \mathbf{Fun}^{h,\kg,!,\oplus_\omega}(\RALGsep,\mathcal{D}),
        \]
        where the superscript $\operatorname{colim}_\omega$ denotes the full subcategory of functors which are countable colimit preserving.
        \item $\KKRsep$ has a symmetric monoidal refinement for both the tensor products when $!=\operatorname{se}$, and only for the maximal tensor product when $!=\operatorname{ex}$. 
        \item $L_{\operatorname{Bott}}$ has a symmetric monoidal refinement for the tensor product(s) on $\SKGexcl$. Consequently, the same assertions are true for $\kkRsep$. Moreover, the functors
        \[
        \otimes_?:\KKRsep\times \KKRsep\rightarrow \KKRsep\] are bi-countable colimit preserving (for every choice of $(?,!)$ except $(\operatorname{min,ex})$).
\item For every left-exact symmetric monoidal additive $\infty$-category $\mathcal{D}$, pullback along the symmetric monoidal refinement(s) of $L_{\operatorname{Bott}}$  induces an equivalence
        \[
        \mathbf{Fun}^{\operatorname{lex}}_{\otimes_?/\operatorname{lax}}(\KKRsep,\mathcal{D}) \xrightarrow[\simeq]{L_{\operatorname{Bott}}^*} \mathbf{Fun}^{\operatorname{lex}}_{\operatorname{\otimes_?/\operatorname{lax}}
        }(\SKGexcl,\mathcal{D}).
        \]
                \item For every symmetric monoidal countably cocomplete stable $\infty$-category $\mathcal{D}$, pullback along the symmetric monoidal refinement(s) of $\kkRsep$  induces an equivalence
        \[
        \mathbf{Fun}^{\operatorname{colim}_{\omega}}_{\otimes_?/\operatorname{lax}}(\KKRsep,\mathcal{D}) \xrightarrow[\simeq]{{\kkRsep}^*} \mathbf{Fun}^{h,\kg,!,\oplus_\omega}_{\otimes_?/\operatorname{lax}}(\RALGsep,\mathcal{D}).
        \]
    \end{enumerate}
\end{prop}
\begin{proof}
    We give some details in this proof as the result is ubiquitous for what we do later.\\\\
    Assertion $1$ follows directly from \Cref{eqn: red toep eqn} and the dual version of \cite[Proposition 5.2.7.4]{HTT}, as has been spelled out at the beginning of Section 7 in \cite{KKandE}.\\\\
    We now show Assertion 2. In light of \cite[Proposition 1.4.2.27]{HA}, it suffices to show $\KKRsep$ is small, pointed, left-exact and the loop functor on it is an equivalence. The first three properties are clear in light of Assertion 1 and the work done in the previous three subsections. It remains to see that loops is an equivalence. But this follows from \Cref{cor: loop eq cor} and the last part of Assertion 1.\\\\
    We now move onto Assertion 3. By \Cref{prop: SKGexcl omnibus}.1, $\SKGexcl$ has countable coproducts, whence so does $\KKRsep$ being a Bousfield localization of it (see \cite[Subsection 7.1.4]{kerodon}). Any stable $\infty$-category having countable coproducts is automatically countably cocomplete. That countable coproducts in $\KKRsep$ are modeled by countable direct sum of the algebras follows from \Cref{prop: SKGexcl omnibus}.3 and because of general properties of (co)limits in a Bousfield localization; see again \cite[Subsection 7.1.4]{kerodon}. The case of $!=\operatorname{ex}$ is the content of \cite[Proposition 3.17]{EisCASS} (the Real structure poses no additional difficulty).\\\\
   Assertion 4 follows directly from \Cref{cor: real red toep extn} and \Cref{defi: KKRsep}. \\\\
    Assertion 5 is arguably the most important result from this section for us. It follows from the series of equivalences
    \[
    \mathbf{Fun}^{\operatorname{colim}_{\omega}}(\KKRsep,\mathcal{D}) \xrightarrow[\simeq]{L_{\operatorname{Bott}}^*} \mathbf{Fun}^{\operatorname{lex},\coprod_\omega}(\SKGexcl,\mathcal{D}) \xrightarrow[\simeq]{L_{h,\kg,!,\oplus_\omega}^*}  \mathbf{Fun}^{{h,\kg,!+Sch,\oplus_\omega}}(\RALGsep,\mathcal{D}),
    \]
    where the first equivalence follows from Assertion 4 and (the proof of) Assertion 3, and the second one is \Cref{uni prop excision iwthout symm mon}. We now conclude using \Cref{prop: automatic Schochet exactness}.\\\\
    Assertion 6 follows from \cite[Proposition 3.3.2]{HIN} and the basic permeance properties of $!$-exact sequences with respect to the minimal and maximal tensor products. \\\\
    The first two parts of Assertion 7 follow immediately from the nuclearity of $\mathcal{K}(l^2\N),\mathcal{T}^R, C(\mathbb{S}^{1,1})$ and \cite[Proposition 3.3.2]{HIN}. For the final part, we use the fact that a left-exact and countable coproduct preserving functor between stable $\infty$-categories is automatically countable colimit preserving.\\\\
     For Assertion 8, we build the diagram 
\[\begin{tikzcd}
	{\mathbf{Fun}_{\otimes_?/\operatorname{lax}}^{\operatorname{lex}}(KK_{!,\operatorname{sep}}^{G,\infty},\mathcal{D})} && {\mathbf{Fun}_{\otimes_?/\operatorname{lax}}^{\operatorname{lex}}(S_{\mathcal{K}_G}G\operatorname{\!-\!}C^*\operatorname{\!-alg}_{\operatorname{R,h,!,\oplus}_\omega}^{\operatorname{sep}},\mathcal{D})} \\
	{\mathbf{Fun}_{\otimes_?/\operatorname{lax}}^{\operatorname{}}(KK_{!,\operatorname{sep}}^{G,\infty},\mathcal{D})} && {\mathbf{Fun}^{W_{\operatorname{Bott}}}_{\otimes_?/\operatorname{lax}}(S_{\mathcal{K}_G}G\operatorname{\!-\!}C^*\operatorname{\!-alg}_{\operatorname{R,h,!,\oplus_\omega}}^{\operatorname{sep}},\mathcal{D})}
	\arrow["{L_{\operatorname{Bott}}^*}", dashed, from=1-1, to=1-3]
	\arrow[hook, from=1-1, to=2-1]
	\arrow[hook, from=1-3, to=2-3]
	\arrow["{L_{\operatorname{Bott}}^*,\simeq}", from=2-1, to=2-3]
\end{tikzcd}\]
    The lower horizontal arrow is an equivalence from \Cref{defi: KKRsep} and \cite[Proposition 3.3.2]{HIN}. The dashed arrow can be seen to exist using the fact that $L_{\operatorname{Bott}}$ is left-exact. It remains to show that it is an equivalence. For that, we use the following general fact about symmetric monoidal left-exact Dwyer-Kan localizations: for any symmetric monoidal $\infty$-category $\mathcal{C}$ and an (essentially small) collection of morphisms $W \subseteq \mathcal{C}$ which are a tensor ideal in $\mathcal{C}$, if the Dwyer-Kan localization of $\mathcal{C}$ at $W$ is left-exact, then we have an equivalence
    \[
    \mathbf{Fun}^{\operatorname{lex}}_{\otimes_?/\operatorname{lax}}(\mathcal{C}[W^{-1}],\mathcal{E}) \xrightarrow[\simeq]{L_W^*} \mathbf{Fun}^{\operatorname{lex},W}_{\otimes_?/\operatorname{lax}}(\mathcal{C},\mathcal{E})
    \]
    for every left-exact symmetric monoidal $\infty$-category $\mathcal{E}$. See \cite[Proposition 7.5.6]{CIS} for details.
    \\\\
    The proof of Assertion 9 follows in parallel to the proof of Assertion 5, in light of Assertion 8.
\end{proof}
We end this subsection by mentioning a couple of important corollaries to the construction of $\KKRsep$.
\begin{cor}\label{cor: s fib seqn}
    \hfill
    \begin{enumerate}
        \item In $\KKRsep$, we have an equivalence $qA \xrightarrow{\simeq} A$ for all $A \in \RALGsep$.
        \item The fiber sequence 
        \[
        \s \xrightarrow{\operatorname{ev}_0} \C \xrightarrow{\iota} Cl^{0,1}
        \]
        in $\RALGsTWOh$ as discussed in \Cref{prop: S to C to Cl fib seqn in homotopy cat} remains a fiber sequence in $\KKRGTWOsep$.
        \item The composite
\[
(\s \hat\otimes Cl^{0,1})_u \xrightarrow{{\operatorname{ev}_0}_*} Cl^{0,1}_u \xrightarrow{\pi_1} \C
\]
is an equivalence in $\KKRsep$.
    \end{enumerate}
\end{cor}
\begin{proof}
    Assertion 1 is the content of \cite[Lemma 11.4]{KKandE}, the generalization to the equivariant case is routine. Assertion 2 follows directly from \Cref{prop: S to C to Cl fib seqn in homotopy cat}, since we have shown before that each of the localizations in \Cref{eqn: kk functor} after $L_h$ are left-exact. Assertion 3 follows from Assertion 1 in light of the function algebra model of $q\C$ as described in \Cref{rem: qA properties}.4. 
\end{proof}
\begin{rem}
    We present a more direct proof of \Cref{cor: s fib seqn}.3 in this remark. Consider the ladder of semi-split exact sequences in $\RALGsep$
\[
\footnotesize
\begin{tikzcd}
	0 & {\{f\in C_0((0,1],M_2(\mathbb{C})) \mid f(1) ~\text{dia}\}} & \begin{array}{c} \{f\in C([0,1],M_2(\mathbb{C})) \mid f(0)=\begin{psmallmatrix} * & 0 \\ 0 & 0 \end{psmallmatrix} ~\&~f(1) \text{dia}\} \end{array} & {\mathbb{C}} & 0 \\
	0 & {\mathbb{C}} & {\mathbb{C} \oplus \mathbb{C}} & {\mathbb{C}} & 0
	\arrow[from=1-1, to=1-2]
	\arrow["{\operatorname{inc}}", from=1-2, to=1-3]
	\arrow["{(\operatorname{ev}_1)_{(1,1)}}"', from=1-2, to=2-2]
	\arrow["{(\operatorname{ev}_0)_{(1,1)}}", from=1-3, to=1-4]
	\arrow["{(\operatorname{ev}_0)_{(1,1)} \oplus (\operatorname{ev}_1)_{(1,1)}}", from=1-3, to=2-3]
	\arrow[from=1-4, to=1-5]
	\arrow["{\operatorname{id}}", no head, from=1-4, to=2-4]
	\arrow[from=2-1, to=2-2]
	\arrow["{\iota_2}", from=2-2, to=2-3]
	\arrow["{\pi_1}", from=2-3, to=2-4]
	\arrow[from=2-4, to=2-5]
\end{tikzcd}\]
We have to show that the left vertical arrow is an equivalence in $\KKRsep$. This will follow once the other two vertical arrows are shown to be equivalences in $\KKRsep$. Only the case of the middle vertical arrow needs an argument. The fiber of 
\[
\{f\in C([0,1],M_2(\mathbb{C})) \mid f(0)=\begin{psmallmatrix} * & 0 \\ 0 & 0 \end{psmallmatrix} ~\&~f(1) ~\text{dia}\} \xrightarrow{(\operatorname{ev}_0)_{1,1} \oplus (\operatorname{ev}_1)_{1,1}} \C \oplus \C
\]
is readily seen to be the mapping cone of the right lower corner inclusion $\C \rightarrow M_2(\C)$, and hence is contractible in $\SKGRALGseph$. 
\end{rem}
We end this subsection by mentioning a $1$-categorical universal property of $\KKRsep$. The nature of this universal property is different in flavor than the ones described earlier in this section, and is only needed so as to make the comparison with classical $KK$- and $E$-theory possible later. This technical result can be skipped until the reader reaches Section 6.

We need a definition before proceeding. See also \cite[Section 13]{KKandE} for some definitions circling half exact additive $1$-categories.
\begin{defi}
    Denote by $W_{\KKRsep}$ the collection of morphisms inverted by the functor \[\kkRsep: \RALGsep \rightarrow \KKRsep.\]
\end{defi}
\begin{prop}\label{prop: uni prop hoKK one categorical}{(cf. \cite[Proposition 13.13]{KKandE})}
    Let $\mathcal{D}$ be a half exact additive $1$-category with countable coproducts. Let $F: \RALGsep \rightarrow \mathcal{D}$ be a functor which is homotopy invariant, $\kg$-stable, half exact with respect to $!$-exact sequences, takes countable direct sums in $\RALGsep$ to countable coproducts in $\mathcal{D}$, and \emph{suspension stable} (i.e. for every morphism $f: A \rightarrow B \in \RALGsep$, $F(f)$ is an isomorphism if and only if $F(\Omega(f))$ is an isomorphism). Then $F$ inverts $W_{\KKRsep}$.
\end{prop}
\begin{proof}
    Since $\KKRsep$ by construction is a series of Dwyer-Kan localizations of $\RALGsep$, we need to show that we can construct factorizations as indicated by the dashed arrows below
\begin{equation}\label{eqn: factorizations in one category}\begin{tikzcd}
\RALGsep
&&
\\
\SKGRALGseph
&&
\mathcal{D}
\\
\SKGexcl
\\
\KKRsep
\arrow["{L_{h,\kg}}", from=1-1, to=2-1]
\arrow["F", from=1-1, to=2-3]
\arrow["{F'}"{description}, dashed, from=2-1, to=2-3]
\arrow["{L_{!,\oplus_\omega}}", from=2-1, to=3-1]
\arrow["{F''}"{description}, dashed, from=3-1, to=2-3]
\arrow["{L_{\mathrm{Bott}}}", from=3-1, to=4-1]
\arrow["{F'''}"{description}, dashed, from=4-1, to=2-3]
\end{tikzcd}\end{equation}
    The first two hypothesis on $F$ yields the functor $F'$. Note that $F'$ inverts $\hat{W}_!$, as follows from \cite[Theorem 21.4.4]{BLAK} (which states that $F$ yields natural long exact sequences) and a standard 5-lemma argument applied to \Cref{eqn: excision diagram}.

    We now claim that $F'$ also inverts the closure $W_!$ of $\hat{W}_!$ under 2-out-of-3, pullbacks, and countable sums. It suffices to show that the collection of morphisms inverted by $F'$ is closed under pullbacks and countable sums. The argument for pullbacks can be extracted verbatim from the proof of the analogous assertion in \cite[Proposition 13.13]{KKandE}. That the collection of morphisms inverted by $F'$ are preserved under countable sums follows immediately from our hypothesis that the natural map \[\coprod_{n \in \N} F(A_n) \xrightarrow{}F(\oplus_{n \in \N} A_n)\]
    is an isomorphism. Thus, we deduce the existence of $F''$.

    The argument for $F''$ inverting the maps $\beta_{!,B}$ can be read off from the analogous step in the proof of \cite[Proposition 13.13]{KKandE} (with the obvious changes to the Toeplitz extension in the Real case as discussed in Subsection 3.4). This shows the existence of $F'''$ and the proof is complete.
\end{proof}
\subsection{Change of group functors}
In this subsection, we discuss the several change of group functors connecting the $\KKRsep$ categories we defined earlier. The key features of these functors are that they are defined already at the level of algebras, and fit in useful adjunctions upon passing to the $\KKRsep$ categories. Note that it would also have been possible to develop the (dual) Green-Julg adjunctions in this framework, but we omit it since we do not have to put it to use.

We begin by defining the relevant functors at the level of algebras
\begin{defi}\label{defi: change of groups}
    Let $H,G$ be locally compact, second countable groups.
    \begin{enumerate}
        \item  Let $\varphi: H \rightarrow G$ be a continuous homomorphism\footnote{\emph{not necessarily} inclusion of a subgroup}. Then we have the obvious restriction functor
        \[
        \res^G_H: \RALGsep \rightarrow \RALGsepH.
        \]
        \item When $H \subseteq G$ is closed, we define the induction functor
        \[
        \ind^G_H: \RALGsepH \rightarrow \RALGsep
        \]
        as follows:
        \begin{itemize}
            \item On objects, $\ind^G_H$ sends $B \in \RALGsepH$ to the invariant $C^*$-subalgebra of $C_b(G,B)$ (the group $G$ acts on $C_b(G,B)$ by $(g,f) \longmapsto (g' \mapsto f(g^{-1}g'))$; note that this action is not neccesarily continuous, but it will be so on $\ind^G_H(B)$) generated by the functions $f:G \rightarrow B$ satisfying 
            \[
            f(gh)=h^{-1}f(g)~\forall h \in H,\forall g \in G~~\&~~\text{projection of}~\operatorname{supp}(f)~\text{to}~G/H~\text{is compact}.
            \]
            \item On morphisms, $\ind^G_H$ sends a morphism $\varphi: B \rightarrow B'$ in $\RALGsepH$ to \[\ind^G_H(\varphi):\ind^G_H(B) \rightarrow \ind^G_H(B')\] given by $(\ind^G_H(\varphi)(f))(g):=\varphi(f(g))$.
        \end{itemize}
        \item When $H \subseteq G$ is closed and $G/H$ is compact, then we also call the associated induction functor as the  coinduction functor, and is denoted by $\coind^G_H$.\footnote{the coinduction functor can be defined under more general conditions, but we do not have any reason to do so.}
    \end{enumerate}
\end{defi}
\begin{rem}
    Note that $\res^G$ denotes the underlying non-equivariant $C^*$-algebra functor, and $\res_G$ denotes the functor which equips a (non-equivariant) $C^*$-algebra with the trivial $G$-action (we usually omit the trivial group from the notation). Similar comment also applies to the functors $\res^{G_2}_{C_2}$ and $\res^{C_2}_{G_2}$.
\end{rem}
\begin{prop}\label{prop: functors descend}
    All the functors defined above descend essentially uniquely to the respective equivariant $KK_{!,\infty}^{\operatorname{sep}}$-categories. Furthermore, the descended functors preserve countable colimits.
\end{prop}
\begin{proof}
    This can be extracted from \cite[Section 4]{BEL} and \cite[Section 2.1]{BUNKEcrossed}; no comparison with classical bivariant $K$-theories is made in loc.\! cit.\! to show that the functors descend, and the proof therein is based on universal properties which we have enforced by construction. Note that the last assertion follows from the universal property of the equivariant $KK_{!,\infty}^{\operatorname{sep}}$-category as described in \Cref{prop: imp}.5.
\end{proof}
\begin{rem}
    Note that the restriction functor has a symmetric monoidal refinement for every choice of $(?,!)$ except $(?,!)=(\min,\operatorname{ex})$, see \cite[Corollary 4.8]{BEL}. The same is not true for the induction or coinduction functors, as can be seen by looking at the unit.
\end{rem}
For the next result, we have to provide proofs without banking on classical bivariant $K$-theories. Compare \cite[Example 2.6]{BUNKEcrossed}. Note that in all the adjunctions in this article, the left-adjoint is written on the left and the right-adjoint is written on the right.
\begin{prop}\label{prop: adjunctions at skg level}
    If $H\subseteq G$ is open\footnote{recall that open subgroups of topological groups are automatically closed}, then we have an adjunction
    \[
\operatorname{ind}^G_H:
\mathsf S_{\mathcal K_H}\mathsf H\CAlg_h^{\sep}
\;\leftrightarrows\;
\mathsf S_{\kg}\mathsf G\CAlg_h^{\sep}
:\operatorname{res}^G_H.
\]
\end{prop}
\begin{proof}
    We begin by the description of the unit and the co-unit.

    The unit will be the natural transformation
    $
        \varepsilon: \operatorname{id} \rightarrow \mathsf{res}^G_H \circ \mathsf{ind}^G_H
    $ with \begin{align*}
        \varepsilon_A: A &\rightarrow \mathsf{res}^G_H \mathsf{ind}^G_H(A) \\
        a &\mapsto \chi_H(g)\alpha_{g^{-1}}a = \begin{cases}
            \alpha_{g^{-1}}a&\text{if}~g \in H \\
            0 &\text{otherwise}
        \end{cases}
    \end{align*}\footnote{this is just notation. $A$ is a $H$-$C^*$-algebra, so the action by a general element $g^{-1}$ in $G$ doesn't make sense. However, this will only be considered when $g^{-1} \in H$.}

    The description of the co-unit $\eta:\mathsf{ind}^G_H \circ \mathsf{res}^G_H \rightarrow \operatorname{id}$ is more complicated. It is easy to see that
    \[
    \mathsf{ind}^G_H (\mathsf{res}^G_H (B)) \cong C_0(G/H) \otimes B
    \]
    Upon representing $C_0(G/H)$ in $\mathcal{B}(L^2(G/H))$ as multiplication operators, $C_0(G/H)$ factors through $\mathcal{K}(L^2(G/H))$ since $G/H$ is discrete, courtesy $H$ being open in $G$. We can now define the co-unit maps by 
    \[
    \eta_B:\mathsf{ind}^G_H(\mathsf{res}^G_H(B)) \cong C_0(G/H) \otimes B \rightarrow \mathcal{K}(L^2(G/H)) \otimes B \simeq B.
    \]
    It remains to check the triangle identities. We have that the composite
    \begin{align*}
        \mathsf{res}^G_H&(B) \xrightarrow{\varepsilon_{\mathsf{res}^G_H(B)}} \mathsf{res}^G_H(\mathsf{ind}^G_H(\mathsf{res}^G_H(B))) \xrightarrow{\mathsf{res}^G_H(\eta_B)} \mathsf{res}^G_H(B)
    \end{align*}
    maps
    \begin{align*}
        \mathsf{res}^G_H(B) \ni b &\longmapsto g\mapsto\chi_H(g)\beta_{g^{-1}}(b) \\
        &\longmapsto gH\mapsto\chi_H(g)\beta_g\beta_{g^{-1}}(b) \\
        &= gH \mapsto \chi_H(g)b \\
        &\longmapsto \chi_H \otimes b \in \mathsf{res}^G_H(\mathcal{K}(L^2(G/H))\otimes B) \\
        &\longmapsfrom b,
    \end{align*}
    where the last map is a left-upper corner inclusion. Note that as a $H$-$C^*$-algebra, $\mathcal{K}(L^2(G/H))$ has the trivial action. It follows that \[
    \mathsf{res}^G_H(\eta_B) \circ \varepsilon_{\mathsf{res}^G_H(B)} \simeq \operatorname{id}_{\mathsf{res}^G_H(B)}.
    \]
    The analysis of the other composite is more subtle. We have that the composite
    \[
    \mathsf{ind}^G_H(A) \xrightarrow{\mathsf{ind}^G_H(\varepsilon_A)} \mathsf{ind}^G_H(\mathsf{res}^G_H(\mathsf{ind}^G_H(A))) \xrightarrow{\eta_{\mathsf{ind}^G_H(A)} } \mathsf{ind}^G_H(A)
    \]
    maps
    \begin{align*}
        \mathsf{ind}^G_H(A) \ni (g \mapsto f(g)) &\quad \longmapsto (g \mapsto (l \mapsto \chi_H(l)\alpha_{l^{-1}}f(g))) \in \mathsf{ind}^G_H(\mathsf{res}^G_H(\mathsf{ind}^G_H(A))) \\ 
        &\quad \longmapsto (gH \mapsto (l \mapsto \chi_H(g^{-1}l)\alpha_{(g^{-1}l)^{-1}}f(g)))\in C_0(G/H,\ind^G_H(A)) \\
        &\quad = (gH \mapsto (l \mapsto \chi_H(g^{-1}l)f(l)))\in C_0(G/H,\ind^G_H(A)) \\
        &\quad \xmapsto{(\dagger)} \sum_{kH \in G/H} \chi_{kH} \otimes \chi_{kH}f \in \mathcal{K}(L^2(G/H)) \otimes \mathsf{ind}^G_H(A) ~\text{(see \Cref{lem: technical manipulation} below)}
    \end{align*}
    It remains to analyse the fate of $\sum_{kH \in G/H} \chi_{kH} \otimes \chi_{kH}f$ under the equivalences 
    \[
    \mathcal{K}(L^2(G/H)) \otimes \mathsf{ind}^G_H(A) \xrightarrow{\simeq} \mathcal{K}(\C \oplus L^2(G/H)) \otimes \mathsf{ind}^G_H(A) \xleftarrow{\simeq} \mathsf{ind}^G_H(A).
    \]
    We will be done if we can prove the following
    \begin{clm}
        The two homomorphisms $\mathsf{ind}^G_H(A) \rightarrow \mathcal{K}(\C \oplus L^2(G/H)) \otimes \mathsf{ind}^G_H(A)$ given by \[f \mapsto \sum_{kH \in G/H} \chi_{kH} \otimes \chi_{kH}f\] and \[f \mapsto p \otimes f\] are homotopic, where $p\in \mathcal{K}(\C \oplus L^2(G/H))$ is the projection onto the summand $\C$.
    \end{clm}
    \begin{proof}
        Let $i:\mathcal{K}(L^2(G/H)) \xrightarrow{} \mathcal{K}(\C \oplus L^2(G/H)) $ denote the embedding. Then, $i(\chi_{kH}) \in \mathcal{K}(\C \oplus L^2(G/H))$ is a (not $G$-equivariant) one-dimensional projection. It is elementary to see that there is a rank-one partial isometry $u \in \mathcal{K}(\C \oplus L^2(G/H))$ such that $upu^*=i(\chi_{eH})$. Defining $u_k:=ku$, we also have $u_kpu_k^*=i(\chi_{kH})$. Our sought after homotopy is then given by the following one-parameter family of $G$-equivariant homomorphisms
        \begin{align*}
            \varphi_t:\mathsf{ind}^G_H(A) &\rightarrow \mathcal{K}(L^2(G/H)) \otimes \mathsf{ind}^G_H(A)\\ \quad 
            f &\mapsto \underset{kH \in G/H}{\Sigma} (\cos(\frac{\pi}{2}t)^2i(\chi_{kH}) 
            +\sin(\frac{\pi}{2}t)^2p+\cos(\frac{\pi}{2}t)\sin(\frac{\pi}{2}t)(u_k+u_k^*)) \otimes \chi_{kH}f.
        \end{align*}
    \end{proof}
    This finishes the proof of the proposition.
\end{proof}
\begin{rem} \label{lem: technical manipulation}
    We give here details of the map labelled $(\dagger)$ in the last proof.

    On the one hand, we have 
    \begin{align*}
        gH \longmapsto& (l \mapsto \chi_H(g^{-1}l)f(l))
        =(l \mapsto \chi_{gH}(l)f(l))
    \end{align*}
    On the other hand, under the identification $C_0(G/H) \otimes \ind^G_H(A) \cong C_0(G/H,\ind^G_H(A))$, we have \[\underset{{kH \in G/H}}{\sum} \chi_{kH} \otimes \chi_{kH}f\] corresponds to the map
    \begin{align*}
        gH \longmapsto (l \mapsto \underset{kH \in G/H}{\sum}\chi_{kH}(g).\chi_{kH}(l)f(l)).
    \end{align*}
    On a fixed $gH \in G/H$, both the functions take the value $f(l)$ if $l \in gH$, and $0$ otherwise.
\end{rem}
\begin{thm}\label{thm: adjunctions at kkg}
     Let $H\subseteq G$ be open. Then we have an adjunction
    \[
\ind^G_H:
\KK^{H,\sep}_{!,\infty}
\;\leftrightarrows\;
\KK^{G,\sep}_{!,\infty}
:\res^G_H.
\]
    Let $H \subseteq G$ be closed and $G/H$ be compact. Then we have an adjunction
    \[
\res^G_H:
\KK^{G,\sep}_{!,\infty}
\;\leftrightarrows\;
\KK^{H,\sep}_{!,\infty}
:\coind^G_H.
\]
\end{thm}
\begin{proof}
    The first statement follows directly from \Cref{prop: functors descend} and \Cref{prop: adjunctions at skg level}, since adjunctions descend to a Dwyer-Kan localization if functors do.

    The second statement also follows from \Cref{prop: functors descend} since we have already have an adjunction 
     \[
\res^G_H:
\mathsf G\CAlg^{\sep}
\;\leftrightarrows\;
\mathsf H\CAlg^{\sep}
:\coind^G_H.
\]
    at the level of algebras (see \cite[Lemma 4.2]{BDS}). For our ease of reference later, we mention the unit and counit of this adjunction here:
    \begin{align*}
        \eta_A: &A \rightarrow \mathsf{coind}^G_H\mathsf{res}^G_H(A),~~a \mapsto \eta(a):=(G \ni t \mapsto ta \in A) \\
    \epsilon_B:&\mathsf{res}^G_H\mathsf{coind}^G_H(B) \rightarrow B,~~f \mapsto \epsilon_B(f):=f(1).
    \end{align*}
\end{proof}
\subsection{Connection to $K$-theories}
In this subsection, we show that the mapping spaces in $\KKRsep$ can be understood completely when in the first variable we plug in certain ``special" algebras. Consequently, we recover the classical as well as graded, $G$-equivariant $K$-theory spaces as described in Section 2 as the infinite loop spaces of the mapping spectrum out of $\C$ and $\s$ in $\KKRsep$ respectively; see \Cref{thm: Recovering classical K-theories from our stable inf cat}.

Since we do not yet know that $\KKRsep$ is an $\infty$-categorical refinement of the classical bivariant $K$-categories, the results of these subsection are crucial for nevertheless being able to make connections to $K$-theory, which we use heavily in Section 5.

We now introduce the notion of a $!$-special $C^*$-algebra. \emph{In what follows, $G$ is a compact group.}
\begin{defi}\label{defi: special}
    $A \in \RALGsep$ is called \emph{$!$-special} if it satisfies the following three conditions\footnote{The H in the enumeration below stands for Hypothesis!}.
    \begin{enumerate}[{(H1)}]
        \item For any $!$-exact sequence $0 \rightarrow B \rightarrow C \rightarrow D \rightarrow 0$ in $\RALGsep$, the induced map 
        \[\operatorname{Map}_{\RALGseph}(A,B \otimes \kg) \rightarrow \operatorname{Map}_{\RALGseph}(A,C \otimes \kg) \rightarrow \operatorname{Map}_{\RALGseph}(A,D \otimes \kg) 
        \]
        is a fiber sequence in Anima.
        \item The natural map of abelian monoids
        \[
        \oplus_{n \in \N}[A,B_n \otimes \kg]_G \rightarrow[A,\oplus_{n \in \N} (B_n \otimes \kg)]_G 
        \]
        is an isomorphism for all countable collection of algebras $\{B_n\}_{n \in \N}$ in $\RALGsep$.
        \item The abelian monoid of stable $G$-homotopy classes of maps $[A,B \otimes \kg]_G$ is a group for all $B \in \RALGsep$.
    \end{enumerate}
\end{defi}
\begin{rem}\label{rem: special sufficient}
    Clearly, $\operatorname{ex-special} \implies \operatorname{se-special}$. In the algebras we consider, we will be showing that they are $\operatorname{ex-special}$.
    
    H1 will be satisfied if the induced map of topological spaces
    \[
    \underline{\operatorname{Hom}}_G(A,B) \rightarrow \underline{\operatorname{Hom}}_G(A,C) \rightarrow \underline{\operatorname{Hom}}_G(A,D)
    \]
    is a Serre fibration of topological spaces for every $!$-exact sequence $0 \rightarrow B \rightarrow C \rightarrow D \rightarrow 0$. Indeed, this follows since $- \otimes \kg$ is an exact functor on $\RALGsep$, and because the $\infty$-category $\RALGseph$ is presented by the topologically enriched category $\RALGsep$; see \Cref{prop: RALGseph omnibus}.1.

    Similarly, a sufficient criterion to check H2 is that the functor 
    \[
    [A,-]_G: \RALGsep \rightarrow \operatorname{Set}
    \]
    preserves countable filtered colimits. This is because countable direct sums can be written as a filtered colimit of finite direct sums, and the functor $[A,-]_G$ clearly preserves finite direct sums of algebras in $\RALGsep$.
\end{rem}
The reason we care about $!$-special algebras is the following result.
\begin{prop}\label{prop: special consequences}
    If $A \in \RALGsep$ is $\operatorname{!-special}$, then 
    \begin{enumerate}
        \item The functor
        \[
        \operatorname{map}_{\SKGRALGseph}(A,-): \SKGRALGseph \rightarrow \mathbf{CMon}(\operatorname{An})
        \]
        takes values in $\mathbf{CGrp}(\operatorname{An})$.
        \item The functor
        \[
        F_A:=\operatorname{map}_{\SKGRALGseph}(A,-): \SKGRALGseph \rightarrow \mathbf{CGrp}(\operatorname{An})
        \]
        factors through $\KKRsep$ (necessarily essentially uniquely) as the functor
        \[
        \tau_{\geq 0}\operatorname{map}_{\KKRsep}(A,-): \KKRsep\rightarrow \mathbf{CGrp}(\operatorname{An})
        \]
        \item For each $B \in \RALGsep$, there exists a natural equivalence in $\mathbf{CGrp}(\operatorname{An})$
        \[
        \operatorname{map}_{\SKGRALGseph}(A,B) \xrightarrow{\simeq} \tau_{\geq 0}\operatorname{map}_{\KKRsep}(A,B)
        \]
    \end{enumerate}
\end{prop}
\begin{proof}
    Assertion 1 is an immediate consequence of H3 and the fact that being a commutative group is a \emph{property} of a commutative monoid which is furthermore tested on $\pi_0$.\\\\
    To show Assertion 2, we first show that we have a factorization of $F_A$ as indicated by the dashed arrows below.
\[
\begin{tikzcd}
	{\SKGRALGseph} & {} && {\mathbf{CGrp}(\operatorname{An})} \\
	{\SKGexcl} \\
	{\KKRsep}
	\arrow["{F_A}", from=1-1, to=1-4]
	\arrow["{L_{!,\oplus_\omega}}"', from=1-1, to=2-1]
	\arrow["{F_A'}"{pos=0.3}, dashed, from=2-1, to=1-4]
	\arrow["{L_{\mathrm{Bott}}}"', from=2-1, to=3-1]
	\arrow["{F_A''}"{pos=0.4}, dashed, from=3-1, to=1-4]
\end{tikzcd}
\]
    To see that $F_A'$ exists, we bring \Cref{prop: SKGexcl omnibus}.3 into play. We need to check that $F_A$ is left-exact, inverts $\hat{W}_!$ and preserves countable coproducts. Being a corepresentable functor, $F_A$ is clearly left-exact. That it inverts $\hat{W}_!$ follows directly from H1 in presence of \Cref{eqn: excision diagram}. That $F_A$ also preserves countable coproducts is a consequence of H2, as can be seen from the Whitehead's theorem for spectra together with the fact that filtered colimits of spectra commute with finite limits. 
    To now see that $F_A''$ exists, we simply use \Cref{prop: imp}.4.
    
    That the factorizations $F_A'$ and $F_A''$ are essentially unique follows from the fact that $L_{!,\oplus_\omega}$ and $L_{\operatorname{Bott}}$ are Dwyer-Kan localizations. That $F_A'$ and $F_A''$ are also corepresentable by $A$ follows since the left-Kan extension of a corepresentable functor is corepresentable (\cite[Corollary 8.3.1.9]{kerodon}), and in our case, $F_A'$ and $F_A''$ are also left-Kan extensions of $F_A$ along the localization maps. \\\\
    Assertion 3 also follows from the previous paragraph.
\end{proof}
As an immediate consequence of Assertion 3 of \Cref{prop: special consequences} and the Eckmann-Hilton argument we have the following corollary (cf. \cite[Remark 3.24]{DJ}).
\begin{cor}\label{cor: special mapping spaces in KK}
    For any $!$-special algebra $A$ and any $B \in \RALGsep$, we have an weak homotopy equivalence of spaces
    \[
    \underline{\operatorname{Hom}}_G(A,B \otimes \mathcal{K}_G) \simeq \Omega^{
    \infty
    }\operatorname{map}_{\KKRsep}(A,B).
    \]
    Moreover, under the above equivalence, the block-sum addition on $[A,B \otimes \mathcal{K}_G]_G$ coincides with the additive structure on $\pi_0(\operatorname{map}_{\KKRsep}(A,B))$.
\end{cor}
The next two lemmas show that $\s$ and $q\C$ are ex-special.
\begin{lem}\label{lem: s special}
    $\s \in \RALGsTWO$ is $\operatorname{ex-special}$.  
\end{lem}
\begin{proof}
    The necessary verifications are the content of the first three assertions of \Cref{thm: spectral model of g eq k theory}. Note that we are implicitly using that the $C^*$-algebra of compact operators on a Hilbert space is inner graded.
\end{proof}
\begin{lem}\label{lem: qC special}
    $q\C \in \RALGsep$ is $\operatorname{ex-special}$.
\end{lem}
\begin{proof}
    That $q\C$ satisfies H3 follows from the fact that $[q\C,B \otimes \kg]_G$ is naturally isomorphic to the abelian group $K^G_0(B)$. Since $K^G_0(-)$ is known to preserve filtered colimits (for example, by combining Assertion 3 and Assertion 6 of \Cref{thm: spectral model of g eq k theory}), $q\C$ also satisfies H2.
    
    That $q\C$ satisfies H1 can be shown directly, but we instead deduce it from \Cref{lem: s special}. To that end, we leverage the isomorphism of $C^*$-algebras $q\C \cong (\s \hat\otimes Cl^{0,1})_u$ as mentioned in \Cref{rem: qA properties}.5. The rest follows from the fact that $\s$ satisfies H1 for the group $G_2$ and the restriction-coinduction adjunction. In more detail, for any $B \in \RALGsep$, we have natural equivalences
   \[
\begin{aligned}
\SKGRALGseph(q\C,B)
&\simeq
\SKGRALGseph\!\left(
\mathsf{res}^{G_2}_{G}(\s\hat\otimes Cl^{0,1}),B
\right)
\\
&\simeq
\mathsf S_{\mathcal K_{G_2}}
\mathsf G_2\CAlg_h^{\sep}
\!\left(
\s\hat\otimes Cl^{0,1},
B\hat\otimes Cl^{0,1}
\right)
\\
&\xrightarrow[\simeq]{\hat\otimes Cl^{1,0}}
\mathsf S_{\mathcal K_{G_2}}
\mathsf G_2\CAlg_h^{\sep}(\s,B).
\end{aligned}
\]
    Note that for the second equivalence, we are using the identification\footnote{the more well known isomorphism is with the ungraded tensor product, but since $B$ has the trivial $C_2$-action, we can also use the graded tensor product.} $\mathsf{coind}^{G_2}_G B \cong B \hat\otimes Cl^{0,1}$. For the final equivalence, we are also using that the graded tensor products $\hat\otimes_!$ on $\RALGsTWO$ also descend to a symmetric monoidal structure on $\mathsf S_{\mathcal K_{G_2}}
\mathsf G_2\CAlg_h^{\sep}$, which is easy to see since the algebras of continuous functions on the unit interval and compact operators on a Hilbert space are inner graded.
\end{proof}
\begin{rem}
    We note that $\C$ is \emph{not} $!$-special, as H3 fails for it. We mention without proof that $C_0(\R^{0,1})$ is $!$-special.
\end{rem}
Finally, we obtain the following important connection between classical models of equivariant and graded $K$-theories and the infinite loop spaces of the mapping spectrum in the stable $\infty$-categories we constructed, as mentioned at the beginning of this subsection.
\begin{thm}\label{thm: Recovering classical K-theories from our stable inf cat}
    We have natural equivalences of spaces for any \(B\in\RALGsep\) and \(\hat B\in\RALGsTWO\):
\[
\mathbb K^{G}(B)
\simeq
\Omega^{\infty}\KKRsep(\C,B),
\qquad
\mathbb{\hat K}^{G}(\hat B)
\simeq
\Omega^{\infty}\KKRGTWOsep(\s,\hat B).
\]
    Moreover, $K_0^G(-), \hoKKRsep(\C,-): \RALGsep \rightarrow \operatorname{Ab}$ are equivalent as lax monoidal functors, where $\hoKKRsep(\C,-)$ denotes the homotopy category of $\KKRsep$.
\end{thm}
\begin{proof}
    The first two assertions follow from \Cref{cor: special mapping spaces in KK} and the last two lemmas. For the first assertion we are also using that $q\C$ and $\C$ are equivalent in $\KKRsep$, as follows from \Cref{cor: s fib seqn}.1. The last assertion follows from the lax case of \Cref{prop: imp}.9.
\end{proof}
\subsection{Extension to all $C^*$-algebras}
In this subsection, we construct presentable stable $\infty$-categories $\KKR$ which contain $\KKRsep$ as its full subcategory of $\aleph_1$-compact objects.
\begin{defi}
    Denote by 
    \[
    \KKR:=\operatorname{Ind}_{\aleph_1}(\KKRsep),
    \]
    where $\operatorname{Ind}_{\aleph_1}(\mathcal{C})$ denotes the $\aleph_1$-ind completion of $\mathcal{C}$, i.e. the category obtained by \emph{freely adding}\footnote{see \cite[Proposition 5.3.5.10]{HTT} for the precise universal property} $\aleph_1$-filtered colimits to $\mathcal{C}$. We define functors
    \[
    \kkR: \RALG \rightarrow \KKR
    \]
    as left Kan-extensions
\[
\begin{tikzcd}
	{\RALGsep} && {\KKRsep} & {\KKR} \\
	&& {\RALG}
	\arrow["{\kkRsep}", from=1-1, to=1-3]
	\arrow["{\operatorname{incl.}}"', from=1-1, to=2-3]
	\arrow["y", from=1-3, to=1-4]
	\arrow["{\kkR}"'{pos=0.6}, dashed, from=2-3, to=1-4]
\end{tikzcd}
\]
where $y: \mathcal{C} \rightarrow \operatorname{Ind}_{\aleph_1}(\mathcal{C})$ denotes the canonical embedding (which preserves countable filtered colimits by \cite[Proposition 5.3.5.14]{HTT}).
\end{defi}
Before studying the general properties of the above construction, we need the following definition and an accompanying lemma.
\begin{defi}\label{defi: s finitary}
    A functor $F$ defined on $\RALG$ is called \emph{s-finitary} if for every $A \in \RALG$, the canonical map 
    \[
    \operatorname{colim}_{A' \subseteq_{\operatorname{sep}}A} F(A') \rightarrow F(A)
    \]
    is an equivalence, where $A'$ runs through the separable $G$-invariant $C^*$-subalgebras of $A$.
\end{defi}
\begin{lem}\label{lem: s finitary}
    Let $F:\RALG \rightarrow \mathcal{D}$ be a functor to an $\infty$-category $\mathcal{D}$ having all small filtered colimits. We have a canonical natural transformation $c: \hat{F} \implies F$, where $\hat{F}$ is the left Kan extension of $F\mid_{\RALGsep}$. Then the functor $F$ is s-finitary if and only if $c$ is an equivalence.
\end{lem}
\begin{proof}
    This follows immediately from the pointwise formula for left Kan extension, and the fact that every morphism $A \rightarrow B$ in $\RALG$ with $A$ separable factors through a $G$-invariant separable $C^*$-subalgebra of $B$. See \cite[Lemma 3.1]{BEL} for more details.
\end{proof}
\begin{prop}\label{prop: ind omnibus}
    \hfill
    \begin{enumerate}
        \item The category $\KKR$ is a presentably symmetric monoidal (for every choice of $(?,!)$ except $(\min,\operatorname{ex})$), stable $\infty$-category, and $\KKRsep$ can be identified with the full subcategory of $\KKR$ comprising of $\aleph_1$-compact objects.
        \item The functor $\kkR: \RALG \rightarrow \KKR$ is symmetric monoidal (for every choice of $(?,!)$ except $(\min,\operatorname{ex})$), homotopy invariant, $\kg$-stable, $!$-exact, preserves countable direct sums and is s-finitary\footnote{note that any functor on $\RALG$ which preserves countable direct sums and is s-finitary automatically preserves all direct sums.}.
        \item Pullback along $\kkR$ induces for every cocomplete symmetric monoidal stable $\infty$-category $\mathcal{D}$ with bi-cocontinous tensor product an equivalence
        \[
        (\kkR)^*: \mathbf{Fun}^{\operatorname{colim}}_{\otimes_?\operatorname{/lax}}(\KKR,\mathcal{D}) \xrightarrow{\simeq} \mathbf{Fun}^{h,\kg,\operatorname{!},\oplus_\omega,\operatorname{sfin}}_{\otimes_?\operatorname{/lax}}(\RALG,\mathcal{D}),
        \]
        where the superscript $\operatorname{sfin}$ denotes s-finitary functors.
    \end{enumerate}
\end{prop}
\begin{proof}
       Assertion 1 follows from the general fact that if $\mathcal{C}$ is a small, symmetric monoidal, $\kappa$-cocomplete (for some regular cardinal $\kappa$), stable $\infty$-category such that the tensor product preserves $\kappa$-small colimits in each variable, then $\operatorname{Ind}_{\kappa}(\mathcal{C})$ is a presentably symmetric monoidal stable $\infty$-category and $\mathcal{C}$ can be identified with the full subcategory of $\operatorname{Ind}_{\kappa}(\mathcal{C})$ comprising of $\aleph_1$-compact objects. See \cite[Theorem 5.5.1.1]{HTT}, \cite[Proposition 1.1.3.6]{HA} and \cite[Corollary 4.8.1.14]{HA}.\\\\
       The proof of Assertion 2 can be taken from \cite[Theorem 1.18]{BEL} and \cite[Proposition 3.8]{BEL}\footnote{the proof of lax monoidality of $\kkR$ is formal, but to show its (strong) monoidality, one has to look into the structure of the minimal and maximal tensor products; see loc.\! cit.\! for details}. Note that although \cite{BEL} is written with $\operatorname{Ind}_{\aleph_0}$-completion, all the filtered diagrams considered in Section 3 of loc.\! cit.\! are indeed $\aleph_1$-filtered, and so the proofs go through in our context too. The Real and the (potentially) non-discrete group case do not pose any difficulties.\\\\
       Assertion 3 follows from the series of equivalences
       \[
\begin{aligned}
\mathbf{Fun}^{\operatorname{colim}}_{\otimes_?\operatorname{/lax}}
    (\KKR,\mathcal{D})
&\xrightarrow[\simeq]{y^*}
\mathbf{Fun}^{\operatorname{colim}_{\omega}}_{\otimes_?\operatorname{/lax}}
    (\KKRsep,\mathcal{D})
\\[0.6em]
&\xrightarrow[\simeq]{{\kkRsep}^*}
\mathbf{Fun}^{h,\kg,!,\oplus_\omega}_{\otimes_?\operatorname{/lax}}
    (\RALGsep,\mathcal{D})
\\
&\xleftarrow[\simeq]{\operatorname{incl}^*}
\mathbf{Fun}^{h,\kg,!,\oplus_\omega,\operatorname{sfin}}_
    {\otimes_?\operatorname{/lax}}
    (\RALG,\mathcal{D}).
\end{aligned}
\]
       where the first equivalence follows from \cite[Proposition 5.3.5.10]{HTT} and \cite[Corollary 4.8.1.14]{HA}, the second equivalence is the content of \Cref{prop: imp}.9, and the last equivalence follows from \Cref{lem: s finitary} and Assertion 2. Note that the inverse to the last equivalence is precisely the left Kan-extension along the inclusion $\RALGsep \xrightarrow{\operatorname{incl.}} \RALG$.
\end{proof}
The next result is a generalization of \Cref{thm: Recovering classical K-theories from our stable inf cat} to the non-separable case.
\begin{thm}\label{thm: Recovering classical K-theories from our stable inf cat gen}
     We have natural equivalences of spaces for any \(B\in\RALG\) and \(\hat{B}\in\GTWOALG\):
\[
\mathbb{K}^{G}(B)
\simeq
\Omega^{\infty}\KKR(\C,B),
\qquad
\mathbb{\hat K}^G(\hat{B})
\simeq
\Omega^{\infty}\KKRGTWO(\s,\hat{B}).
\]
    Moreover, $K_0^G(-), \hoKKR(\C,-): \RALG \rightarrow \operatorname{Ab}$ {are equivalent as lax monoidal functors, where $\hoKKR$ denotes the homotopy category of $\KKR$.}
\end{thm}
\begin{proof}
    This follows immediately from \Cref{thm: Recovering classical K-theories from our stable inf cat} upon noting that all the relevant functors preserve $\aleph_1$-filtered colimits.
\end{proof}
\begin{cor}\label{cor: EC2 SS ring}
        $\hoKKCTWO(\s,\s) \cong \frac{\Z[x]}{x^2=1}$ as rings, with the preferred isomorphism mapping $[\operatorname{1}_{\s}]$ to $1$, and $[\alpha_\s]$ to $x$. For any graded $C^*$-algebra $A$, the action of $\alpha_A$ on $\hat{K}_0(A)$ coincides with the action of $\alpha_\s \in \hoKKCTWO(\s,\s)$ on $\hat{K}_0(A)$ stemming from the product in the  $\hoKKCTWO$ category.
    \end{cor}
    \begin{proof}
        This is immediate from Assertion 8 of \Cref{prop: omnibus graded k-theory} and \Cref{thm: Recovering classical K-theories from our stable inf cat gen}.
    \end{proof}
    \begin{rem}\label{rem: C2 action on graded k theory}
        The second perspective in \Cref{cor: EC2 SS ring} directly shows why even the boundary maps in the long exact sequences of graded $K$-theory groups preserve the action of the grading (and this is because of the associativity of composition in the $\hoKKCTWO$ category), and hence the associated long exact sequence can be viewed as a long exact sequence of $C_2$-abelian groups.
    \end{rem}
The final result of this section discusses the change of group functors on the categories of all $C^*$-algebras.
\begin{thm}\label{thm: change of group general}
    There are obvious analogues of the change of group functors of \Cref{defi: change of groups} to the categories of all $C^*$-algebras. These functors then descend to cocontinuous functors between the respective equivariant $KK_{!,\infty}$-categories. The restriction functors have symmetric monoidal refinement for every choice of $(?,!)$ except for $(\min,\operatorname{ex})$. If \(H\subseteq G\) is open, then we have an adjunction
\[
\operatorname{ind}^G_H:
\KK^{H}_{!,\infty}
\leftrightarrows
\KK^{G}_{!,\infty}
:
\res^G_H.
\]

If \(H\subseteq G\) is closed and \(G/H\) is compact, then we have an adjunction
\[
\res^G_H:
\KK^{G}_{!,\infty}
\leftrightarrows
\KK^{H}_{!,\infty}
:
\coind^G_H.
\]
\end{thm}
\begin{proof}
    In light of the results of Section 3.5 and \Cref{prop: ind omnibus}.3, we have to show that the functors $\mathsf{ind}^{G}_H(-),\mathsf{res}^G_H(-), \mathsf{coind}^G_H(-)$ between the relevant categories of all $C^*$-algebras are s-finitary. But this is contained in \cite[Section 4]{BEL}, see also \cite[Section 2.1]{BUNKEcrossed}.
\end{proof}
\section{Graded, Real $KK^G_{!,\infty}$-theory}
In this section, we define the graded $KK^G_{!,\infty}$-categories. We follow some ideas from \cite{BUNKELUDEWIG} and an unpublished note of Bunke-Duenzinger (\cite{BDGRADED}). Note that we make our definitions on the large category $\RGraded$ from the get go. 

It will be ubiquitous for our approach to graded bivariant $K$-theories that $\mathcal{S}$ is a cocommutative coalgebra (also called comonoid) in $(\RGraded,\widehat{\otimes}_?)$. In what follows, we stop decorating the tensor product because $\s$ is nuclear, see \cite[Corollary E.2.19]{WY}. We first come with a diagonal map $\mathcal{S} \rightarrow \mathcal{S} \hat\otimes \mathcal{S}$. The following result is proven in detail in \cite[Proposition 4.3]{D}.
\begin{prop}\label{prop: dia on s}
    There is a unique graded $*$-homomorphism $\s \rightarrow \s \hat\otimes \s$ such that \[u \mapsto u \hat\otimes u,~~~~v \mapsto u \hat\otimes v + v \hat\otimes u,\] where $u:=e^{-x^2}$, $v=xe^{-x^2}$.
\end{prop}
We also have a graded $*$-homomorphism $\eta:\s \rightarrow \C$, given by $\eta(f)=f(0)$. 
\begin{cor}\label{cor: s cocommutative coalgebra}
    The $*$-homomorphisms $\eta$ and $\triangle$ equip $\s$ with a cocommutative coalgebra structure; in the sense that the following diagrams commute
\[\begin{tikzcd}
	{\mathcal{S}} && {\mathcal{S} \hat{\otimes} \mathcal{S}} &&& {\mathcal{S}\hat\otimes \mathcal{S}} &&& {\mathcal{S}} & \\
	&&&& {\mathcal{S}} &&& {\mathcal{S}} && {\mathcal{S}} \\
	{\mathcal{S} \hat{\otimes} \mathcal{S}} && {\mathcal{S} \hat{\otimes} \mathcal{S} \hat{\otimes} \mathcal{S}} &&& {\mathcal{S} \hat\otimes \mathcal{S}} &&& {\mathcal{S} \hat{\otimes} \mathcal{S}}
	\arrow["\Delta", from=1-1, to=1-3]
	\arrow["\Delta"', from=1-1, to=3-1]
	\arrow["{1 \hat{\otimes} \Delta}", from=1-3, to=3-3]
	\arrow["{\operatorname{flip}}", from=1-6, to=3-6]
	\arrow["{=}"', from=1-9, to=2-8]
	\arrow["{=}", from=1-9, to=2-10]
	\arrow["\Delta", from=1-9, to=3-9]
	\arrow["\Delta", from=2-5, to=1-6]
	\arrow["\Delta"', from=2-5, to=3-6]
	\arrow["{\Delta \hat{\otimes} 1}"', from=3-1, to=3-3]
	\arrow["{\eta \hat{\otimes} 1}", from=3-9, to=2-8]
	\arrow["{1 \hat{\otimes} \eta}"', from=3-9, to=2-10]
\end{tikzcd}\]
\end{cor}
\begin{proof}
    This is immediate by considering the elements $u$ and $v$ in $\s$.
\end{proof}
\begin{lem}
    $\KKRGTWO$ has a presentably symmetric monoidal structure for the graded tensor products $\hat{\otimes}_?$ for every choice of $(?,!)$ except $(?,!)=(\min,\operatorname{ex})$, and the canonical functor $\GTWOALG \rightarrow \KKRGTWO$ has a symmetric monoidal refinement too.
\end{lem}
\begin{proof}
    We just discuss the separable analogue of the statement; the generalization to the large categories is a routine argument in light of the results of Section 3.7.
    
    One checks directly that each of the steps in the localizations 
    \[
\RALGsTWO
\xrightarrow{L_h}
\RALGsTWOh
\xrightarrow{L_{\mathcal K_{G_2}}}
\mathrm S_{\mathcal K_{G_2}}\mathsf G_2\CAlg_h^{\sep}
\xrightarrow{L_{!,\oplus}}
\mathrm S_{\mathcal K_{G_2}}\mathsf G_2\CAlg_{h,!,\oplus_\omega}^{\sep}
\xrightarrow{L_{\operatorname{Bott}}}
\KKRGTWOsep.
\]
    has a symmetric monoidal refinement for the graded tensor products, since under the restrictions imposed on $(?,!)$, it is clear that at each step $A \hat\otimes_!-$ preserves the class of morphisms that are inverted at that step (to see this for the excision step, we use that the underlying Banach spaces of the graded and ungraded tensor products coincide, see \cite[Lemma E.2.18]{WY}). The rest of the details are exactly as in the ungraded case.
\end{proof}
\begin{notn}
    Henceforth whenever we mention symmetric monoidal structures on equivariant $KK_{!,\infty}$ categories, we consider all choices of $(?,!)$ except $(?,!)=(\min,\operatorname{ex})$. We also stop decorating the tensor product every time; the decoration is implicitly assumed to be there abiding by the exception mentioned in the previous sentence.
\end{notn}

We now proceed with the task of defining symmetric monoidal bivariant $K$-categories for $\RGraded$. We achieve this by defining $\KKRGRA$ as \emph{a full subcategory of the category of comodules} over $\s$ in $\KKRGTWO$. In general, if $(\cal{C},\otimes)$ is any presentably symmetric monoidal category with a cocommutative coalgebra object $\mathcal{S}$, then we have an adjunction
\[
\operatorname{forget}: \mathbf{Comod}_{\mathcal{C}}(\s) \leftrightarrows \mathcal{C}: \s \otimes -
\]
where $\mathbf{Comod}_{\mathcal{C}}(\s)$ denotes the presentable $\infty$-category of comodules over $\mathcal{S}$, and $\mathcal{S} \otimes -$ denotes the cofree comodule functor. The details of this can be found in \cite[Section 4.2.4]{HA}. \footnote{Note that use of the word (co)free is \emph{not} analogous to the use of the word free as in ``free abelian group on a set", and should not create any confusion!} We mention that the proof of presentability of the category of comodules is subtle (see \cite[Proposition 3.8]{RAM}), and was explained to the author by Maxime Ramzi.

The next definition should be read together with the remarks following it.
\begin{defi}\label{defi: KK-graded}
    The graded bivariant $KK^G_!$-category, denoted $\KKRGRA$, is the full subcategory spanned by cofree comodules in $\mathbf{Comod}_{\KKRGTWO}(\s)$. Note that $\KKRGRA$ becomes a symmetric monoidal category with respect to the balanced tensor product of comodules (\cite[Section 4.4]{HA}). 
    
    Denote by $\kkRGRA$ the composition of the symmetric monoidal functors
    \[
    \kkRGRA: \RGraded \xrightarrow{\kkRGTWO} \KKRGTWO \xrightarrow{\s \hat\otimes -} \KKRGRA\
    \]
\end{defi}
\begin{rem}\label{lem: KKANA cat}
    The following remarks about \Cref{defi: KK-graded} are in order.
    \begin{enumerate}
        \item Explicitly written out, the relative tensor product of $\s \hat\otimes A$ and $\s \hat\otimes B$ is $\s \hat\otimes A \hat\otimes B$, with the codiagonal map $\s \xrightarrow{\Delta} \s \hat\otimes \s$ being used to cancel out one action of $\s$. This also roughly shows why $\KKRGTWO \xrightarrow{\s \hat\otimes -} \KKRGRA$ is symmetric monoidal.
        \item It is not clear if the category $\KKRGRA$ is stable, since cofree comodules need not be closed under the formation of finite limits or colimits. Nevertheless, it is enriched over Spectra being the full subcategory of a stable $\infty$-category, as follows from the next point.
        \item At first glance, a natural solution to the above problem could seem to simply work with the category of all $\s$-comodules. But then, there is no reason for this category to be symmetric monoidal, since one does not have appropriate commutation of the tensor product with cosimplicial limits in general (necessary for the dual of \cite[Theorem 4.5.2.1]{HA} to be applicable). Note that the monoidality assertions in \cite[Definition 2.45]{BUNKELUDEWIG} are incorrect.
        \item In applications, it is more important to have symmetric monoidality of the graded bivariant $K$-categories than having presentability. Therefore we choose the definition above.
        \item \emph{Denote by $\KKANA$ the homotopy category of $\KKRGRA$}. From the forgetful-cofree adjunction, we immediately get the hom groups in $\KKANA$ are given by
    \[
        {\KKANA}(A,B) \cong {\hoKKGTWO}(\s \hat\otimes A,B).
    \]
    The composite of $\varphi \in \KKANA(A,B)\cong\hoKKGTWO(\s \hat\otimes A,B)$ and $\psi \in \KKANA(B,C)\cong\hoKKGTWO(\s \hat\otimes B,C)$ is given by the composite in $\hoKKGTWO$
    \[
    \s \hat\otimes A\xrightarrow{\Delta \hat\otimes 1_A} \s \hat\otimes \s \hat\otimes A \xrightarrow{1_\s \hat\otimes \varphi} \s \hat\otimes B \xrightarrow{\psi} C.
    \]
    The functor \[\kkANA:\RGraded \rightarrow \KKANA\] is identity on objects and precomposes with $\s \hat\otimes A\xrightarrow{{\operatorname{ev}_0}_*} A$ on morphisms emanating out of $A$.
    \end{enumerate}
\end{rem} 
We now mention important by-products of our definitions. These will be crucially used in the next section.
\begin{prop}\label{thm: kkgraeq vs kkctwoeq}
    If $A,B \in \RGraded$ are equivalent in $\KKRGRA$, then $\s \hat\otimes A$ and $\s \hat\otimes B$ are equivalent in $\KKRGTWO$.
\end{prop}
\begin{proof}
    By \Cref{defi: KK-graded}, $A$ and $B$ are equivalent in $\KKRGRA$ means that we have an equivalence \[\kkRGTWO(\s) \hat\otimes \kkRGTWO(A)\simeq \kkRGTWO(\s) \hat\otimes \kkRGTWO(B)\] of cofree comodules. Invoking the functoriality of the forgetful functor $\mathbf{Comod}_{\KKRGTWO}(\s) \rightarrow \KKRGTWO$ now finishes the argument. 
\end{proof}
\begin{thm}\label{prop: equivalence KKgra and KKsp}
    For any \(\hat B\in\RGraded\), we have a natural equivalence of spaces
\[
\mathbb{\hat K}^{G}(\hat B)
\simeq
\Omega^{\infty}\KKRGRA(\C,\hat B).
\]
Moreover,
\(
\hat K_0^G(-),\;
\KKANA(\C,-):
\RGraded
\longrightarrow
\operatorname{Ab}
\)
are equivalent as lax monoidal functors.
\end{thm}
\begin{proof}
    This is an immediate consequence of \Cref{defi: KK-graded}, \Cref{thm: Recovering classical K-theories from our stable inf cat gen}, and \Cref{thm: spectral model of g eq k theory}.4. The lax monoidality assertion follows  upon observing that the lax monoidality of $\KKRGTWO(\s,-)$ stems from the codiagonal map $\s \xrightarrow{\Delta} \s \hat\otimes \s
    $ too, as mentioned in \Cref{lem: KKANA cat}.1.
\end{proof}
We now show that graded $\KKRGRA$-theories recover ungraded $\KKR$-theories when the algebras are trivially graded. More precisely,
\begin{thm}\label{thm: higson email to johannes}
    In the commutative square 
\[
\small
\begin{tikzcd}
	{\RALG} &&& {\KKR} \\
	{\RGraded} && {\KKRGTWO} & {\KKRGRA}
	\arrow["{\kkR}", from=1-1, to=1-4]
	\arrow["{\mathsf{res}^G_{G_2}(-)}"', from=1-1, to=2-1]
	\arrow["{\mathsf{res}^G_{G_2}(-)}"{description}, from=1-4, to=2-3]
	\arrow["{\mathcal{S}\hat\otimes\mathsf{res}^G_{G_2}(-)}", from=1-4, to=2-4]
	\arrow["{\kkRGTWO}"', from=2-1, to=2-3]
	\arrow["{\mathcal{S}\hat\otimes -}"', from=2-3, to=2-4]
\end{tikzcd}
\]
the right functor is fully faithful.
\end{thm}
\begin{proof}
    We have to show that the composite
    \begin{equation}\label{eqn: triv gra is ungra}
        \KKR(A,B) \xrightarrow{\mathsf{res}^G_{G_2}(-)} \KKRGTWO(\mathsf{res}^G_{G_2}(A),\mathsf{res}^G_{G_2}(B)) \xrightarrow{\operatorname{ev}_0^*} \KKRGTWO(\s \hat\otimes \mathsf{res}^G_{G_2}(A),\mathsf{res}^G_{G_2}(B))
    \end{equation}
    is an equivalence. To that end, we accommodate \Cref{eqn: triv gra is ungra} into the commutative diagram (note that the left part of the diagram commutes can be seen by explicitly calculating the fate of a $G$-equivariant homomorphism from $A$ to $B$ using the definition of the unit of the restriction-coinduction adjunction as described in \Cref{defi: change of groups}. The formulae of loc.\! cit.\! also work in the non-separable case)
\[
\begin{tikzcd}[
	column sep=1.2em,
	row sep=2.2em,
	scale=0.78,
	transform shape
]
	{\KKR(A,B)}
	&
	{\KKRGTWO\!\left(\mathsf{res}^G_{G_2}(A),\mathsf{res}^G_{G_2}(B)\right)}
	&
	{\KKRGTWO\!\left(
		\mathcal{S}\hat\otimes\mathsf{res}^G_{G_2}(A),
		\mathsf{res}^G_{G_2}(B)
	\right)}
	\\
	{}
	&
	{\KKRGTWO\!\left(
		\mathsf{res}^G_{G_2}(A)\hat\otimes Cl^{0,1},
		\mathsf{res}^G_{G_2}(B)\hat\otimes Cl^{0,1}
	\right)}
	&
	{\KKRGTWO\!\left(
		\mathcal{S}\hat\otimes\mathsf{res}^G_{G_2}(A)\hat\otimes Cl^{0,1},
		\mathsf{res}^G_{G_2}(B)\hat\otimes Cl^{0,1}
	\right)}
	\\
	{}
	&
	{\KKR\!\left(
		\mathsf{res}^{G_2}_G(A\otimes Cl^{0,1}),B
	\right)}
	&
	{\KKR\!\left(
		\mathsf{res}^{G_2}_G(\mathcal{S}\hat\otimes Cl^{0,1})\otimes A,B
	\right)}
	\arrow["{\mathsf{res}^G_{G_2}(-)}", from=1-1, to=1-2]
	\arrow[
		"{(Cl^{0,1}_u\xrightarrow{\pi_1}\C)^*}"'{pos=0.4},
		bend right=24,
		from=1-1,
		to=3-2
	]
	\arrow["{\operatorname{ev}_0^*}", from=1-2, to=1-3]
	\arrow["{-\hat\otimes Cl^{0,1}}"', from=1-2, to=2-2]
	\arrow["{-\hat\otimes Cl^{0,1}}"', from=1-3, to=2-3]
	\arrow["{\operatorname{ev}_0^*}", from=2-2, to=2-3]
	\arrow[
		"{\mathsf{res}^{G_2}_G\dashv\mathsf{coind}^{G_2}_G}"',
		from=2-2,
		to=3-2
	]
	\arrow[
		"{\mathsf{res}^{G_2}_G\dashv\mathsf{coind}^{G_2}_G}"',
		from=2-3,
		to=3-3
	]
	\arrow["{\operatorname{ev}_0^*}", from=3-2, to=3-3]
\end{tikzcd}
\]
Note that we implicitly have used that $A$ and $B$ have the trivial grading so as to replace some graded tensor products with the ungraded ones.
In light of the above diagram, our thesis would follow once we have shown that the composite
\[
(\s \hat\otimes Cl^{0,1})_u \xrightarrow{{\operatorname{ev}_0}_*} Cl^{0,1}_u \xrightarrow{\pi_1} \C
\]
is an equivalence in $\KKR$. But this readily follows from \Cref{cor: s fib seqn}.3.
    
\end{proof}
\subsection{The Bott map in $\KKRGRA$ and consequences}
We can now export results from \cite[Section 9]{D} in our setup. We begin with the definition of the Bott class. For $V$ any Euclidean vector space equipped with a self-adjoint involution ($V$ was called a \emph{pseudo Euclidean vector space} in \cite[Section 9.1]{D}), we denote by $\mathcal{C}(V)$ the graded $C^*$-algebra $C_0(V,Cl(V))$. For the next few results, we do not consider the external group $G$.
\begin{defi}
    Let $V$ be a pseudo Euclidean vector space. Let $\iota: V \rightarrow Cl(V)$ denote the canonical inclusion of $V$ into $Cl(V)$. Then the functional calculus of $\iota$ yields a graded homomorphism
    \[
    b_V:\s \rightarrow \mathcal{C}(V),
    \]
    which we call the Bott map. It naturally determines a point in the spectral $K$-theory space
$
\mathbb{\hat K}\bigl(\mathcal C(V)\bigr),
$
and hence also a point in
$
\Omega^\infty \wideKK_{!,\infty}
\bigl(\C,\mathcal C(V)\bigr).
$ In light of \Cref{prop: equivalence KKgra and KKsp},  we denote either of them by $b_V$.
\end{defi}
\begin{rem}\label{rem: Bott facts}
    \cite[Section 9]{D} contains a detailed proof that
$\pi_0\hat{\mathbb K}\bigl(\mathcal{C}(V)\bigr) \cong \Z$,
and that $b_V$ is a generator of this group. Consequently, $b_V$ is also a generator of
$\pi_0\operatorname{map}_{\wideKK_{!,\infty}}\bigl(\C,\mathcal{C}(V)\bigr) \cong \Z$.
Furthermore, the proof of \Cref{prop: omnibus graded k-theory}.4 shows that
\[b_V \times b_{\overline V}
\in
\pi_0\hat{\mathbb K}\bigl(
\mathcal{C}(V)\hat\otimes\mathcal{C}(\overline V)
\bigr)
\cong \pi_0\hat{\mathbb{K}}(\mathcal{C}(V \oplus \overline{V})) \cong \Z\]
is a generator of the group.
\end{rem}
We end this subsection with a technical lemma. Although the result is an immediate corollary of \Cref{thm: dirac-dual Dirac abstract}, we need the lemma in the proof of it, and so must have it a priori.
\begin{lem}\label{lem: product of bott elements iso}
    For any pseudo-Euclidean vector space $V$, the morphism
\[
b_V \times b_{\overline V}
\in
\wideKK_{!,0}\bigl(\C,\mathcal{C}(V)\hat\otimes\mathcal{C}(\overline V)\bigr)
\]
is an isomorphism in $\wideKK_{!,0}$.
\end{lem}
\begin{proof}
    \Cref{rem: Bott facts} shows that $b_V \times b_{\overline{V}}$ is a unit in $\KKANAnon(\C, \mathcal{C}(V) \hat\otimes \mathcal{C}(\overline{V}))$. On the other hand, the Toeplitz extension theory shows that there is a $\KKANAnon$-equivalence between $\C$ and $\mathcal{C}(V) \hat\otimes \mathcal{C}(\overline{V})$, given by the image of the boundary map of the Toeplitz extension under the canonical functor $\hoKKCTWO \rightarrow \KKANAnon$. This also has to be a unit in $\Z$ because of the following general fact: in any additive $1$-category $\mathcal{D}$, if there are isomorphic objects $x$ and $y$ such that $\operatorname{hom}_{\mathcal{D}}(x,y) \cong \Z$ as abelian groups, then the isomorphism necessarily corresponds to $\pm 1$.\footnote{a proof of this fact can be deduced from the elementary algebra fact that any commutative unital ring structure on the abelian group $\Z$ must be isomorphic to the standard ring structure on $\Z$ (\cite[Exercise 2.16]{ALUFFI}).} To prove our thesis, we simply use that if $f$ is an isomorphism in an additive category, then so is $-f$.
\end{proof}
We next prove the following result, which is well known in the classical frameworks of $KK$- and $E$-theory; see \cite{kasparov1981operator}, \cite{KASEQQUIVARIANT}, \cite{HG}. We remind the reader that we \emph{do not} yet know that our $\KKANAnon$ categories are equivalent to classical graded $KK$- or $E$-categories, and hence need to give an independent proof. 
\begin{thm}\label{thm: dirac-dual Dirac abstract}
For any pseudo Euclidean vector space $V$, the Bott element $b_V$ is an equivalence in $\KKANAnon$.
\end{thm}
We now build up to the proof of \Cref{thm: dirac-dual Dirac abstract}. We start off by spelling out the \emph{interchange law} in a monoidal ($\infty$)-categories.
\begin{lem}\label{lem: interchange law}
    Let $\mathcal{C}$ be a monoidal $(\infty)$-category, and let
$A,B,C,A',B',C'$ be objects of $\mathcal{C}$. For composable morphisms
\[
A \xrightarrow{f} B \xrightarrow{g} C,
\qquad
A' \xrightarrow{h} B' \xrightarrow{k} C',
\]
the functoriality of the tensor product
$
\otimes : \mathcal{C}\times\mathcal{C}\longrightarrow\mathcal{C}
$
gives a canonical homotopy
\[
(g\circ f)\otimes(k\circ h)
\simeq
(g\otimes k)\circ(f\otimes h)
\]
in the mapping space
$
\operatorname{Map}_{\mathcal{C}}
(A\otimes A',\,C\otimes C').
$
\end{lem}
\begin{lem}\label{lem: bott induces iso in KKANA}
    For any $X \in \CAlgGr$, the following diagram of abelian groups commutes 
\begin{equation}\label{eqn: diagram KKGRA}
\begin{tikzcd}
	{\KKANAnon(\mathcal{C}(V),X)}
	&&
	{\KKANAnon(\C,X)}
	\\
	{\KKANAnon\bigl(
		\mathcal{C}(V)\hat\otimes\mathcal{C}(\overline V),
		X\hat\otimes\mathcal{C}(\overline V)
	\bigr)}
	\\
	{\KKANAnon\bigl(
		\C,
		X\hat\otimes\mathcal{C}(\overline V)
	\bigr)}
	&&
	{\KKANAnon(\C,X)}
	\arrow["{b_V^*}", from=1-1, to=1-3]
	\arrow["{-\hat\otimes\mathcal{C}(\overline V)}", from=1-1, to=2-1]
	\arrow["{=}", no head, from=1-3, to=3-3]
	\arrow["{(b_V\times b_{\overline V})^*}", from=2-1, to=3-1]
	\arrow["{-\times b_{\overline V}}"', from=3-3, to=3-1]
\end{tikzcd}
\end{equation}
Furthermore, all the arrows in the above diagram are isomorphisms.
\end{lem}
\begin{proof}
    That all the morphisms in the diagram are additive is clear. That it also commutes is an application of \Cref{lem: interchange law}. In the notation of loc.\! cit.\! we take
    \[
    \mathcal{C}:=\KKANAnon, A=A'=\C, B=\mathcal{C}(V), B'=C'=\mathcal{C}(\overline{V}), C=X.
    \]
    As for the morphisms, we can take any $g \in \KKANAnon(\mathcal{C}(V),X)$, and
    \[
    f=b_V, h=b_{\overline{V}}, k=\operatorname{id}_{\mathcal{C}(\overline{V})}.
    \]
    Since \Cref{eqn: diagram KKGRA} commutes, it suffices to show that the two left vertical arrows and the bottom horizontal arrows are isomorphisms. That \[- \hat\otimes \mathcal{C}(\overline{V}): \KKANAnon(\mathcal{C}(V),X) \rightarrow \KKANAnon(\mathcal{C}(V) \hat\otimes \mathcal{C}(\overline{V}), X \hat\otimes \mathcal{C}(\overline{V}))\] is an isomorphism follows because tensoring with $\mathcal{C}(\overline{V})$ yields an isomorphism in $\hoKKCTWO$ (with the inverse being tensoring with $\mathcal{C}(V)$), and because the cofree comodule functor $\KKRCTWO \xrightarrow{\s \hat\otimes -} \KKRGRAnon$ is symmetric monoidal. That $b_V \times b_{\overline{V}}$ is an isomorphism follows from \Cref{lem: product of bott elements iso}, and the bottom horizontal arrow being an isomorphism is the content of \cite[Theorem 9.13]{D}.
\end{proof}
\begin{rem}
    We mention that the proof of \cite[Theorem 9.13]{D} uses non-trivial functional analysis and features of classical graded $K$-theory (such as graded asymptotic morphisms between graded $C^*$-algebras induce morphisms between the graded $K$-groups). However,  the ``$K$-theoretic" result is enough for the proof of \Cref{lem: bott induces iso in KKANA}, and we do not need to know that there is any relation between the $\wideKK_{!,0}$ categories and asymptotic morphisms. And indeed, when $!=\operatorname{se}$, a general asymptotic morphism would not induce a morphism in $\wideKK_{\mathrm{se},0}$.
\end{rem}
\begin{proof}[Proof of \Cref{thm: dirac-dual Dirac abstract}]
    This readily follows from the Yoneda lemma and \Cref{lem: bott induces iso in KKANA}.
\end{proof}
We next deduce an important corollary of \Cref{thm: dirac-dual Dirac abstract}. We can now bring back our external compact group $G$. 
\begin{cor}\label{cor: loops is clifford one zero}
    In $\KKRGRA$, we have an equivalence $C_0(V) \simeq Cl(\overline{V})$. In particular, the functor
    \[
    - \hat\otimes Cl^{1,0}:\KKRGRA \rightarrow \KKRGRA
    \]
    is equivalent to the loops functor on $\KKRGRA$.
\end{cor}
\begin{proof}
    We begin by first observing that \Cref{thm: dirac-dual Dirac abstract} also shows that the Bott element is an equivalence in $\KKRGRA$. This is because $\mathbb{C}$ and $\mathcal{C}(V)$ have the trivial $G$-action, and we have an obvious ``equipping with the trivial $G$-action" functor $\KKRGRAnon \rightarrow \KKRGRA$.\footnote{More formally, we could have developed a graded variant of the change of group functors, and deduced this from the functoriality of the restriction functor. But because we would not have the occassion to use it anywhere else, we omit the details.}
    
    It suffices to show that $C_0(V)$ and ${Cl}(\overline{V})$ are isomorphic in $\KKANA$. But this follows from the implications
    \[
    \C \cong \mathcal{C}(V) \implies Cl(\overline{V}) \cong C_0(V) \hat\otimes Cl(V) \hat\otimes Cl(\overline{V}) \implies Cl(\overline{V}) \cong C_0(V),
    \]
    where we used that tensoring with $Cl(\overline{V})$ is already an isomorphism in $\hoKKGTWO$, and because $\C \rightarrow Cl(V \oplus \overline{V})$ is a $\hoKKGTWO$-equivalence.
    
    The last assertion follows from the observation that the canonical functor $\KKRGTWO \rightarrow \KKRGRA$ is left-exact, and loops in $\KKRGTWO$ is modeled by $C_0(\R^{0,1})$.
\end{proof}
We end this section by showing a fundamental feature (or defect!) of graded $\wideKK_{!,\infty}$-categories. Recall that (equivariant) $KK_{!,\infty}$-categories for ungraded, separable $C^*$ algebras enjoy the universal property of being a Dwyer-Kan localization at certain class of morphisms. We show that the analogous statement for $\wideKK_{!,\infty}$ categories is false. This issue seems to not have been discussed anywhere in the literature, and Bunke-Duenzinger left it as an open question (in the complex case) in \cite{BDGRADED}. The author learnt the crux of the proof from Johannes Ebert. The proof in its present form was suggested by Maxime Ramzi.
\begin{prop}\label{prop: graded kk not dk loc}
    The full subcategory of $\KKRGRAnon$ spanned by separable, graded $C^*$-algebras (denoted $\wideKK_{!,\infty}^{\operatorname{sep}}$) is not the Dwyer-Kan localization of the category of separable, graded $C^*$-algebras at any class of morphisms.
\end{prop}
\begin{proof}
    Since the grading on a graded $C^*$-algebra induces a $C_2$-action on the graded $K$-groups, we have a commutative diagram of functors
\[
\begin{tikzcd}
	&& {C_2\!\operatorname{-Ab}} \\
	{\CAlgGr^{\sep}} && {\operatorname{Ab}} \\
	\\
	{\wideKK^{\sep}_{!,\infty}}
	\arrow["{\operatorname{forget}}", from=1-3, to=2-3]
	\arrow["{\hat{K}_0(-)}", from=2-1, to=1-3]
	\arrow["{\hat{K}_0(-)}", from=2-1, to=2-3]
	\arrow["{\widekk^{\sep}_{!,\infty}}"', from=2-1, to=4-1]
	\arrow["{\pi_0\wideKK^{\sep}_{!,\infty}(\C,-)}"', from=4-1, to=2-3]
\end{tikzcd}
\]
Note that the bottom triangle commutes from \Cref{prop: equivalence KKgra and KKsp}. But since the forgetful functor $C_2\!\operatorname{-Ab} \rightarrow \operatorname{Ab}$ reflects isomorphisms, had $\widehat{KK}_{!,\infty}^{\operatorname{sep}}$ been a Dwyer-Kan localization, we would obtain a lift in the above diagram $\wideKK_{!,\infty}^{\operatorname{sep}} \rightarrow C_2\!\operatorname{-Ab}$, and thereby conclude that the $\hat{K}_0$-groups of two $\wideKK_{!,\infty}^{\operatorname{sep}}$-equivalent $C^*$-algebras are in fact isomorphic as $C_2$-abelian groups. In light of \Cref{thm: dirac-dual Dirac abstract}, this a contradiction, since $\hat{K}_0(\C) \cong \Z$ with the trivial $C_2$-action, whereas $\hat{K}_0(C_0(\R,Cl^{0,1})) \cong \Z$ with the action $n \mapsto -n$, as shown in \Cref{prop: omnibus graded k-theory}.4. 
\end{proof}
\begin{rem}
    Note that there is an obvious equivariant version of \Cref{prop: graded kk not dk loc} too.
\end{rem}
\section{Applications: certain fiber sequences connecting different K-theories}
    We now use the technology developed in the previous section in order to study the relationships between graded $K$-theory, $C_2$-equivariant $K$-theory and ungraded (classical) $K$-theory. All the results can be deduced in the presence of an external \emph{compact} group $G$, as we will be doing. The following remarks are in order.
    \begin{rem}\label{rem: obvious}
        We will be using the following elementary observations in this section even without mentioning.
        \begin{enumerate}
            \item The datum of a $G \times H$-action on a $C^*$-algebra $A$ is same as the datum of a $G$-action and a $H$-action on $A$ \emph{which commute}. Under this identification, it makes sense to talk about the $G$- and $H$-actions on a $(G\times H)$-$C^*$-algebra. Similar comment also applies to equivariant projective modules. For $G_2$-equivariant projective modules, we will be omitting the $G$-action from the notation, and just keep track of the $C_2$-action. For example, \emph{when we write $(A,\alpha) \in \GTWOALG$, we mean the $C_2$-action is given by $\alpha$.}
            \item Suppose $A,B,C \in \GTWOALG$ have trivial $G$-action, and fit in a fiber sequence in $\KKRCTWO$. Then, since $\res^{C_2}_{G_2}$ is left-exact, they also fit in a fiber sequence in $\KKRGTWO$.
            \item  With \Cref{thm: Recovering classical K-theories from our stable inf cat gen} and \Cref{prop: equivalence KKgra and KKsp} at our disposal, we will henceforth be (abusing notation and) denoting by $\K^G$ (resp. $\hat{\K}^G$) the $G$-equivariant (resp. graded $G$-equivariant) $K$-theory \emph{spectra}, given by $KK_{!,\infty}^{G}(\C,-)$ (resp. $\KKRGRA(\C,-) \simeq KK_{!,\infty}^{G_2}(\s,-)$). As before, we will be denoting their homotopy groups by $K_*^G(-)$ (resp. $\hat{K}_*^G(-)$).
            \item \emph{In the remainder of this section, we fix a compact group $G$.}  For concreteness, we work with $\KKRex$ in this section. The same calculations can also be done with $KK_{\operatorname{se},\infty}^{G_2}$.
        \end{enumerate}
    \end{rem}
    \subsection{A fiber sequence connecting the different $K$-theory spectra}
    The low hanging fruit that we can now pluck off is the following. We will be using the shorthand $(-)_u$ to denote the functor $\res^{G_2}_G(-)$ ($u$ stands for the ``underlying ungraded $C^*$-algebra").
\begin{prop}\label{lem: fib seqn}
    We have a natural fiber sequence of spectra \begin{equation}\label{eqn: first fib seqn}
         \K^G(A_u) \rightarrow \K^{G_2}(A) \rightarrow \hat{\K}^G(A)
    \end{equation}
    for every $A \in \GTWOALG$.
\end{prop}
\begin{proof}
Upon applying the exact functor
\[
\KKRex(-,A): \GTWOALG \longrightarrow \operatorname{Sp}
\]
to the fiber sequence of \Cref{cor: s fib seqn}.2, we are done once we justify the maps $(1)$, $(2)$, and $(3)$ in the following diagram:
\[
\begin{tikzcd}
	{\KKRex(Cl^{0,1},A)}
	&&
	{\KKRex(\C,A)}
	&&
	{\KKRex(\s,A)}
	\\
	{\mathbb K^G(A_u)}
	&&
	{\mathbb K^{G_2}(A)}
	&&
	{\hat{\mathbb K}^{G}(A)}
	\arrow["{\iota^*}", from=1-1, to=1-3]
	\arrow["{\simeq\,(1)}"', tail reversed, from=1-1, to=2-1]
	\arrow["{\operatorname{ev}_0^*}", from=1-3, to=1-5]
	\arrow["{\simeq\,(2)}"', tail reversed, from=1-3, to=2-3]
	\arrow["{\simeq\,(3)}"', tail reversed, from=1-5, to=2-5]
\end{tikzcd}
\]
The map labeled $(1)$ arises from the induction-restriction adjunction (\Cref{thm: change of group general}), since we can identify
$
\mathsf{ind}^{G_2}_G \mathbb{C}$ with $Cl^{0,1}$.
The second and third equivalence follow from \Cref{thm: Recovering classical K-theories from our stable inf cat gen}.
\end{proof}
In the next two remarks, we explain a couple of different formulations of the fiber sequence in \Cref{eqn: first fib seqn}. Though the above sequence will be our main workhorse, these other formulations come handy in certain cases.
\begin{rem}
     We claim that there is a fiber sequence 
    \begin{equation}\label{eqn: another formulation}
        \K^G(A_u) \xrightarrow{(A\rightarrow A \hat\otimes Cl^{0,1})_* } \K^G((A \hat\otimes Cl^{0,1})_u) \rightarrow \hat{\K}^G(A)
    \end{equation}
    which is equivalent to the fiber sequence in \Cref{eqn: first fib seqn}. Indeed, this will follow from the following commutative diagram
\begin{equation}\label{eqn: abs formulation fib seqn}
\begin{tikzcd}
	&&& {\KKRex(\C,A)} \\
	{\KKRex(Cl^{0,1},A)} &&& {\KKRex(Cl^{0,1},A\hat\otimes Cl^{0,1})} \\
	\\
	{\ER(\C,A_u)} &&& {\ER(\C,(A\hat\otimes Cl^{0,1})_u)}
	\arrow["{-\hat\otimes Cl^{0,1}}", from=1-4, to=2-4]
	\arrow["{\iota^*}", from=2-1, to=1-4]
	\arrow["{(A\rightarrow A\hat\otimes Cl^{0,1})_*}"', from=2-1, to=2-4]
	\arrow["{\mathsf{ind}^{G_2}_G \dashv \mathsf{res}^{G_2}_G}"', from=2-1, to=4-1]
	\arrow["{\mathsf{ind}^{G_2}_G \dashv \mathsf{res}^{G_2}_G}", from=2-4, to=4-4]
	\arrow["{(A\rightarrow A\hat\otimes Cl^{0,1})_*}"', from=4-1, to=4-4]
\end{tikzcd}
\end{equation}
Note that the top triangle commutes because the diagram below commutes in $\KKRex$, where $f: Cl^{0,1} \rightarrow A$ is any morphism in $\KKRex$. The left triangle below \emph{does not} commute at the level of algebras, but only up to homotopy.
\begin{equation}\label{eqn: Morita move}
    \begin{tikzcd}
	&& {Cl^{0,1}} && A \\
	{Cl^{0,1}} && {Cl^{0,2}} && {A \hat\otimes Cl^{0,1}}
	\arrow["f", from=1-3, to=1-5]
	\arrow["{=}"', no head, from=1-3, to=2-1]
	\arrow["{1_{Cl^{0,1}} \hat\otimes \iota}", from=1-3, to=2-3]
	\arrow["{1_A \hat\otimes \iota}", from=1-5, to=2-5]
	\arrow["{\iota \hat\otimes Cl^{0,1}}"', from=2-1, to=2-3]
	\arrow["{f \hat\otimes 1_{Cl^{0,1}}}"', from=2-3, to=2-5]
\end{tikzcd}
\end{equation}
\end{rem}
\begin{rem}\label{rem: abs style seqn}
    We claim there is a fiber sequence
    \begin{equation}\label{eqn: fib seq ABS style}
         \K^{G_2}(A \hat\otimes Cl^{1,0}) \xrightarrow{(A \rightarrow A \hat\otimes Cl^{1,0})^*} \K^{G_2}(A) \rightarrow \hat{\K}^G(A)
    \end{equation}
   which is equivalent to the fiber sequence in \Cref{eqn: first fib seqn}. The map labeled $A \rightarrow (A \hat\otimes Cl^{1,0})^*$ should be taken with a pinch of salt as $\K^{G_2}(-)$ is naturally covariantly functorial; what we mean is the map obtained from the zig-zag \[A\hat\otimes Cl^{1,0} \rightarrow A \hat\otimes Cl^{1,1} \xleftarrow{\simeq} A.\] When $A$ is unital, it can be shown (\cite[Section 4.5.2]{AS}) that on $\pi_0$, this has the same effect as precomposing the structure map of a $G_2$-equivariant $A \hat\otimes Cl^{1,0}$ projective module with the homomorphism $A \rightarrow A \hat\otimes Cl^{1,0}$. The equivalence of the fiber sequences in \Cref{eqn: first fib seqn} and \Cref{eqn: fib seq ABS style} can be shown by a similar treatment as in the last remark.

    This formulation of the fiber sequence is directly related to the classical Atiyah-Bott-Shapiro description of $K$-groups of a point, see \Cref{rem: abs} for more details.
\end{rem}
\begin{prop}
    When $A$ is inner graded, the fiber sequence in \Cref{eqn: first fib seqn} (and hence also the ones in \Cref{eqn: another formulation} and \Cref{eqn: fib seq ABS style}) splits.
\end{prop}
\begin{proof}
    In light of \Cref{eqn: another formulation}, we have to show that the map of spectra
    \[
\ER(\C,A_u)
\xrightarrow{(A\to A\hat\otimes Cl^{0,1})_*}
\ER\bigl(\C,(A\hat\otimes Cl^{0,1})_u\bigr).
\]
    has a retraction when $A$ is inner graded. In this case, since $A \hat\otimes Cl^{0,1} \cong A \otimes Cl^{0,1}$ in $\GTWOALG$, we can extend the aforementioned map to
   \[
\begin{aligned}
\ER(\C,A_u)
&\xrightarrow{(A\to A\hat\otimes Cl^{0,1})_*}
\ER\bigl(\C,(A\hat\otimes Cl^{0,1})_u\bigr)
\\
&\xrightarrow{\simeq}
\ER\bigl(\C,(A\otimes Cl^{0,1})_u\bigr)
\xrightarrow{\simeq}
\ER(\C,A_u\oplus A_u)
\\
&\xrightarrow{\simeq}
\ER(\C,A_u)\oplus \ER(\C,A_u).
\end{aligned}
\]
    It can be directly checked that this composite is the diagonal map of $KK_{\operatorname{ex},\infty}^{G}(\C,A_u)$. Since the diagonal map of any spectrum has two canonical retractions, we are done.
\end{proof}
\begin{rem}
    At this stage, the reader may want to see an example where the fiber sequence in \Cref{eqn: first fib seqn} does not split. This is intimately related to a folklore question (\Cref{que: que}), and we discuss it in \Cref{rem: abs}.
\end{rem}
One next want to analyze the fiber sequence in \Cref{eqn: first fib seqn} when $A$ is not inner graded. To that end, we give an explicit description of the maps in the sequence at level of $\pi_0$.
\begin{lem}\label{lemlem}
    Let $(A,\alpha)$ be a unital algebra in $\GTWOALG$. We have a natural group homomorphism
    \[
    F_A:K_0^{G_2}(A) \rightarrow {K}_0^{G,\operatorname{Fr}}(A)
    \]
   given by taking a finitely generated $G_2$-equivariant projective $A$-module $M$ to the Fredholm cycle $(M,0) \in K_0^{G,\operatorname{Fr}}(A)$.
\end{lem}
\begin{proof}
    The assertion is clear, modulo the verification that every finitely generated $G_2$-equivariant projective $A$-module is canonically a graded $G$-equivariant Hilbert $A$-module. This follows by an easy case by case analysis with signs in light of \cite[Proposition 11.2.3]{BLAK}.
\end{proof}
\begin{cor}
    Following the notation of \Cref{lemlem}, the diagram
\[
\begin{tikzcd}
	{\pi_0\KKRex(\C,A)}
	&&
	{\pi_0\KKRex(\s,A)}
	\\
	{K_0^{G_2}(A)}
	&&
	{\hat{K}_0^{G,\operatorname{Fr}}(A)}
	\arrow["{\operatorname{ev}_0^*}", from=1-1, to=1-3]
	\arrow["\cong", from=1-3, to=2-3]
	\arrow["\cong", from=2-1, to=1-1]
	\arrow["{F_A}", from=2-1, to=2-3]
\end{tikzcd}
\]
commutes.
\end{cor}
\begin{proof}
    Given a $G_2$-equivariant projective $A$-module $M$, we study the fate of $[M] \in K_0^{G_2}(A)$ along the up-right-down itinerary. 

    The classifying map of $M$ is a $G_2$-equivariant $*$-homomorphism 
    \[
    \varphi: \C \rightarrow \mathcal{K}_A(V \otimes A),
    \]
    where $V$ is a finite dimensional $G_2$-vector space. $\varphi$ is the image of $[M]$  in $\pi_0(KK_{\operatorname{ex},\infty}^{G_2}(\mathbb{C},A))$ under the left vertical map, as can be seen by unraveling the maps in \Cref{thm: Recovering classical K-theories from our stable inf cat gen} and the equivariant Kasparov stabilization theorem. Also note that $\mathcal{K}_A(V \otimes A)= \mathcal{B}_A(V \otimes A)$ since $\mathcal{K}_A(V \otimes A)$ is unital, see \cite[Theorem 15.4.2]{WG}. It remains to determine the image of $\operatorname{\varphi \circ ev_0}$ under the right vertical map. From \Cref{thm: spectral model of g eq k theory}.5, one gets that this image is the Fredholm cycle (we are also using the notation of the proof of \Cref{thm: spectral model of g eq k theory}.5) \[(\overline{\varphi \circ \operatorname{ev}_0(\s)(V \otimes A)},\Tilde{\varphi}(x))=(M,0)\]
    Indeed, this follows from the commutative diagram 
\[\begin{tikzcd}
	{\mathcal{S}} & {\mathcal{K}_A(\overline{\varphi \circ \operatorname{ev}_0(\mathcal{S})(V \otimes A)})} & {\mathcal{K}_A(V \otimes A)} \\
	{C[-1,1]} & {\mathbb{C}}
	\arrow["{(\varphi \circ \operatorname{ev}_0)^{\operatorname{nd}}}", from=1-1, to=1-2]
	\arrow[hook, from=1-1, to=2-1]
	\arrow["{\operatorname{ev}_0}"{pos=0.8}, from=1-1, to=2-2]
	\arrow["\iota", from=1-2, to=1-3]
	\arrow["{\exists !}"{pos=0.1}, dashed, from=2-1, to=1-2]
	\arrow["{\operatorname{ev}_0}"', from=2-1, to=2-2]
	\arrow["{\varphi^{\operatorname{nd}}}"', from=2-2, to=1-2]
	\arrow["\varphi"', from=2-2, to=1-3]
\end{tikzcd}\]
The dotted extension of $(\varphi \circ \operatorname{ev}_0)^{\operatorname{nd}}$ is unique. However, $\varphi^{\operatorname{nd}}\circ \operatorname{ev}_0$ is yet another extension, whence they must agree. Our thesis is deduced readily from the latter.  
\end{proof}
Before stating the next result, we need to setup some notation
\begin{notn}
    Let $(A,\alpha) \in \GTWOALG$. Denote by $A_\alpha$ the $G$-equivariant $A_u$-module $\underline{A}$ (the underlying abelian group of $A$) with the $A_u$-action given by $\underline{a}.b:=\underline{a\alpha(b)}$ and the $G$-action given by $g.\underline{a}:=\underline{ga}$. Note that $A_\alpha$ is a $A_u$-$A_u$ bimodule in a canonical way.
\end{notn}
\begin{lem}\label{lem: GA definition}
    Let $(A,\alpha)$ be a unital algebra in $\GTWOALG$. We have a natural group homomorphism
    \[
    G_A:K_0^G(A_u) \rightarrow K_0^{G_2}(A)
    \]
given by taking a finitely generated $G$-equivariant projective $A_u$-module $M$ to the $G_2$-equivariant finitely generated projective $A$-module $M \oplus M_\alpha$, where $M_{\alpha}:=M \otimes_{A_u} A_{\alpha}$\footnote{explicitly, $M_\alpha$ is the abelian group $\underline{M}$ with the $A_u$-action given by $\underline{m}.a:=\underline{m\alpha(a)}$. $M$ and $M_{\alpha}$ are not isomorphic as  $A_u$-modules in general.} and the $C_2$-action on $M \oplus M_{\tau}$ is flip, i.e. given by $(m,\underline{n}) \mapsto (n,\underline{m})$.
\end{lem}
\begin{proof}
    Obvious.
\end{proof}
\begin{cor}\label{cor: map induced by C in CL01}
    Following the notation of \Cref{lem: GA definition}, the diagram
\[
\begin{tikzcd}
	{\pi_0\KKRex(Cl^{0,1},A)}
	&&
	{\pi_0\KKRex(\C,A)}
	\\
	{K_0^G(A_u)}
	&&
	{K_0^{G_2}(A)}
	\arrow["{\iota^*}", from=1-1, to=1-3]
	\arrow["\cong", from=1-3, to=2-3]
	\arrow["\cong", from=2-1, to=1-1]
	\arrow["{G_A}", from=2-1, to=2-3]
\end{tikzcd}
\]
commutes.
\end{cor}
\begin{proof}
    This follows by chasing through the maps in the diagram. Given a $G$-equivariant finitely generated projective $A_u$-module $M$, the classifying map of $[M] \in K_0(A_u)$ is a $G$-equivariant $*$-homomorphism
    \[
    p:\mathbb{C} \rightarrow \mathcal{K}_{A_u}(V \otimes A_u)
    \]
    where $V$ is some finite dimensional $G$-vector space. We therefore obtain
    \[
    \ind^{G_2}_G \C \xrightarrow{\ind^{G_2}_Gp} \ind^{G_2}_G(\mathcal{K}_{A_u}(V \otimes A_u)) \hookrightarrow \mathcal{K}_A(V \otimes A_u \oplus V \otimes A_u, \operatorname{twisted~flip}),
    \]
    where for the last inclusion we used elementary properties of induction for group representations on vector spaces and $C^*$-algebras (see \cite[Lemma 4.10]{BEL}), as well as the description of the co-unit of the induction-restriction adjunction. Precomposing with the inclusion $\C \rightarrow \ind^{G_2}_G \C$, we conclude that the up-right-down itinerary maps $p$ to the homomorphism
    \[
    \tilde{p}: \C \rightarrow \mathcal{K}_A(V \otimes A_u \oplus V \otimes A_u, \operatorname{twisted~flip}), ~~~~\tilde{p}(1)=(p(1) \oplus\alpha_*p(1)).
    \]
    It remains to check that the $G_2$-equivariant $A$-modules $(p(1)(V \otimes A_u) \oplus \alpha_*p(1)(V \otimes A_u), \operatorname{twisted~flip})~$and $(M \oplus M_\alpha,\operatorname{flip})$ are isomorphic. We have $p(1)(V \otimes A_u) \cong M$ as $G$-equivariant $A_u$-modules by definition. To see that $\alpha_*p(1)(V \otimes A_u) \cong M_{\alpha}$ as $G$-equivariant $A_u$-modules, note that we have 
    \[
    \alpha_*p(1)(V \otimes A_u)=\{v \otimes \alpha a \mid v \otimes a \in p(1)(V \otimes A_u)\}.
    \]
    The map
    \begin{align*}
        \varphi: (p(1)(V \otimes A_u))_{\alpha} &\rightarrow \alpha_*p(1)(V \otimes A_u) \\
        \underline{v \otimes a} &\mapsto v\otimes \alpha(a)
    \end{align*}
 now readily produces an isomorphism between the $G$-equivariant $A_u$-modules $M_\alpha$ and $\alpha_*p(1)(V \otimes A_u)$. It is now also direct to check that juxtaposing the two isomorphisms yields a $G_2$-equivariant isomorphism between $(M \oplus M_\alpha,\operatorname{flip})$ and $(p(1)(V \otimes A_u) \oplus \alpha_*p(1)(V \otimes A_u), \operatorname{twisted~flip}).$
\end{proof}
We end this subsection with the following remark, which connects our results to the classical result of Atiyah-Bott-Shapiro as mentioned in \Cref{thm: Atiyah-Bott-Shapiro}. For simplicity, we ignore the external group $G$ in this remark.
\begin{rem}\label{rem: abs}
    When $A=Cl^{n,0}$, Atiyah-Bott-Shapiro asserts that there are short exact sequences of abelian groups
    \begin{equation}\label{eqn: abs classical}
         0 \rightarrow K_0^{C_2}(Cl^{n+1,0}) \xrightarrow{(Cl^{n,0} \rightarrow Cl^{n+1,0})^*} K_0^{C_2}(Cl^{n,0}) \rightarrow \hat{K}_0(Cl^{n,0}) \rightarrow 0,
    \end{equation}
    where the first map is as described in \Cref{rem: abs style seqn}. Since $\hat{K}_0(Cl^{n,0}) \cong K_n(\C)$, one obtains a convenient (and perhaps miraculous!) description of higher $K$-groups of a point as certain \emph{Clifford module quotients}. This immediately makes one ask the following question.
    \begin{que}\label{que: que}
        Can the higher $K$-groups of any (graded), (Real) unital $C^*$-algebra $A$ be expressed in terms of Clifford module quotients a'la Atiyah-Bott-Shapiro? More precisely, does there exist a short exact sequence of abelian groups of the form
        \[
        0 \rightarrow K_0^{C_2}(A \hat\otimes Cl^{n+1,0}) \xrightarrow{(A \hat\otimes Cl^{n,0} \rightarrow A \hat\otimes Cl^{n+1,0})^*} K_0^{C_2}(A \hat\otimes Cl^{n,0}) \rightarrow \hat{K}_n(A) \rightarrow 0~~~~~~~~~~?
        \]
    \end{que}
    We explain that the answer to this question is in the negative. To that end, we recall that to show exactness of \Cref{eqn: abs classical} at $\hat{K}_0(Cl^{n,0})$, one explicitly constructs a section of the map $K_0^{C_2}(Cl^{n,0}) \rightarrow \hat{K}_0(Cl^{n,0})$ using the notion of index of a Fredholm operator on a \emph{Hilbert space} with Clifford symmetries. The construction of this index map crucially needs the analytic fact that the kernel of a Fredholm operator on a Hilbert space is automatically finitely generated projective. Since the same is not true for a general Hilbert module over a $C^*$-algebra, the index of a Fredholm operator over an \emph{ungraded} unital $C^*$-algebra $B$ is typically defined using the notion of \emph{taming} (as described in detail in \cite[Section 8.2]{D}). Roughly speaking, taming works by adding a dummy finite dimensional free $B$-module to the game, and then  subtracting it later in the group $K_0(B)$. When $B$ is \emph{graded}, the taming construction can fail; for example when $B:=A \hat\otimes Cl^{n,0}$ the taming construction can be seen to not necessarily preserve the relevant Clifford symmetries. So in general, one cannot expect to construct the crucial index map witnessing a section of the map $K_0^{C_2}(A \hat\otimes Cl^{n,0}) \rightarrow \hat{K}_0(A \hat\otimes Cl^{n,0})$.
    
    And indeed, concrete counterexamples exist. Taking $A$ to be the complex $C^*$-algebra $C(S^1,\C)$ and $n=1$, we know that $\hat{K}_1(C(S^1,\C)) \cong K_1(C(S^1,\C)) \cong \Z$, but the relevant {Clifford module quotient} can be shown to be $0$. In context of the fiber sequence in \Cref{eqn: fib seq ABS style}, this failure can be explained to be due to the corresponding boundary map being injective. More elaborations on this can be found in \cite[Example 4.51]{AS}. See also \cite[Appendix A]{KOEN} for other examples and discussions.
\end{rem}
\subsection{The dual fiber sequence and endomorphisms of $K$-groups}

In this subsection, we obtain a dual of the fiber sequence described in \Cref{lem: fib seqn}. The key to the existence of this fiber sequence is the equivalence between $Cl^{1,0}$ and $C_0(\R^{0,1})$ in $\KKRGTWO$, as shown in \Cref{cor: loops is clifford one zero}. 
\begin{prop}\label{prop: Cl10 fib seqn}
    We have a fiber sequence
    \[
    Cl^{1,0} \xrightarrow{j}\mathbb{C} \xrightarrow{p} \mathcal{S}
    \]
    in $\KKRex$, where $j:Cl^{1,0} \rightarrow \C$ is obtained from the zig-zag
    $
    Cl^{1,0} \rightarrow Cl^{1,1} \leftarrow \C,
    $
    and $p \in \pi_0(\KKRex(\C,\s))$ is the image under the $\res^{C_2}_{G_2}$ homomorphism of a generator of $K_0^{C_2}(\s)\cong \Z$ (\Cref{prop: K0C2S}). 
\end{prop}
\begin{proof}
    We begin the proof by mentioning that the assertion of the proposition is indeed just: in any stable $\infty$-category $\mathcal{C}$, a fiber sequence $A \xrightarrow{f} B \xrightarrow{g} C$ naturally yields a equivalent fiber sequence $A \xrightarrow{f} B \xrightarrow{-g} C$. In light of \Cref{rem: obvious}.2, it suffices to demonstrate the existence of the fiber sequence in $\KKRexTWO$.
    
    We begin by applying the functor $- \hat\otimes Cl^{1,0}$ to \Cref{eq: s fib seqn}, thereby obtaining the fiber sequence
    \[
    \s \hat\otimes Cl^{1,0} \xrightarrow{\operatorname{ev}_0{_*}} Cl^{1,0} \xrightarrow{\operatorname{\iota}_*} Cl^{1,1}.
    \]
    Now, \Cref{thm: kkgraeq vs kkctwoeq} and \Cref{cor: loops is clifford one zero} show that we have a fiber sequence 
    \[
    \Omega\s \rightarrow Cl^{1,0} \xrightarrow{\iota_*} Cl^{1,1}
    \]
    in $\KKRexTWO$. Taking cofibering in $\KKRexTWO$ and using the equivalence $\mathbb{C} \simeq Cl^{1,1}$ we obtain a fiber sequence
    \[
    Cl^{1,0} \xrightarrow{j} \mathbb{C} \rightarrow \s.
    \]
    To prove our thesis, it will suffice to show that the second morphism above is a generator of $K_0^{C_2}(\s) \cong \Z$. Consider the long exact sequence of $K_*^{C_2}(-)(=\pi_*\KKRexTWO(\C,-))$ groups associated to this fiber sequence
    \[
    ... \rightarrow K_0^{C_2}(Cl^{1,0}) \xrightarrow{j_*} K_0^{C_2}(\C) \rightarrow K_0^{C_2}(\s) \xrightarrow{\partial} K_{-1}^{C_2}(Cl^{1,0}) \rightarrow ...
    \]
    By tensor invertibility of $Cl^{0,1}$ in $\KKRexTWO$ and the induction-restriction adjunction, we have \[
     K_{-1}^{C_2}(Cl^{1,0})  \cong K_{-1}(\mathbb{C}_u) \cong 0, 
    \]
    whence the map 
    \[
    K_0^{C_2}(\C) \rightarrow K_0^{C_2}(\s) 
    \]
    above is surjective. We now consider the square
\[\begin{tikzcd}
	{K_0^{C_2}(\mathbb{C})} && {K_0^{C_2}(\mathcal{S})} \\
	{\mathbb{Z} \oplus \mathbb{Z}} && {\mathbb{Z}}
	\arrow[two heads, from=1-1, to=1-3]
	\arrow["\cong"', from=1-1, to=2-1]
	\arrow["\cong", from=1-3, to=2-3]
	\arrow["\dag", two heads, from=2-1, to=2-3],
\end{tikzcd}\]
where the vertical maps are isomorphisms as mentioned in \Cref{prop: KOC2} and \Cref{prop: K0C2S}, and the map $\dagger$ is so defined to make the diagram commute. Now, since the morphisms in the square are also $C_2$-equivariant and the $C_2$-abelian group $\Z \oplus \Z$ is generated by $(1,0)$, we conclude from the surjectivity of $\dag$ that it maps $(1,0)$ to $\pm1$.
\end{proof}
\begin{prop}\label{lem: dual fib seqn}
    We have a natural fiber sequence 
    \begin{equation}\label{eqn: dual fiber sequence}
            \hat{\K}^G(A) \xrightarrow{} \K^{G_2}(A) \xrightarrow{} \K^G((A \hat\otimes Cl^{0,2})_u)
    \end{equation}

    for every $A \in \GTWOALG$.
\end{prop}
\begin{proof}
We apply the exact functor \[
\KKRex(-,A):
\GTWOALG
\longrightarrow
\operatorname{Sp}.
\] to the fiber sequence in \Cref{prop: Cl10 fib seqn}. In light of the proof of \Cref{lem: fib seqn}, we only need to identify $\KKRex(Cl^{1,0},A)$ with $\K^G((A \hat\otimes Cl^{0,2})_u)$. For that we use the route
    \[
\resizebox{\textwidth}{!}{$
\KKRex(Cl^{1,0},A)
\xrightarrow[\simeq]{-\hat\otimes Cl^{0,2}}
\KKRex(Cl^{1,2},A\hat\otimes Cl^{0,2})
\xrightarrow[\simeq]{(Cl^{0,1}\rightarrow Cl^{1,2})^*}
\KKRex(Cl^{0,1},A\hat\otimes Cl^{0,2})
\xrightarrow[\simeq]{\mathsf{ind}\dashv\mathsf{res}}
\K^G\bigl((A\hat\otimes Cl^{0,2})_u\bigr)
$}.
\]
\end{proof}
At this juncture, we record an interpretation of the map $j^*:\KKRex(\C,A) \rightarrow \KKRex(Cl^{1,0},A)$ which will come handy later.
\begin{lem}\label{lem: another description of j star}
    The diagram
\[
\begin{tikzcd}
	{\KKRex(\C,A)}
	&&
	{\KKRex(Cl^{1,0},A)}
	\\
	{}
	&&
	{\KKRex(Cl^{1,1},A\hat\otimes Cl^{0,1})}
	\\
	{\KKRex(\C,A)}
	&&
	{\KKRex(\C,A\hat\otimes Cl^{0,1})}
	\\
	{} & {}
	\arrow["{j^*}", from=1-1, to=1-3]
	\arrow["{=}", no head, from=1-1, to=3-1]
	\arrow["{\hat\otimes Cl^{0,1},\,\simeq}", from=1-3, to=2-3]
	\arrow["{(\C\rightarrow Cl^{1,1})_*,\,\simeq}", from=2-3, to=3-3]
	\arrow["{(A\rightarrow A\hat\otimes Cl^{0,1})_*}", from=3-1, to=3-3]
\end{tikzcd}
\]
commutes.
\end{lem}
\begin{proof}
    We expand the diagram to
\[
\begin{tikzcd}
	{\KKRex(\C,A)} &&&& {\KKRex(Cl^{1,0},A)} \\
	{} & {\KKRex(Cl^{1,1},A)} &&& {\KKRex(Cl^{1,1},A \hat\otimes Cl^{0,1})} \\
	{\KKRex(\C,A)} &&&& {\KKRex(\C,A \hat\otimes Cl^{0,1})} \\
	{} & {}
	\arrow["{j^*}", from=1-1, to=1-5]
	\arrow["{=}", no head, from=1-1, to=3-1]
	\arrow["{\hat\otimes Cl^{0,1}}", from=1-5, to=2-5]
	\arrow["{(\C \rightarrow Cl^{1,1})^*}"{description}, from=2-2, to=1-1]
	\arrow["{(Cl^{1,0} \rightarrow Cl^{1,1})^*}"{description}, from=2-2, to=1-5]
	\arrow["{(A \rightarrow A \hat\otimes Cl^{0,1})_*}"', from=2-2, to=2-5]
	\arrow["{(\C \rightarrow Cl^{1,1})^*}"{description}, from=2-2, to=3-1]
	\arrow["{(\C \rightarrow Cl^{1,1})^*}", from=2-5, to=3-5]
	\arrow["{(A \rightarrow A \hat\otimes Cl^{0,1})_*}", from=3-1, to=3-5]
\end{tikzcd}
\]
It suffices to show that each of the four shapes with one vertex $\KKRex(Cl^{1,1},A)$ commutes. But all of them are clear (for the right-angled triangle shaped diagram, use an argument similar to the one for \Cref{eqn: Morita move}).
\end{proof}
Unlike in the case of \Cref{lem: fib seqn}, the induced maps at the level of $\pi_0$ groups of the fiber sequence in \Cref{eqn: dual fiber sequence} are not easy to write down explicitly. We instead discuss the various endomorphisms of the $K$-theory groups that arise from the fiber sequences in \Cref{eqn: first fib seqn} and \Cref{eqn: dual fiber sequence}.

Consider the following diagram in $\KKRexTWO$:
\[\begin{tikzcd}
	&& {Cl^{1,0}} & \\
	{} & {\mathcal{S}} & {\mathbb{C}} & {Cl^{0,1}} \\
	&& {\mathcal{S}}
	\arrow["j", from=1-3, to=2-3]
	\arrow["{\operatorname{ev_0}}", from=2-2, to=2-3]
	\arrow["\iota", from=2-3, to=2-4]
	\arrow["p", from=2-3, to=3-3]
\end{tikzcd}\]
\begin{lem}\label{lem: endo of K0C2}
    $\operatorname{ev}_0 \circ p = 1-x \in \hoKKRexT(\C,\C) \cong \frac{Z[x]}{x^2=1}$, with the isomorphism being as in \Cref{prop: KOC2}.1 (keeping in mind \Cref{thm: Recovering classical K-theories from our stable inf cat gen}\footnote{the two natural ring structures on the endomorphisms of the unit in any monoidal category coincide thanks to the Eckmann-Hilton argument. See also \cite[Lemma 5.1]{BUNLANPEN}.}).
\end{lem}
\begin{proof}
    From \Cref{prop: KOC2 CL01} and \Cref{prop: K0C2S}, we have a short exact sequence
    \[
    0\rightarrow K_0^{C_2}(\s) \xrightarrow{{\operatorname{ev}_0}_*} K_0^{C_2}(\C) \xrightarrow{\iota_*} K_0^{C_2}(Cl^{0,1})\rightarrow 0.
    \]
    The composition in the $\hoKKRexT$-category will be the image of $p$ under $\operatorname{ev_{0_*}}$. The rest follows from \Cref{prop: KOC2S to KOC2C commutes}.
\end{proof}
We now record a corollary of the above calculation which will come handy later.
\begin{cor}\label{cor: ev0 circ p is mult by 2}
    The diagram 
\[
\begin{tikzcd}
	{\hoKKRexT(\s,\C)} && {\hoKKRexT(\s,\C)} \\
	\Z && \Z
	\arrow["{(\operatorname{ev}_0 \circ p)_*}", from=1-1, to=1-3]
	\arrow["\cong", from=1-1, to=2-1]
	\arrow["\cong", from=1-3, to=2-3]
	\arrow["{n \mapsto 2n}", from=2-1, to=2-3]
\end{tikzcd}
\]
commutes.
\end{cor}
\begin{proof}
    It suffices to show that the diagram commutes on $[\operatorname{ev}_0]$. The bottom-right itinerary maps it to $2 \in \Z$. To analyse its fate under the right-down itinerary, we bring \Cref{lem: endo of K0C2} into play, which shows that the upper horizontal arrow is product in $\KKRCTWO$-product with $([\C^{\operatorname{triv}}]-[\C^{\operatorname{sign}}])$. 
   
    Postcomposition with $[\C]^{\operatorname{triv}}$ maps $[\operatorname{ev}_0]$ to itself, since $[\C^{\operatorname{triv}}]$ is the identity element of $\hoKKRexT(\C,\C)$. To prove our thesis, we have to show that postcomposition with $[\C^{\operatorname{sign}}]$ maps $[\operatorname{ev}_0]$ to $-[\operatorname{ev}_0]$. Unraveling the definition of postcomposition with $[\C^{\operatorname{sign}}]$, this follows directly from the proof of \Cref{thm: spectral model of g eq k theory}.1.
\end{proof}
\begin{cor}
    Let $A$ be a unital algebra in $\GTWOALG$. The composite
    \[
    K_0^{G_2}(A) \xrightarrow{\operatorname{ev}_{0}^*} \hat{K}^G_0(A) \xrightarrow{p^*} K^{G_2}_0(A)
    \]
    maps a finitely generated $G_2$-equivariant projective $A$-module $M$ to $[M]-[M^{\operatorname{op}}] \in K_0^{G_2}(A)$, where $M^{\operatorname{op}}$ is $M$ as a $G$-equivariant $A_u$-module, with the $C_2$-action given as the negative of the $C_2$-action on $M$.
\end{cor}
\begin{proof}
    Unraveling the definitions and from \Cref{lem: endo of K0C2}, we conclude that the composite is nothing but the composition with the image under $\res^{C_2}_{G_2}$ of $1-x \in \hoKKRex(\C,\C) \cong \frac{Z[x]}{x^2=1}$. Identifying $\hoKKRex(\C,\C)$ with the classical $G_2$-equivariant $K$-theory of a point, it is clear that $1$ corresponds to the trivial representation, and $x$ corresponds to the sign representation of $G_2$ on $\C$.\footnote{more precisely, $G$ acts trivially on $\C$ in both cases, and $C_2$ acts either trivially or via the sign representation.} Since the internal and external products
    \[
    \hoKKRex(\C,\C) \times \hoKKRex(\C,A) \rightarrow \hoKKRex(\C,A)
    \]
    are the same (this is true in any monoidal category), our thesis follows since the external multiplication of $[M]$ with $x$ yields $[M^{\operatorname{op}}]$ (keeping in mind the lax monoidality assertion of \Cref{thm: Recovering classical K-theories from our stable inf cat gen}).
\end{proof}
The proof of the next result requires slightly more work than needed in the proof of \Cref{lem: endo of K0C2}, but is similar in spirit. Somewhat surprisingly \Cref{cor: ev0 circ p is mult by 2} is needed in the proof!
\begin{lem}\label{lem: 1 plus x is the composition}
    $p \circ \operatorname{ev}_0=1+x \in \hoKKRex(\s,\s) \cong \frac{Z[x]}{x^2=1}$, with the isomorphism being as in \Cref{cor: EC2 SS ring}.
\end{lem}
\begin{proof}
 From the fiber sequence of \Cref{prop: Cl10 fib seqn}, 
    we obtain a long exact sequence of $\hat{K}_*(-)(=\hoKKRex(\s,-))$ groups
  \[
\cdots \longrightarrow
\hat K_0(Cl^{1,0})
\xrightarrow{j_*}
\hat K_0(\C)
\xrightarrow{p_*}
\hat K_0(\s)
\xrightarrow{\partial}
\cdots
\]
    The composition $p \circ \operatorname{ev}_0$ will be the image of $[\operatorname{ev}_0]$ under $p_*$. Since $p_*$ is a $C_2$-equivariant map, there exists $k \in \Z$ rendering the diagram
\begin{equation}\label{eqn: commuting upto sign}
\begin{tikzcd}
	{\hat K_0(\C)} && {\hat K_0(\s)} \\
	\Z && {\Z \oplus \Z}
	\arrow["{p_*}", from=1-1, to=1-3]
	\arrow["\cong"', no head, from=1-1, to=2-1]
	\arrow["\cong", no head, from=1-3, to=2-3]
	\arrow["{1 \mapsto (k,k)}", from=2-1, to=2-3]
\end{tikzcd}
\end{equation}
commutative, where the left isomorphism is the standard one (mapping $[\operatorname{ev}_0]$ to $1$), and the right one is the one mentioned in \Cref{prop: omnibus graded k-theory}.8.
\begin{clm}
    $k=1$ in the above diagram.
\end{clm}
\begin{proof}
        We expand the diagram to 
\begin{equation}\label{eqn: p circ ev0 claim}
\begin{tikzcd}
	{\hoKKRexT(\s,\C)} && {\hoKKRexT(\s,\s)} && {\hoKKRexT(\s,\C)} \\
	\Z && {\Z \oplus \Z} && \Z
	\arrow["{p_*}", from=1-1, to=1-3]
	\arrow["\cong", no head, from=1-1, to=2-1]
	\arrow["{\operatorname{ev_0}_*}", from=1-3, to=1-5]
	\arrow["\cong", no head, from=1-3, to=2-3]
	\arrow["\cong", no head, from=1-5, to=2-5]
	\arrow[from=2-1, to=2-3]
	\arrow["{(n,m) \mapsto n+m}", from=2-3, to=2-5]
\end{tikzcd}
\end{equation}
Observe that the right square commutes, as can be explicitly checked on the two generators $[\operatorname{id_\s}]$ and $[\alpha_{\s}]$. Furthermore, the rectangle
\[
\begin{tikzcd}
	{\hoKKRexT(\s,\C)} && {\hoKKRexT(\s,\C)} \\
	\Z && \Z
	\arrow["{(\operatorname{ev}_0 \circ p)_*}", from=1-1, to=1-3]
	\arrow["\cong", no head, from=1-1, to=2-1]
	\arrow["\cong", no head, from=1-3, to=2-3]
	\arrow["{n \mapsto 2n}", from=2-1, to=2-3]
\end{tikzcd}
\]
commutes too, as seen in \Cref{cor: ev0 circ p is mult by 2}. The claim follows.
\end{proof}
This also finishes the proof of the lemma.
\end{proof}
\begin{cor}
    Let $A$ be a unital algebra in $\GTWOALG$. The composite 
    \[
    {K}_0^{G,\operatorname{Fr}}(A) \xrightarrow{p^*} K_0^{G_2}(A) \xrightarrow{\operatorname{ev}_0^*} {K}^{G,\operatorname{Fr}}_0(A)
    \]
    maps the Fredholm cycle $(\E,F) \in {K}_0^{G,\operatorname{Fr}}(A)$ to $(\E \oplus \E,F \oplus (-F)) \in {K}_0^{G,\operatorname{Fr}}(A)$.
\end{cor}
\begin{proof}
    This follows from \Cref{lem: 1 plus x is the composition} and the formula of the explicit isomorphism $\pi_0{\K}^{\operatorname{G,sp}}(A) \cong {K}_0^{G,\operatorname{Fr}}(A)$ as mentioned in \Cref{thm: spectral model of g eq k theory}.6. 
\end{proof}
\begin{rem}\label{rem: simplify endo of K fr}
    Note that $(\E,-F)$ is \emph{not} the inverse of $(\E,F) \in {{K}}^{G,\operatorname{Fr}}_0(A)$; the usual formula of the inverse also requires to flip the grading on $\E$ (see \cite[Lemma 8.24]{D} for details). In fact, when $A$ is inner graded, $(\E,-F)=(\E,F)$, as can be seen from \Cref{thm: spectral model of g eq k theory}.6. Thus, for inner graded $A$, ${K}_0^{G,\operatorname{Fr}}(A) \xrightarrow{\operatorname{ev}_0^* \circ p^*} {K}_0^{G,\operatorname{Fr}}(A)$ is the multiplication by $2$ map.
\end{rem}
We now consider the composite
\[
\K^G(A_u) \xrightarrow{\iota^*} \K^{G_2}(A) \xrightarrow{j^*} \K^G((A \hat\otimes Cl^{0,2})_u).
\]
For a general unital $A \in \GTWOALG$. the composite does not have an illuminating description, since we do not have a good understanding of the algebra $(A \hat\otimes Cl^{0,2})_u$. However, the situation is better in the following cases
\begin{enumerate}
     \item $A$ is inner graded (and potentially Real).
    \item $A$ is complex (with potentially non-inner grading).
\end{enumerate}
We begin with the first case. So let $A$ be a unital, inner graded algebra in  $\GTWOALG$. Then, we assert that the map of spectra $\KKRex(\C,A) \xrightarrow{j^*} \KKRex(Cl^{1,0},A)$ splits. To see this, we identify $\KKRex(Cl^{1,0},A)$ with $\K^G(A_u)$ via the route:
\begin{equation}\label{eqn: A inner graded over R}
    \KKRex(Cl^{1,0},A) \xrightarrow[\simeq]{- \hat\otimes Cl^{0,1}} \KKRex(Cl^{1,1},A\hat\otimes Cl^{0,1}) \xrightarrow[A~\text{inner graded},\simeq]{(\C \rightarrow Cl^{1,1})^*} \KKRex(\C,A \otimes Cl^{0,1}) \xrightarrow[\simeq]{\res\dashv\coind} \K^G(A_u).
\end{equation}
We want to understand the induced map $\K^{G_2}(A) \rightarrow \K^G(A)$ under the identification above. To that end, we need the following lemma.
\begin{lem}\label{prop: j compose iota A triv graded over R}
    The diagram
\[
\begin{tikzcd}
	{\KKRex(\C,A)} && {\KKRex(Cl^{1,0},A)} \\
	{\mathbb{K}^{G_2}(A)} && {\mathbb{K}^G(A_u)} 
	\arrow["{j^*}", from=1-1, to=1-3]
	\arrow["\simeq", from=1-3, to=2-3]
	\arrow["\simeq", from=2-1, to=1-1]
	\arrow["{\mathsf{res}^{G_2}_G(-)}", from=2-1, to=2-3]
\end{tikzcd}
\]
commutes, where the left vertical equivalence is induced from \Cref{thm: Recovering classical K-theories from our stable inf cat gen}, and the right vertical equivalence is induced by \Cref{eqn: A inner graded over R}.  In particular, for unital $A$, the induced map $K_0^{G_2}(A) \rightarrow K_0^G(A)$ forgets the $C_2$-action on a $G_2$-equivariant projective $A$-module in $K_0^{G_2}(A)$.
\end{lem}
\begin{proof}
    In light of \Cref{lem: another description of j star}, we just need to show that the composite
    \[
    \KKRex(\C,A) \xrightarrow{(A \rightarrow A \otimes Cl^{0,1})_*} \KKRex(\C,A \otimes Cl^{0,1}) \xrightarrow[\cong]{\text{res}\dashv \text{coind}} \ER(\C,A_u)
    \]
    identifies with the restriction map. But this is clear in light of description of the counit of the restriction-coinduction adjunction as mentioned in \Cref{thm: adjunctions at kkg} (the adjunction holds good for non-separable algebras too).
\end{proof}
\begin{prop}
    Let $A$ be an inner graded algebra in $\GTWOALG$. Then, the fiber sequence in \Cref{eqn: dual fiber sequence} splits.
\end{prop}
\begin{proof}
    In light of \Cref{prop: j compose iota A triv graded over R}, it suffices to construct a section of the map $\K^{G_2}(A) \xrightarrow{\mathsf{res}^{G_2}_G(-)} \K^G(A_u)$. If $A$ were trivially graded, this section could have been taken to be $\mathsf{res}^G_{G_2}(-)$. In the general inner graded case, we use the fact that $A \simeq \res^{G}_{G_2}\res^{G_2}_GA$ in $\KKRex$ as follows by an application of \cite[Example E.2.13]{WY} (see also \cite[Lemma 2.1]{BUNKEcrossed}).
\end{proof}
As an immediate consequence of \Cref{prop: j compose iota A triv graded over R}, we have the following corollary.
\begin{cor}\label{cor: composition last ones}
Let $(A,\alpha)$ be a unital, inner graded algebra in $\GTWOALG$.
\begin{enumerate}
    \item The composite $K_0^G(A_u) \xrightarrow{\iota^*} K_0^{G_2}(A) \xrightarrow{j^*} K_0^G(A_u)$ is given by
    \begin{align*}
        K_0^G(A_u) &\rightarrow K_0^G(A_u) \\
        [M] &\mapsto [M \oplus M_{\alpha}]=2[M]
    \end{align*}
    \item The composite $K_0^{G_2}(A) \xrightarrow{j^*} K_0^G(A_u) \xrightarrow{\iota^*} K_0^{G_2}(A)$ is given by
    \begin{align*}
    K_0^{G_2}(A) &\rightarrow K_0^{G_2}(A) \\
    [M,\alpha_M] &\mapsto[M \oplus M_{\alpha},\operatorname{flip}] = [M \oplus M, (m,n) \mapsto (\alpha_M(n),\alpha_M(m))]
    \end{align*}
\end{enumerate}
\end{cor}
\begin{proof}
    Everything follows immediately from \Cref{cor: map induced by C in CL01} and \Cref{prop: j compose iota A triv graded over R}, except that we have to justify the equalities in Assertions 1 and 2. 
    
    For Assertion 1, it suffices to show that for $M$ a $G$-equivariant $A_u$-module over an inner graded $A$, then $M_\alpha \cong M$ as $G$-equivariant $A_u$-modules. If $u$ is the symmetry implementing the grading on $A$, then 
    \[
    M_{\alpha} \rightarrow M, ~~~ \underline{m} \mapsto mu
    \]
    is such an isomorphism.
    
    For Assertion 2, we first show that when $M$ is a $G_2$-equivariant projective $A$-module, then $M \cong M_{\alpha}$ as $G$-equivariant $A_u$-modules. Indeed, an isomorphism is given by 
    \begin{align*}
        M &\rightarrow M_\alpha \\
        m &\mapsto \underline{\alpha_M(m)}
    \end{align*}
    Under this isomorphism, $[M \oplus M_{\alpha},\operatorname{flip}]=[M \oplus M, (m,n) \mapsto (\alpha_M(n),\alpha_M(m))] \in K_0^{G_2}(A)$. Also compare the proof of \Cref{cor: map induced by C in CL01}.\footnote{actually this argument did not use that $A$ is inner graded, see also \Cref{cor: endomorphisms in the complex graded case}.}
\end{proof}

We now tackle the case when $A$ is a (potentially non-inner graded) \emph{complex} $C^*$-algebra. In this case, we recall that there is a preferred isomorphism $Cl^{1,0} \cong Cl^{0,1}$, mapping the canonical generator of $Cl^{1,0}$ to $i$ times the canonical generator of $Cl^{0,1}$. Furthermore, the diagram
\begin{equation}\label{eqn: cliffords}
    \begin{tikzcd}
	{Cl^{1,0}} && {Cl^{1,1}} \\
	{Cl^{0,1}}
	\arrow["{-\hat\otimes Cl^{0,1}}", from=1-1, to=1-3]
	\arrow["\cong"', no head, from=1-1, to=2-1]
	\arrow["{- \hat\otimes Cl^{1,0}}"', from=2-1, to=1-3]
\end{tikzcd}
\end{equation}
commutes up to homotopy, as can be easily constructed via the usual sine-cosine formula. We can therefore identify $\KKRex(Cl^{1,0},A)$ with $\K^G(A_u)$ via the route
\begin{equation}
    \KKRex(Cl^{1,0},A) \simeq \KKRex(Cl^{0,1},A) \simeq \K^G(A_u),
\end{equation}
where the last equivalence stems from the induction-restriction adjunction.
\begin{lem}
    The diagram 
\[
\begin{tikzcd}
	{\hoKKRex(\C,A)} && {\hoKKRex(Cl^{1,0},A)} \\
	&& {\hoKKRex(Cl^{0,1},A)} \\
	{K_0^{G_2}(A)} && {K_0^G(A_u)}
	\arrow["{j^*}", from=1-1, to=1-3]
	\arrow[from=1-1, to=2-3]
	\arrow["\cong"', from=1-3, to=2-3]
	\arrow["\cong", from=3-1, to=1-1]
	\arrow["{\text{forget $C_2$ action}}", from=3-1, to=3-3]
	\arrow["\cong", from=3-3, to=2-3]
\end{tikzcd}
\]
commutes, where the left vertical isomorphism is induced from \Cref{thm: Recovering classical K-theories from our stable inf cat gen}, and the other maps are as described above.\footnote{we use the same notations as in the Real case; this should not create any confusion.}
\end{lem}
\begin{proof}
    The commutativity of the upper right triangle follows from the commutativity up to homotopy of \Cref{eqn: cliffords}. The commutativity of the trapezoid follows from explicitly chasing the definitions, similar to the proof of \Cref{cor: map induced by C in CL01}.
\end{proof}
\begin{cor}\label{cor: endomorphisms in the complex graded case}
    Let $(A,\alpha)$ be a unital, complex $G \times C_2$-equivariant complex $C^*$-algebra.
\begin{enumerate}
    \item The composite $K_0^G(A_u) \xrightarrow{\iota^*} K_0^{G_2}(A) \xrightarrow{j^*} K_0^G(A_u)$ is given by
    \begin{align*}
        K_0^G(A_u) &\rightarrow K_0^G(A_u) \\
        [M] &\mapsto [M \oplus M_{\alpha}]
    \end{align*}
    \item The composite $K_0^{G_2}(A) \xrightarrow{j^*} K_0^G(A_u) \xrightarrow{\iota^*} K_0^{G_2}(A)$ is given by
    \begin{align*}
    K_0^{G_2}(A) &\rightarrow K_0^{G_2}(A) \\
    [M,\alpha_M] &\mapsto[M \oplus M_{\alpha},\operatorname{flip}] = [M \oplus M, (m,n) \mapsto (\alpha_M(n),\alpha_M(m))]
    \end{align*}
\end{enumerate}
\end{cor}
\begin{proof} 
    Exactly the same proof as \Cref{cor: composition last ones}.
\end{proof}
\begin{rem}
    We also mention another description of the map $K_0^{G_2}(A) \rightarrow K_0^{G_2}(A)$ as in \Cref{cor: endomorphisms in the complex graded case}.2. Using the isomorphism
    $
    Cl^{1,0} \cong Cl^{0,1}
    $, we can treat the composition 
    \[
    \hoKKRex(\C,A) \xrightarrow{j^*} \hoKKRex(Cl^{1,0},A) \simeq \hoKKRex(Cl^{0,1},A) \xrightarrow{\iota^*} \hoKKRex(\C,A)
    \]
    as the $\hoKKRex$ product by the image under $\res^{C_2}_{G_2}$ of a certain element in $\hoKKRexT(\C,\C)$. It turns out that this element is\footnote{we omit the details since it is not really needed anywhere} $(1,1)$, whence the induced map must be
    \begin{align*}
       K_0^{G_2}(A) &\rightarrow K_0^{G_2}(A)  \\
       [M,\alpha_M] &\mapsto [M \oplus M^{\operatorname{op}},\alpha_M \oplus (-\alpha_{M})]
    \end{align*}
    And indeed, we have an isomorphism of $G_2$-equivariant $A$-modules
    \begin{align*}
        M \oplus M_\alpha &\rightarrow M \oplus M^{\operatorname{op}}\\
        (m,\underline{n}) &\mapsto (\frac{m+\alpha_M(n)}{2},\frac{m-\alpha_M(n)}{2}).
    \end{align*}
    Thus, we actually have isomorphisms of $G_2$-equivariant $A$-modules
    \[
    (M \oplus M^{\operatorname{op}},\alpha_M \oplus (-\alpha_{M})) \cong (M \oplus M_{\alpha},\operatorname{flip})\cong (M \oplus M, (m,n) \mapsto (\alpha_M(n),\alpha_M(n))).
    \]
\end{rem}
\begin{rem}
    Note that when $A$ is a non-inner graded $C^*$-algebra, the fiber sequence in \Cref{eqn: dual fiber sequence} need not split. A concrete counterexample is $A=\mathbb{C}l^{1}$, since $K_0^{C_2}(\mathbb{C}l^{1}) \cong \Z$ (\Cref{prop: KOC2 CL01}) and $K_0(\mathbb{C}l^{3}_u)\cong K_0(\mathbb{C}l^{1}_u)\cong \Z \oplus \Z$.
\end{rem}
\section{Connections with classical bivariant $K$-theories}
We finally want to show that the homotopy categories of the full subcategories of the graded $\KKRGRA$ categories spanned by separable algebras can be identified with classical graded (Kasparov's) $KK^G$ (\cite{kasparov1981operator}, \cite{KASEQQUIVARIANT})-category or (Connes-Higson's) $E^G$-category (\cite{CH},\cite{HGT},\cite{HG}), according as $!=\operatorname{se}$ or $!=\operatorname{ex}$, respectively. We give a self contained treatment for the case of $E$-theory; see \Cref{last rem} for a discussion on the graded $KK$-theoretic analogue.

As mentioned in the introduction, the results of this section are \emph{not} relevant for our study of graded/equivariant/classical $K$-theory as in Section 5.
\subsection{Recapitulation of classical ungraded $KK^G$-theory}

We begin by recalling the classical definitions of the \emph{ungraded} (Real) $KK^G$-categories. Kasparov in his seminal work defined them in terms of now-called \emph{Kasparov modules}: given $G$-$C^*$-algebras $A$ and $B$, $\KKKASgen(A,B)$ was built out of certain equivalence classes of triples $(\mathcal{E},\varphi,F)$ comprising of a Hilbert $G$-$B$-module $\E$, a equivariant representation $\varphi$ of $A$ on the adjointable operators on $\E$, and an adjointable operator $F$ on $\E$ satisfying certain compactness and equivariance conditions\footnote{Kasparov worked with graded algebras and modules from the get-go.}. It then takes a truckload of work to prove that upon restricting to separable algebras there exists a composition law, now known as the Kasparov product
\[
\KKKAS(A,B) \times \KKKAS(B,C) \rightarrow \KKKAS(A,C)
\]
which is unital and associative, and one concludes that $\KKKAS$ is a category.

Since every equivariant $*$-homomorphism naturally gives a Kasparov module, we have a canonical functor
\[
\kkKAS: \RALGsep \rightarrow \KKKAS.
\]
The question whether the category $\KKKAS$ has a universal property was settled in \cite[Theorem 2.2]{THOMSEN},\cite[Theorem 6.6]{MEYER} (with non-equivariantly the issue being first tackled in \cite[Theorem 4.5]{HIGFIRST}). Note that Meyer also mentions  (as can be readily checked too) that the results also work in the Real case.
\begin{thm}{(Higson-Meyer-Thomsen)}\label{thm: HMT}
    The functor $\kkKAS: \RALGsep \rightarrow \KKKAS$ is initial among functors to an additive $1$-category which are reduced, homotopy invariant, $\kg$-stable and split exact.
\end{thm}
\begin{rem}\label{rem: KKAS properties}
    The following remarks about Kasparov's $KK^G$-theory are in order
    \begin{enumerate}
        \item It is a deep result of Kasparov (\cite[Theorem 2.9]{KASEQQUIVARIANT}) that $\KKKAS$ has countable coproducts, and the functor $\kkKAS$ takes countable direct sums in $\RALGsep$ to countable coproducts in $\KKKAS$. An alternative proof using the Cuntz-Meyer picture of $\KKKAS$ can be given following \cite[Corollary 11.16]{KKandE}. 
        \item One can view the additive $1$-category $\KKKAS$ as a half exact additive category by taking the maximal marking and taking the distinguished half exact sequences to be all short exact sequences in $\RALGsep$ admitting a equivariant cpc-split; the verification that this indeed is a half exact structure follows from \cite[Section 7]{BS} (see also \cite[Theorem 3.28]{ECH} and the discussion above it).
        \item $\KKKAS$ is a symmetric monoidal category and the functor $\kkKAS: \RALGsep \rightarrow \KKKAS$ has a symmetric monoidal refinement for both the minimal and the maximal tensor products on $\RALGsep$; this is explicitly stated in \cite[Theorem 4.a]{CuntzBiv}
    \end{enumerate}
\end{rem}
Using \Cref{thm: HMT}, we can prove the comparison result between $\KKKAS$ and $\hoKKse$.
\begin{thm}\label{thm: hoKK inf is KKKas}
    We have a natural equivalence of additive, symmetric monoidal (for both symmetric monoidal refinements) categories
    \[
    \hoKKse \simeq \KKKAS.
    \]
\end{thm}
\begin{proof}
        By construction, the canonical functor 
    \[
    \hokkse: \RALGsep \rightarrow \hoKKse
    \]
    is reduced, homotopy invariant, $\kg$-stable, and additive. Thus, \Cref{thm: HMT} says that there exists a unique additive functor $\mathcal{F}: \KKKAS \rightarrow \hoKKse$ such that $\mathcal{F} \circ \kkKAS = \hokkse$.

On the other hand, the canonical functor \[
\kkKAS: \RALGsep \rightarrow \KKKAS
\]
is clearly reduced, homotopy invariant and $\kg$-stable, as can be read off from \Cref{thm: HMT}. That it also satisfies the other hypothesis asked for in \Cref{prop: uni prop hoKK one categorical} is the content of \Cref{rem: KKAS properties}. We thus obtain a unique functor $\mathcal{G}: \hoKKse \rightarrow \KKKAS$ such that $\mathcal{G} \circ \hokkse = \KKKAS$. $\mathcal{G}$ is furthermore symmetric monoidal, as follows from the proof of \Cref{prop: uni prop hoKK one categorical} and \cite[Proposition 3.3.2]{HIN}. Our thesis now follows from a standard yoga with universal functors.
\end{proof}
\begin{rem}
    We implicitly used in the above proof the following basic facts: If $\mathcal{C},\mathcal{D}$ are additive (resp. symmetric monoidal) categories, and $F: \mathcal{C} \rightarrow \mathcal{D}$ is an additive (resp. symmetric monoidal) functor which is an equivalence of the underlying categories, then the inverse is automatically additive (resp. symmetric monoidal).
\end{rem}
We finally compare $\hoKKsegen$ with the classical non-separable extension of $\KKKAS$, as described in \cite[Definition 5.5]{KKSKANDALIS} and \cite[Section 2.1]{AZZA}.
\begin{prop}\label{cor: KK large same}
    For $A \in \RALGsep$ and $B \in \RALG$, we have a natural isomorphism of abelian groups
    \[
    \KKKASgen(A,B) \cong \hoKKsegen(A,B).
    \]
\end{prop}

\begin{proof}
    In light of \Cref{prop: ind omnibus}, we have for $A$ and $B$ as in the statement of this proposition, an equivalence (with notation as in \Cref{defi: s finitary}) of spectra
    \begin{align*}
        \operatorname{colim}_{B' \subseteq_{\operatorname{sep}} B}
        \KKse\bigl(\kkse(A),\kkse(B')\bigr)
        \xrightarrow{\simeq}
        \KKse\bigl(\kkse(A),\kkse(B)\bigr)
    \end{align*}
    Since homotopy groups of spectra commute with filtered colimits, our thesis follows from the classical result that
    \[
    \operatorname{colim}_{B' \subseteq_{\operatorname{sep}} B}
    \KKKASgen(A,B')
    \xrightarrow{\cong}
    \KKKASgen(A,B),
    \]
    see for example \cite[Lemma 2.1]{AZZA}.
\end{proof}
\begin{rem}\label{rem: KK SKA}
    When the $C^*$-algebra $A$ in the first variable is not separable, $\hoKKsegen(A,-)$ need not coincide with $\KKKASgen(A,-)$ a'la Skandalis; it boils down to the issue that homotopy groups of homotopy limits need not be limits of homotopy groups.
\end{rem}
\subsection{Recapitulation of classical ungraded $E^G$-theory (with some details) }
Just as in the last subsection, we begin by recalling the classical definitions of the \textit{ungraded} (Real) $E^G$-categories. 

Classically, ungraded $E^G$-theory was defined in terms of \textit{asymptotic morphisms} by Connes-Higson in \cite{CH}. The foundational works on (equivariant) $E$-theory (\cite{CH}, \cite{HGT}, \cite{HIGFRAC}) did not discuss the Real structure, and the author (of this article) is only aware of Boersema's work \cite{BOERSEMA} on the same \footnote{which was pointed out to him by Yuezhao Li}. However, loc.\! cit.\! does not discuss equivariant matters. Note that although \cite{BOERSEMA} is written with real $C^*$-algebras, the modifications to the Real case are obvious.

In the remainder of this subsection, we will therefore be making a comparison of $\hoKKex$ with a equivariant version of Boersema's (Real) $E$-category. The complex situation can be dealt similarly, and with lesser difficulties. Note that we omit a discussion of non-separable $E$-theory (the analogues of \Cref{cor: KK large same} and \Cref{rem: KK SKA} hold in this context too). Therefore, we use the notation $\EBGHT$ to always denote the \emph{separable} theory.

We begin with the key definitions.
\begin{defi}\label{defi: e class preparation}
    \hfill
    \begin{enumerate}
        \item 	Let $A,B \in \RALGsep$. An asymptotic morphism
from $A$ to $B$ is a one-parameter family of functions $\varphi_t : A \to B$ with $t \in
[1, \infty)$, satisfying the continuity condition that for all $a \in
A$, the map
\[
t \longmapsto \varphi_t(a) : [1, \infty) \to B
\]
 is bounded and
continuous, and the asymptotic conditions that for all $a, a_1, a_2 \in A$ and
$\lambda \in \mathbb{C}$ 
\[
\begin{rcases}
	\varphi_t(a_1a_2) -  \varphi_t(a_1)\varphi_t(a_2)\\
	\varphi_t(a_1+a_2)- \varphi_t(a_1) - \varphi_t(a_2)\\
	\varphi_t(\lambda a ) - \lambda \varphi_t(a)\\
	\varphi_t(a^{*}) - \varphi_t(a)^{*}\\
    \varphi_t(g.a)-g.\varphi_t(a)
\end{rcases}
\to 0, \qquad \text{as }t \to \infty
\]
When $A$ and $B$ are Real, we shall require that in addition
\[
\tau(\varphi_t(a)) - \varphi_t(\tau(a)) \to 0 \qquad \text{as }t
\to \infty.
\] 
We shall denote an
asymptotic morphism $\varphi : A
\dashrightarrow B$ with  a dashed arrow.
\item There is a natural notion of homotopy between asymptotic morphisms which is an equivalence relation (see \cite[Chapter 2]{HGT}), and we denote by $[[A,B]]_G$ the $G$-equivariant asymptotic homotopy classes of maps\footnote{as in the ordinary homotopy case, if the external group $G$ is clear from the context, we omit it from the notation.} from $A$ to $B$. This forms a category $\ASYG$, whose objects are separable $G$-$C^*$-algebras, and morphisms are $G$-equivariant asymptotic homotopy classes of maps.
\item Just as the passage from $\RALGseph$ to $\SKGRALGseph$ (see \Cref{prop: SKGRALKGseph omnibus}.1), one can form the category $\SKGASYG$ via left and right Bousfield localizations of $\ASYG$, whose hom monoids would be $\kg$-stabilized asymptotic homotopy classes of maps.

Note that $\ASYG$ and $\SKGASYG$ are $1$-categories; an $\infty$-categorical refinement of these categories has been studied in \cite{EisCASS} when $G$ is a discrete countable group. The methods of loc.\! cit.\! would not generalize to general locally compact second countable groups. We anyway work only with the $1$-categorical variants in this section.
\item The classical $E^G$-groups are given by
\[
\EBGHT(A,B):=\operatorname{hom}_{\SKGASYG}(\Omega A,\Omega B)=[[\Omega A \otimes \kg,\Omega B \otimes \kg]],
\]
and yield a symmetric monoidal additive category structure on $\RALGsep$ with respect to the maximal tensor product on $\RALGsep$. That it indeed becomes a symmetric monoidal additive category can be seen by combining the approaches and relevant results of \cite{HGT} and \cite{BOERSEMA}. There is an obvious functor
\[
\eBGHT: \RALGsep \rightarrow \EBGHT.
\]
\item One can view the additive $1$-category $\EBGHT$ as a half exact additive category by taking the maximal marking and taking the distinguished half exact sequences to be all short exact sequences in $\RALGsep$; see \Cref{lem: folklore lem}.3 for details. \emph{In the remainder of this subsection, we by default work with this half exact structure.} 
    \end{enumerate}
\end{defi}
For later use, we mention a few technical results about asymptotic morphisms now. In the non-equivariant case, some of these results are also discussed in the PhD thesis of Andreas Thom (\cite[Section 2.2]{ThomPhD}) with a little more elaborations.
\begin{lem}\label{lem: technical asy}
    Let $A,B \in \RALGsep$. Denote by $\widetilde{[[A,B]]}$ the set of homotopy classes of asymptotic morphisms $\varphi: A \dashrightarrow B$ which have the property that $\varphi_1(a)=0 ~\forall~ a \in A$. Homotopies are also required to satisfy this condition. Then the natural map
    \[
    \Phi_{A,B}:\widetilde{[[A,B]]} \rightarrow [[A,B]]
    \]
    is an isomorphism.
\end{lem}
\begin{proof}
    We just write a map in the other direction; that it is well-defined and inverse to $\Phi_{A,B}$ is a routine check.

    Fix a function $f:[1,\infty) \rightarrow [0,1]$ with the property that $f(1)=0, f(t)=1~\forall~t \geq 2$. Then the pointwise multiplication of $f$ with an element of $[[A,B]]$ does the job.
\end{proof}
The above lemma has the following interesting consequence. For $B \in \RALGsep$, denote by $C_{0b}([1,\infty),B)$ the (non-separable) $G$-$C^*$-algebra of $G$-continuous, continuous, bounded functions $f:[1,\infty) \rightarrow B$ which further satisfy $f(1)=0$. \Cref{lem: technical asy} shows that the datum of an (equivalence class) of an asymptotic morphism $A \dashrightarrow B$ is the same datum as a $G$-equivariant homomorphism $A \rightarrow \frac{C_{0b}([1,\infty),B)}{C_0((1,\infty),B)}$. Thus, we can naturally associate to an asymptotic morphism $\varphi: A \dashrightarrow B$ the extension given in the top row of the diagram
\begin{equation}\label{eqn: E phi of thomsen}
    \begin{tikzcd}
	0 & {C_0((1,\infty),B)} & {P_{\varphi}} & A & 0 \\
	0 & {C_0((1,\infty),B)} & {C_{0b}([1,\infty),B)} & {\frac{C_{0b}([1,\infty),B)}{C_0((1,\infty),B)}} & 0
	\arrow[from=1-1, to=1-2]
	\arrow[from=1-2, to=1-3]
	\arrow["{=}", from=1-2, to=2-2]
	\arrow[from=1-3, to=1-4]
	\arrow[from=1-3, to=2-3]
	\arrow[from=1-4, to=1-5]
	\arrow["\varphi", from=1-4, to=2-4]
	\arrow[from=2-1, to=2-2]
	\arrow[from=2-2, to=2-3]
	\arrow[from=2-3, to=2-4]
	\arrow[from=2-4, to=2-5]
\end{tikzcd}
\end{equation}
where the right square is a pullback. Note that $P_\varphi \in \RALGsep$ courtesy both $C_0((1,\infty),B)$ and $A$ being separable.
\begin{lem}\label{lem: HLT lemma}
    In the notation of the above discussion, let $\psi: \Omega A \dashrightarrow \Omega B$ be the asymptotic morphism obtained by applying the equivariant (Real) Connes-Higson construction (\cite[Proposition 5.5]{HGT}, \cite[Proposition 4.3]{BOERSEMA}) to the top extension in \Cref{eqn: E phi of thomsen}. Then $[\psi]=-[\Omega \varphi]$ in $[[\Omega A,\Omega B]]$.
\end{lem}
\begin{proof}
    This is the content of \cite[Lemma 1.3]{ThomsenLarsen}, see also \cite[Lemma 3.4]{THOASY}. Note that there is a sign error in the analogous assertion \cite[Proposition 5.12]{HGT}.
\end{proof}
Our first important result of this subsection is to prove the following theorem, which is folklore even in the complex case: to the best of the author's knowledge, the classical literature on $E$-theory just \emph{state} the result without a proof; see also \cite[Remark 3.33]{EisCASS}.

\begin{thm}{(folklore)}\label{thm: folklore}
    The functor $\eBGHT: \RALGsep \rightarrow \EBGHT$ is initial among functors to an additive $1$-category which are reduced, homotopy invariant, $\kg$-stable and half exact.
\end{thm}
A (detailed) proof of the above result will take us quite some time. We begin with some preparations. The first check one needs to make to prove \Cref{thm: folklore} is that $\eBGHT: \RALGsep \rightarrow \EBGHT$ is homotopy invariant, $\kg$-stable and exact. While the first two are obvious from the definition, half exactness is not so. The first crucial step to that end is demonstrating for each $A \in \RALGsep$, the functor
\[
\EBGHT(A,-): \RALGsep\rightarrow \operatorname{Ab}
\]
is \emph{split} exact. This follows directly from the equivariant, Real analogue of \cite[Corollary 25.5.7]{BLAK}. However, it can also be seen as a consequence of the following stronger result, which is also going to come handy in the next subsection. The complex, non-equivariant case of the result is due to Dadarlat-Loring (\cite[Proposition 3.2]{DL}).
\begin{prop}\label{prop: asy split exact functor}
    For each $D \in \RALGsep$, the functor $[[D,-]]: \RALGsep \rightarrow \operatorname{Set}_*$ is split exact.
\end{prop}
\begin{proof}
    Let $0 \rightarrow A \rightarrow B \xrightarrow{f} C \rightarrow0$ be a short exact sequence in $\RALGsep$ which is split by a homomorphism $\sigma: C \rightarrow B$. Without loss of generality, we can assume that the map $A \rightarrow B$ is the inclusion of an ideal. From the standard excision properties in the asymptotic category and Puppe sequences in $\RALGsep$,(\cite[Proposition 5.16]{HGT}, \cite[Proposition 2.5]{HG}), we get a long exact sequence of pointed sets
    \begin{equation}\label{eqn: les asy}
        ...\rightarrow [[D,\Omega B]] \xrightarrow{\Omega f_*}[[D,\Omega C]] \xrightarrow{\partial}[[D,C(f)]] \rightarrow [[D,B]] \xrightarrow{f_*} [[D,C]] \rightarrow 0
    \end{equation}
    where $C(f)$ denotes the mapping cone of $f$.
    Since $f$ is split by $\sigma$, $\Omega f$ is split by $\Omega \sigma$, and we can extract the split exact sequence
    \[
    0 \rightarrow [[D,C(f)]] \rightarrow [[D,B]] \xrightarrow{f_*} [[D,C]] \rightarrow 0
    \]
    from \Cref{eqn: les asy}. To prove our thesis, it therefore suffices to show that the inclusion $A \rightarrow C(f)$ is an isomorphism in $\ASYG$.
    
    Using the split $\sigma$, we can re-write $C(f)$ as
    \[
    C(f)=\{(a+\sigma(\gamma(0)),\gamma)\mid a \in A\} \subseteq B \oplus PC.
    \]
    Let $u_t$ be a (Real) quasicentral approximate unit in $A$ (\cite[Lemma 5.3]{HGT}). In light of \cite[Lemma 5.6.(iii)]{HGT}, we can now define a $1$-parameter family of homomorphisms
    \begin{align*}
        \varphi_t: C_0(0,1] &\rightarrow A \\
        f &\mapsto f(u_t)
    \end{align*}
    given by functional calculi of the individual $u_t$. We can now define an asymptotic morphism given by the family of maps
    \begin{align*}
        \rho_t:C_0(0,1] \otimes C &\rightarrow A \\
        f \otimes c &\mapsto f(u_t)\sigma(c)
    \end{align*}
     Note that in defining it, we used the universal property of the maximal tensor product, see \cite[Section 25.2.3]{BLAK}. We can now define an asymptotic morphism $\varphi: C(f) \dashrightarrow A$ given by 
    \[
    \varphi_t(a+\sigma(\gamma(0)),\gamma)=a+\rho_t(f).
    \]
    The verification that it indeed is an inverse to the inclusion $A \rightarrow C(f)$ in $\ASYG$ can now be taken from the proof of \cite[Proposition 3.2]{DL}. The equivariance and Real structure do not pose any additional difficulties.
\end{proof}
As the reader may guess, the reason we care about the split exactness result is to bring $\KKKAS$ (or equivalently, $\hoKKRsep$) into the game. From results of the previous subsection and playing around a bit with the uniqueness of universal functors (\cite[Theorem 3.2]{BOERSEMA}), we conclude that there is a unique functor $c: \hoKKRsep \rightarrow \EBGHT$ such that $c \circ \hokkRsep = \eBGHT$ \footnote{c stands for comparison functor}, which is furthermore symmetric monoidal with respect to the maximal tensor product.

Since we already have knowledge of periodicity in $\KKKAS$ (equivalently, in $\hoKKRsep$), we can now get periodicity results in $\EBGHT$ too. All this (and a little more) is the content of the next lemma.
\begin{lem}\label{lem: folklore lem}
    \hfill 
    \begin{enumerate}
        \item Let $F:\RALGsep \rightarrow \mathcal{D}$ be a functor to a half exact additive $1$-category which is homotopy invariant, $\kg$-stable, half exact with respect to exact sequences. Then there is an isomorphism \[F(C_0(\R^{2,2}) \otimes A) \xrightarrow[\cong]{\operatorname{Bott}^F} F(A)\] which is natural with respect to morphisms in $\RALGsep$ and also with respect to the boundary maps of the long exact sequence associated to $F$, as described in \cite[Theorem 21.4.4]{BLAK}.
        \item The natural map
        \[
        C_0(\R^{0,1}) \otimes -: \EBGHT(A,B) \rightarrow \EBGHT(\Omega A, \Omega B)
        \]
        is an isomorphism.
        \item The functor $\eBGHT: \RALGsep \rightarrow \EBGHT$ is homotopy invariant, $\kg$-stable, and half exact.
        \item The natural map
        \[
        C_0(\R^{1,0}) \otimes -: \EBGHT(A,B) \rightarrow \EBGHT(C_0(\R^{1,0}) \otimes A, C_0(\R^{1,0}) \otimes B)
        \]
        is an isomorphism.
        \item The isomorphisms of Assertion 1 applied to the functors $\EBGHT(A,-)$ and $\EBGHT(-,C_0(\R^{2,2}) \otimes B)$ fit into a commutative diagram
\[\begin{tikzcd}
	&& {E^G_{\operatorname{class}}(C_0(\R^{2,2}) \otimes A, C_0(\R^{2,2}) \otimes B)} \\
	{E^G_{\operatorname{class}}(A,B)} && {E^G_{\operatorname{class}}(A, C_0(\R^{2,2}) \otimes B)}
	\arrow["\cong", no head, from=1-3, to=2-3]
	\arrow["{C_0(\R^{2,2}) \otimes -}"{pos=0.4}, no head, from=2-1, to=1-3]
	\arrow["\cong"', no head, from=2-1, to=2-3]
\end{tikzcd}\]
    \end{enumerate}
\end{lem}
\begin{proof}
We begin with Assertion 1. In light of the given properties of $F$, the Real analog of the proof of \cite[Proposition 6.16]{HGT} shows that there exists an asymptotic morphism \[\alpha:C_0(\R^{1,1}) \dashrightarrow \mathcal{K}(l^2\N)\] 
such that for each $A \in \RALGsep$, the asymptotic morphism
\[
\alpha \otimes 1_A:C_0(\R^{1,1}) \otimes A\dashrightarrow \mathcal{K}(l^2\N) \otimes A
\]
yields a natural isomorphism
\[
F(C_0(\R^{1,1}) \otimes A) \xrightarrow[\cong]{\operatorname{Bott}_A^F} F(A).
\]
Our sought after natural map is now obtained by applying the above isomorphism twice
\[
F(C_0(\R^{2,2}) \otimes A) \xrightarrow[\cong]{\operatorname{Bott}_{C_0(\R^{1,1})\otimes A}^F} F(C_0(\R^{1,1}) \otimes A) \xrightarrow[\cong]{\operatorname{Bott}_{A}^F} F(A).
\]\\
  The map in Assertion 2 being an isomorphism can be readily taken from \cite[Lemma 4.5]{BOERSEMA}, with the obvious modifications to tackle the equivariant case. \\\\
    We now prove Assertion 3 (note that it settles the easy part of \Cref{thm: folklore}). That $\eBGHT$ is homotopy invariant and $\kg$-stable is obvious. That it is also half exact follows by combining \cite[Theorem 4.4]{BOERSEMA} and the excision theory of equivariant $E$-theory developed in chapters 5 and 6 of \cite{HGT}. The careful reader will observe that (variants of) Assertions 1 and 2 go as an input to proving the exactness in the first variable in \cite{BOERSEMA} and \cite{HGT}.\\\\
    For Assertion 4, we can accommodate the map in the commutative diagram
\[\begin{tikzcd}
	{E^{G}_{\operatorname{class}}(A,B)} && {E^{G}_{\operatorname{class}}(C_0(\R^{1,1})\otimes A,C_0(\R^{1,1})\otimes B)} \\
	& {E^{G}_{\operatorname{class}}(C_0(\R^{1,0})\otimes A,C_0(\R^{1,0})\otimes B)}
	\arrow["{C_0(\R^{1,1}) \otimes -}", from=1-1, to=1-3]
	\arrow["{C_0(\R^{1,0}) \otimes -}"'{pos=0.2}, from=1-1, to=2-2]
	\arrow["{C_0(\R^{0,1}) \otimes -}"'{pos=0.8}, from=2-2, to=1-3]
\end{tikzcd}\]
The discussion preceding the lemma shows that the horizontal arrow is an isomorphism, since \[C_0(\R^{1,1}) \simeq_{\KKKAS} \mathbb{C}.\]The right arrow is an isomorphism from Assertion 2.\\\\
Finally we show Assertion 5. In light of Assertion 1, it essentially follows from the commutativity of the diagram 
\[\begin{tikzcd}
	{C_0(\R^{1,1}) \otimes A} && {\mathcal{K}(l^2\N) \otimes A} \\
	{C_0(\R^{1,1}) \otimes B} && {\mathcal{K}(l^2\N) \otimes B}
	\arrow["{\alpha \otimes 1_A}", dashed, from=1-1, to=1-3]
	\arrow["{1_{C_0(\R^{1,1})} \otimes f}"', dashed, from=1-1, to=2-1]
	\arrow["{1_{\mathcal{K}(l^2\N)} \otimes f}", dashed, from=1-3, to=2-3]
	\arrow["{\alpha \otimes 1_B}", dashed, from=2-1, to=2-3]
\end{tikzcd}\]
for every asympototic morphism $f: A \dashrightarrow B$, as can be seen from the symmetric monoidal properties of $\ASYG$. See also \cite[Theorem 4.6]{HGT} and the proof of \cite[Proposition 6.16]{HGT}.
\end{proof}
Before continuing with the proof of \Cref{thm: folklore}, we mention a corollary to some results of the last lemma, which will come handy in the next section.
\begin{cor}\label{cor: Bott inverted}
    There exists an asymptotic morphism $\alpha:C_0(\R^{1,1})\dashrightarrow \mathcal{K}(l^2)$ such that \[\alpha \otimes 1_B: C_0(\R^{1,1}) \otimes B \dashrightarrow \mathcal{K}(l^2) \otimes B\] determines an invertible class in $\EBGHT(C_0(\R^{1,1}) \otimes B,B)$. Furthermore, the functor 
    \[
    \Omega: \SKGASYG \rightarrow \EBGHT
    \]
    exhibits $\EBGHT$ as the $1$-categorical localization of $\SKGASYG$ at the classes of maps $\alpha \otimes 1_B$.
\end{cor}
\begin{proof}
    The first statement follows readily from the proof of Assertion 1 of \Cref{lem: folklore lem} by taking $\mathcal{D}=\EBGHT$. The second statement follows from the following general fact about $(\infty-)$categories (it is the dual version of \cite[Proposition 5.2.7.4]{HTT}) : Let $\mathcal{C}$ be a category with an endofunctor $R: \mathcal{C} \rightarrow \mathcal{C}$ and a natural transformation $\mathfrak{T}: R \implies \operatorname{id}_{\mathcal{C}}$. If for every object $C \in \mathcal{C}$ the morphims
    \[
    \mathfrak{T}_{R(C)},R(\mathfrak{T}_C): R(R(C)) \rightarrow R(C)
    \]
    are equivalences, then $R:\mathcal{C} \rightarrow R\mathcal{C}$ is the right-adjoint of a right Bousfield localization with counit $\mathfrak{T}$. It is also the Dwyer-Kan localization at the set of morphisms $W_R:=\{\mathfrak{T}_C \mid C \in \mathcal{C}\}$. 

    Indeed, in the above notation, take $\mathcal{C}:=\SKGASYG$, and $R:= C_0(\R^{1,1}) \otimes - :\SKGASYG \rightarrow \SKGASYG$. The rest follows from the results of \Cref{lem: folklore lem} and the definition of $\EBGHT$.
\end{proof}
The heart of the proof of the uniqueness assertion in \Cref{thm: folklore} is contained in the following result.
\begin{prop}{(cf. \cite[Theorem 4.6]{THOASY})}\label{prop: thomsen heart}
    Let $F$ be a homotopy invariant, $\kg$-stable and half exact functor from $\RALGsep$ to $\operatorname{Ab}$. Fix $A$ in $\RALGsep$ and consider an arbitrary element $d \in F(A)$. Then there is a unique natural transformation 
    \[
    T: \EBGHT(A,-) \implies F(-)
    \]
    such that $T_A([\operatorname{id}_A])=d$.
\end{prop}
\begin{proof}
    Firstly, we massage the hom groups of $\EBGHT$ a bit, and write them as
    \[
    \EBGHT(A,B)=[[C_0(\R^{2,1}) \otimes A \otimes \kg,  C_0(\R^{2,1}) \otimes B \otimes \kg]].
    \]
    Indeed this can be done from Assertions 2 and 4 of \Cref{lem: folklore lem} (the reason we do this maneuver will be clear from the upcoming proof, but see also \Cref{rem: why not one zero}).

    \textit{To avoid notational clutter, we subsequently denote $C_0(\R^{p,q}) \otimes X$ as $\R^{p,q}X$ for any $X \in \RALGsep$. Same comment applies to functors involving these algebras.}
    
    Now consider $x \in [[\R^{2,1} A \otimes \kg,  \R^{2,1}B \otimes \kg]]$. Let it be represented by the equivariant asymptotic morphism 
    \[
    \varphi: \R^{2,1}A \otimes \kg \dashrightarrow \R^{2,1}B \otimes \kg.
    \]
     As in \Cref{eqn: E phi of thomsen}, this yields an extension
    \begin{equation}\label{eqn: thomsen proof}
        0 \rightarrow \R^{2,2}B \otimes \kg \rightarrow P_\varphi \rightarrow \R^{2,1}A \otimes \kg \rightarrow 0.
    \end{equation}
    Consequently, we obtain a morphism $F(\R^{2,2}A \otimes \kg ) \rightarrow F(\R^{2,2}B \otimes \kg )$ in $\operatorname{Ab}$ via the connecting map of the long exact sequence associated to $F$ as in \cite[Theorem 21.4.4]{BLAK}. The $\kg$-stability of $F$ together with Assertion 1 of \Cref{lem: folklore lem} naturally produces a morphism $T_x: F(A) \rightarrow F(B)$. Now define
    \begin{align*}
        T_B: [[\R^{2,1}A \otimes \kg,  \R^{2,1}B \otimes \kg]] \rightarrow F(B) \\
        x \mapsto T_x(d).
    \end{align*}
    In light of the machinery we developed in this subsection and functoriality properties of the Connes-Higson construction, the well-definedness and naturality of $T_B$, as well as the verification $T_A([\operatorname{id}_A])=d$ is a routine check, as has been done in the proof of \cite[Theorem 4.6]{THOASY}.

    We now prove the uniqueness part. Let 
    \[
    T':[[\R^{2,1}A \otimes \kg,  \R^{2,1}- \otimes \kg]] \rightarrow F(-)
    \]
    be a natural transformation such that $T'([\operatorname{id}_A])=d$. Consider $x =[\varphi] \in [[\R^{2,1}A \otimes \kg,  \R^{2,1}B \otimes \kg]]$. Let $\partial_1,\partial_2,\partial_3,\partial_4$ denote the connecting maps obtained by applying the functors \[[[\R^{4,3}A \otimes \kg, \R^{2,1}- ]], [[\R^{2,1}A\otimes \kg, \R^{2,1}-]], F(\R^{2,2} -), F(-)\] respectively to the extension in \Cref{eqn: thomsen proof}\footnote{denoting by $\mathcal{G}$ any one of these four functors, we are implicitly using the natural identification of $\mathcal{G}(-)$ with $\mathcal{G}(\R^{2,2}\otimes \kg \otimes -)$ in what follows.}. We get a commutative diagram
\begin{equation}\label{eqn: thomsen dia}
\begin{tikzcd}
	&& {E^G_{\operatorname{class}}(\R^{2,2}A,\R^{2,2}A)} & {E^G_{\operatorname{class}}(\R^{2,2}A,\R^{2,2}B)} && \\
	\\
	{E^G_{\operatorname{class}}(A,A)} && {E^G_{\operatorname{class}}(A,\R^{2,2}A)} & {E^G_{\operatorname{class}}(A, \R^{2,2}B)} && {E^G_{\operatorname{class}}(A,B)} \\
	&& {F(\R^{2,2}A)} & {F(\R^{2,2}B)} \\
	&& {F(A)} & {F(B)}
	\arrow["{\partial_1}", color={rgb,255:red,92;green,92;blue,214}, from=1-3, to=1-4]
	\arrow["{\operatorname{Bott}^{E^G_{\operatorname{class}}(-,\R^{2,2}A)}}"{description}, no head, from=1-3, to=3-3]
	\arrow["{\operatorname{Bott}^{E^G_{\operatorname{class}}(-,\R^{2,2}B)}}"{description}, no head, from=1-4, to=3-4]
	\arrow["{\R^{2,2}-}", shift left, color={rgb,255:red,92;green,92;blue,214}, from=3-1, to=1-3]
	\arrow["{\operatorname{Bott}^{E^G_{\operatorname{class}}(A,-)}}"', no head, from=3-1, to=3-3]
	\arrow["{T'_A}"', from=3-1, to=5-3]
	\arrow["{\partial_2}", from=3-3, to=3-4]
	\arrow["{T'_{\R^{2,2}A}}", from=3-3, to=4-3]
	\arrow["{\operatorname{Bott}^{E^G_{\operatorname{class}}(A,-)}}"', no head, from=3-4, to=3-6]
	\arrow["{T'_{\R^{2,2}B}}", from=3-4, to=4-4]
	\arrow["{\R^{2,2}-}"', shift right, color={rgb,255:red,92;green,92;blue,214}, from=3-6, to=1-4]
	\arrow["{T'_B}", from=3-6, to=5-4]
	\arrow["{\partial_3}", from=4-3, to=4-4]
	\arrow["{\operatorname{Bott}^F}", no head, from=4-3, to=5-3]
	\arrow["{\operatorname{Bott}^F}", no head, from=4-4, to=5-4]
	\arrow["{\partial_4}", from=5-3, to=5-4]
\end{tikzcd}
\end{equation}
where the undirected edges are induced by the maps described in \Cref{lem: folklore lem}.1. Note that the fact that the diagram indeed commutes needs a lot of the machinery we developed earlier in this subsection.

We now claim the following.
\begin{clm}
    The composite 
    \[
    \partial_1 \circ (\R^{2,2}-) : \EBGHT(A,A) \rightarrow \EBGHT(\R^{2,2}A, \R^{2,2}B)
    \]
    equals the composite
    \[
    (\R^{2,2}-) \circ \varphi_* :\EBGHT(A,A) \rightarrow \EBGHT(\R^{2,2}A, \R^{2,2}B).
    \]
\end{clm}
\begin{proof}
    We build the diagram
\[\begin{tikzcd}
	{E^G_{\operatorname{class}}(\R^{2,2}A,\R^{2,2}A)} &&&& {E^G_{\operatorname{class}}(\R^{2,2}A,\R^{2,2}B)} \\
	{E^G_{\operatorname{class}}(\R^{2,1}A,\R^{2,1}A)} &&&& {E^G_{\operatorname{class}}(\R^{2,1}A,\R^{2,1}B)} \\
	{E^G_{\operatorname{class}}(\R^{2,0}A,\R^{2,0}A)} &&&& {E^G_{\operatorname{class}}(\R^{2,0}A,\R^{2,0}B)} \\
	{E^G_{\operatorname{class}}(A,A)} &&&& {E^G_{\operatorname{class}}(A,B)}
	\arrow["{\partial_1}", from=1-1, to=1-5]
	\arrow["{\R^{0,1}-}"', from=2-1, to=1-1]
	\arrow["{-\partial_{[[\R^{4,2}A \otimes \mathcal{K}_G, \R^{2,0}- \otimes \mathcal{K}_G]]}}", from=2-1, to=2-5]
	\arrow["{\R^{0,1}-}", from=2-5, to=1-5]
	\arrow["{\R^{0,1}-}"', from=3-1, to=2-1]
	\arrow["{(\R^{2,0}\varphi)_*}", from=3-1, to=3-5]
	\arrow["{\R^{0,1}-}", from=3-5, to=2-5]
	\arrow["{\R^{2,0}-}"', from=4-1, to=3-1]
	\arrow["{\varphi_*}", from=4-1, to=4-5]
	\arrow["{\R^{2,0}-}", from=4-5, to=3-5]
\end{tikzcd}\]
We need to show that it commutes. The bottom rectangle commutes by the symmetric monoidal properties of $\EBGHT$. The middle rectangle commutes by an application of \Cref{lem: HLT lemma}, together with the fact that composition with the Connes-Higson morphism yields the boundary map in $\EBGHT$, as described in \cite[Theorem 6.15]{HGT}. The top square commuting is a consequence of the general fact that for any $A',B',C' \in \RALGsep$ and $\psi \in \EBGHT(B',C')$, the diagram
\[\begin{tikzcd}
	{E^G_{\operatorname{class}}(\Omega A',\Omega B')} && {E^G_{\operatorname{class}}(\Omega B',\Omega C')} \\
	{E^G_{\operatorname{class}}(A',B')} && {E^G_{\operatorname{class}}(B',C')}
	\arrow["{-\Omega\psi_*}", from=1-1, to=1-3]
	\arrow["{\Omega_*}"', from=2-1, to=1-1]
	\arrow["{\psi_*}", from=2-1, to=2-3]
	\arrow["{\Omega_*}", from=2-3, to=1-3]
\end{tikzcd}\]
commutes. This boils down to the fact that the map $S^1 \wedge S^1 \xrightarrow{\operatorname{flip}} S^1 \wedge S^1$ has degree $-1$.

\end{proof}
The above claim therefore shows that the composite of the purple arrows (we are reversing the direction of the last arrow so as to compose) maps $[\operatorname{id}_A] \in \EBGHT(A,A)$ to $[\varphi] \in \EBGHT(A,B)$. The commutativity of \Cref{eqn: thomsen dia} now tells that 
\[
\partial_4(T_A'([\operatorname{id}_A]))= T_B'([\varphi]).
\]
Since $T_A'([\operatorname{id}_A])$ is given to be $d$, the proof of the proposition follows.
\end{proof}
\begin{rem}\label{rem: why not one zero}
    In the above proof, $(2,2)$-periodicity played a crucial role. It would indeed be natural to expect that the same could be done by the more standard $(1,1)$-periodicity. Unraveling the proof of \Cref{prop: thomsen heart}, one would then want to take 
    \[
    \EBGHT(A,B)=[[\R^{1,0}A \otimes \kg, \R^{1,0}B \otimes \kg]].
    \]
    But it is not clear why this definition coincides with $[[\R^{0,1}A \otimes \kg, \R^{0,1}B \otimes \kg]]$: the results of \Cref{lem: folklore lem} do not say this! Tracing further back, the reader may wonder ``why not \emph{define} $\EBGHT(A,B):=[[\R^{1,0}A \otimes \kg, \R^{1,0}B \otimes \kg]]$?" The reason is that then it would not be clear why this category has excision; the development of the excision results in $E$-theory and the Connes-Higson morphism depend on Puppe sequences, where the loops are $C_0(\R^{0,1})$ and not $C_0(\R^{1,0})$.
\end{rem}
\begin{proof}[Proof of \Cref{thm: folklore}]
    With \Cref{prop: thomsen heart} at our disposal, the argument now follows by  standard yoga with the uniqueness of the natural transformations described in loc.\! cit.\! See for example \cite[Theorem 4.5]{HIGFIRST}. A few more details on a similar line of ideas is in \cite[Theorem 4.9]{THOASY}.
\end{proof}
We can now compare $\hoKKex$ with $E^G_{\operatorname{class}}$.
\begin{thm}\label{thm: hoKKex is Eclass}
    We have a natural equivalence of additive, symmetric monoidal categories 
    \[
    \hoKKex \simeq \EBGHT.
    \]
\end{thm}
\begin{proof}
    The proof of \Cref{thm: hoKK inf is KKKas} can be readily adapted to this case. The main check one needs to make is that $\EBGHT$ has countable coproducts, and the functor $\EBGHT$ takes countable direct sums in $\RALGsep$ to countable coproducts in $\EBGHT$. But this is the content of \cite[Proposition 7.1]{HGT}. The Real structure poses no additional difficulty.
\end{proof}
\subsection{A technical result of Dadarlat-Loring}
In this subsection, we prove the following technical result, which will be handy in the later subsections when we study graded $E$-theory. The complex, non-equivariant case of the result is due to Dadarlat-Loring (\cite{DL})
\begin{thm}{(Generalized Dadarlat-Loring)}\label{thm: Gen DL}
    Let $A \in \RALGsep$. Suppose that for all $B \in \RALGsep$ the abelian monoid
    \[
    \operatorname{hom}_{\SKGASYG}(A,B)=[[A \otimes \kg,B \otimes \kg]]
    \]
    is a group. Then, the map
    \[
    \Omega:  [[A \otimes \kg, B \otimes \kg]] \rightarrow \EBGHT(A,B)
    \]
    is an isomorphism of abelian groups for all $B \in \RALGsep$.
\end{thm}
\begin{rem}
    An $A$ in $\RALGsep$ satisfying the hypothesis of the above theorem is called \emph{homotopy symmetric}. Note that it is equivalent to asking that the abelian monoid $[[A \otimes \kg,A \otimes \kg]]$ is a group.
\end{rem}
Our proof strategy will be very different from \cite{DL}, except that we too use \Cref{prop: asy split exact functor} in a crucial way. We begin by proving the following result, wherein we use the notation from the last subsection, especially \Cref{cor: Bott inverted}.
\begin{prop}\label{prop: prop for GDL}
    For each homotopy symmetric $A \in \SKGASYG$, the functor
    \[
    \operatorname{hom}_{\SKGASYG}(A,-):\SKGASYG \rightarrow \operatorname{Ab}
    \]
    inverts $\alpha: C_0(\R^{1,1}) \dashrightarrow \mathcal{K}(l^2)$.
\end{prop}
\begin{proof}
    In light of \Cref{prop: asy split exact functor} and the universal property of $KK^G_{\operatorname{Kas}}$, there is a unique dashed arrow rendering the below diagram commutative.
\begin{equation}\label{eqn: hom asy factor through KK}
\begin{tikzcd}
	{\RALGsep} & {\SKGASYG} &&& {\operatorname{Ab}} \\
	& {\KKKAS}
	\arrow[from=1-1, to=1-2]
	\arrow["{\kkKAS}"', from=1-1, to=2-2]
	\arrow["{\operatorname{hom}_{\SKGASYG}(A,-)}", from=1-2, to=1-5]
	\arrow[dashed, from=2-2, to=1-5]
\end{tikzcd}
\end{equation}
We now momentarily ignore the external group $G$. Since $\KK_{\Kas}^{\sep}$  satisfies $(1,1)$-periodicity, we know in light of the Cuntz picture of $\KK_{\Kas}^{\sep}$ that there is a morphism $q\C \xrightarrow{\beta} C_0(\R^{1,1}) \otimes \mathcal{K}(l^2)$ in $\CAlg^\sep$ which implements the $(1,1)$-periodicity. There is also a homomorphism $q\C \xrightarrow{\gamma} \mathcal{K}$ given by the composite
\[
q\C \xrightarrow{\pi_1} \C \xrightarrow{\iota} \mathcal{K}(l^2).
\]
We now claim the following.
\begin{clm}
    The diagram
\begin{equation}\label{eqn: alpha beta gamma}
    \begin{tikzcd}
	{q\mathbb{C}} && {\mathcal{K}(l^2\mathbb{N})} \\
	& {C_0(\R^{1,1}) \otimes \mathcal{K}(l^2\mathbb{N})}
	\arrow["\gamma", from=1-1, to=1-3]
	\arrow["\beta"', from=1-1, to=2-2]
	\arrow["\alpha"', dashed, from=2-2, to=1-3]
\end{tikzcd}
\end{equation}
commutes upto a sign in the $\mathcal{K}(l^2\N)$-stabilized asymptotic homotopy category $L_\mathcal{K}\!\operatorname{Asy}$.
\end{clm}
\begin{proof}
    We consider the following diagram, where the vertical maps are induced by the localization functor $L_\mathcal{K}\!\operatorname{Asy} \rightarrow E_{\operatorname{class}}$ (i.e. the trivial group case of \Cref{cor: Bott inverted}).
\[
\resizebox{\textwidth}{!}{$
\begin{array}{ccccc}
[[q\mathbb{C},C_0(\R^{1,1})\otimes\mathcal{K}(l^2\N)]]
&
\times
&
[[C_0(\R^{1,1})\otimes\mathcal{K}(l^2\N),\mathcal{K}(l^2\N)]]
&
\longrightarrow
&
[[q\mathbb{C},\mathcal{K}(l^2\N)]]\cong\Z
\\[1.2ex]
\downarrow
&&
\downarrow
&&
\downarrow\scriptstyle{\cong}
\\[1.2ex]
E_{\operatorname{class}}
\bigl(q\mathbb{C},C_0(\R^{1,1})\otimes\mathcal{K}(l^2\N)\bigr)
&
\times
&
E_{\operatorname{class}}
\bigl(C_0(\R^{1,1})\otimes\mathcal{K}(l^2\N),
\mathcal{K}(l^2\N)\bigr)
&
\longrightarrow
&
E_{\operatorname{class}}
\bigl(q\mathbb{C},\mathcal{K}(l^2\N)\bigr)\cong\Z
\\[0.5ex]
\cong \Z
&&
\cong \Z
&&
\end{array}
$}
\]
We now justify the isomorphisms appearing on the diagram. That $[[q\C,\mathcal{K}(l^2\N)]] \cong \Z$ follows from 
\[
\Z \xrightarrow[\cong]{\Cref{prop: Cuntz picture k-theory}} [q\C,\mathcal{K}(l^2\N)] \xrightarrow[\cong]{\Cref{rem: qA properties}.6} [[q\C,\mathcal{K}(l^2\N)]].
\]
 We furthermore note that $\gamma$ corresponds to $1$ under this isomorphism; see \Cref{rem: the map in the cuntz picture iso}. The isomorphisms in the bottom row follow since $q\C,\mathcal{K}(l^2\N), C_0(\R^{1,1})$ are all $KK_{\operatorname{Kas}}^{\operatorname{sep}}$-equivalent to $\C$, and $K_0(\C) \cong \Z$. That the right vertical map is an isomorphism follows from the composite map
\[
[q\C,\mathcal{K}(l^2\N)] \xrightarrow{\cong} [[q\C,\mathcal{K}(l^2\N)]] \rightarrow E_{\operatorname{class}}(\C,\C)
\]
being an isomorphism from \Cref{lem: qC special} (we are implicitly also using the universal property of $\mathcal{K}(l^2)$-stabilized homotopy category of separable $C^*$-algebras here). From the discussion preceding the claim, $\beta$ and $\alpha$ map to a generator in $\Z$ under the two left vertical arrows. An argument similar to the one in \Cref{lem: product of bott elements iso} shows that the bottom bilinear map must send $1 \times 1$ to $\pm 1$. The conclusion follows.
\end{proof}
We are now in a position to bring back the group $G$. Since all the algebras in \Cref{eqn: alpha beta gamma} have the trivial $G$-action, we conclude that \Cref{eqn: alpha beta gamma} also commutes in $\SKGASYG$. Furthermore, $\gamma$ and $\beta$ remain equivalences in $\KKKAS$, whence so is $\alpha$. The proof of the proposition follows in light of \Cref{eqn: hom asy factor through KK}.
\end{proof}
\begin{proof}[Proof of \Cref{thm: Gen DL}]
    In light of \Cref{cor: Bott inverted} and general facts about mapping spaces (in our context, mapping abelian groups) out of colocal objects (\cite[Remark 9.9]{KKandE}), it suffices to show that the functor 
    \[
    \operatorname{hom}_{\SKGASYG}(A,-): \SKGASYG \rightarrow \operatorname{Ab}
    \]
    inverts the morphisms $\alpha \otimes 1_B: C_0(\R^{1,1}) \otimes B \dashrightarrow \mathcal{K}(l^2\N) \otimes B$. But this follows upon tensoring \Cref{eqn: alpha beta gamma} with $B$; the diagram continues to commute up to a sign (which is independent of $B$), and the morphisms $q\C \otimes B \xrightarrow{\beta \otimes 1_B} C_0(\R^{1,1}) \otimes \mathcal{K}(l^2\N) \otimes  B$ and $q\C \otimes B \xrightarrow{\gamma \otimes 1_B} \mathcal{K}(l^2\N) \otimes B$ are isomorphisms in $\KKKAS$.
\end{proof}
\begin{rem}
    The following remarks about \Cref{thm: Gen DL} are in order.
    \begin{enumerate}
        \item There is a non-technical way to think about the result. Note that \Cref{cor: Bott inverted} roughly says that one obtains $\EBGHT$ from $\SKGASYG$ by ``inverting Bott". When $A$ is homotopy symmetric, \Cref{eqn: hom asy factor through KK} shows that the functor \[\operatorname{hom}_{\SKGASYG}(A,-): \SKGASYG \rightarrow \operatorname{Ab}\] indeed ``inverts Bott", since it factors through $\KKKAS$. \Cref{prop: prop for GDL} basically makes this line of argument rigorous. And probably this is also the argument which was hinted at the end of Section 4 in \cite{DL}.
        \item A treatment analogous to Section 3.4 would now show that $\EBGHT$ is the full subcategory of group objects of $\SKGASYG$. This seems to not have been mentioned in the literature.
    \end{enumerate}
\end{rem}
\subsection{Comparison with classical, graded $\widehat{E}^G$-theory}
We are finally in a position to discuss classical, graded $\Egrclass$-categories. The main reference for the same is \cite[Section 2]{HG} (in the complex case). Following the formulation given there, we introduce its Real analogue. In what follows, $[[A,B]]_{\operatorname{gr}}$ denotes the homotopy classes of graded, $G$-equivariant asymptotic morphisms from $A$ to $B$. Recall that $\Egrclass$ still denotes the \emph{separable} theory.          
\begin{defi}{(cf. \cite[Chapter 2]{HG})}
    For $A,B \in \RGradedsep$, define a category $\Egrclass$ with objects the objects of $\RGraded$, and morphims from $A$ to $B$ being given by 
    \[
    \Egrclass(A,B):=[[\s \hat\otimes A \hat\otimes \hatkg, B \hat\otimes \hatkg]]_{\operatorname{gr}}.
    \]
    The composition law in $\Egrclass$ is given by the usual formula (see \cite[Definition 2.9]{HG}) using the codiagonal map $\s \xrightarrow{\Delta} \s \hat\otimes \s$. Denote by
    \[
    \egrclass: \RGradedsep \rightarrow \Egrclass
    \]
    the canonical functor induced by precomposition with the homomorphism $\s\xrightarrow{\operatorname{ev}_0} \mathbb{C}$.
\end{defi}
To connect $\Egrclass$ and $\hoEgra$, we need the following lemma.
\begin{lem}\label{lem: last lemma}
    The abelian monoid $[[\s \hat\otimes A \hat\otimes \hatkg, B \hat\otimes \hatkg]]_{\operatorname{gr}}$ is a group for all $A,B \in \RGradedsep$. 
\end{lem}
\begin{proof}
    This follows from an argument very similar to the argument given in the proof of \Cref{thm: spectral model of g eq k theory}.1. See \cite[Lemma 2.1]{HG} for more details.
\end{proof}
\begin{thm}\label{thm: final}
\hfill
\begin{enumerate}
    \item For all $A,B \in \RGradedsep$, we have a natural isomorphism $\ETheory^{G_2}_{\mathrm{class}}(\s \hat\otimes A,B) \cong \Egrclass(A,B)$.
    \item We have an equivalence of additive, symmetric monoidal (for the graded maximal tensor product) categories \[
    \hoEgra \simeq \Egrclass.
    \]
    \item For ungraded $C^*$-algebras $A$ and $B$ in $\RGradedsep$, we have a natural isomorphism \[\Egrclass(A,B) \cong \EBGHT(A,B).\]
\end{enumerate}  
\end{thm}
\begin{proof} 
    Assertion 1 follows immediately from \Cref{thm: Gen DL} and \Cref{lem: last lemma}, keeping in mind that $\hatkg$ is inner graded. Assertion 2 follows from \Cref{thm: hoKKex is Eclass} for the group $G \times C_2$ (with the monoidal structure being the graded maximal tensor product), Assertion 1, and \Cref{lem: KKANA cat}.5. Assertion 3 was shown in the context of our stable $\infty$-categories in \Cref{thm: higson email to johannes}. We are therefore done by Assertion 2.
\end{proof}
We end the article with a remark circling the analogous comparison between $\widehat{KK}_{\operatorname{Kas}}^{G,\operatorname{sep}}$ and $\hoKKgra$.
\begin{rem}\label{last rem}
    In light of the main result of \cite{THOASY}, we can rework the last two subsections with \emph{completely positive} asymptotic morphisms, and obtain that for all $A,B \in \RGradedsep$,
    \[
    \hoKKgra(A,B) \cong [[\s \hat\otimes A \hat\otimes \hatkg, B \hat\otimes \hatkg]]_{\operatorname{gr}}^{\operatorname{cp}}.
    \]
    However, unlike the definition of $\Egrclass$ (which on the face of it is by asymptotic morphisms), the definition of $\KKKASgr$ is in terms of Kasparov cycles. Therefore to complete the circle of ideas, we would need to obtain a natural isomorphism
    \[
    \KKKASgr(A,B) \cong [[\s \hat\otimes A \hat\otimes \hatkg, B \hat\otimes \hatkg]]_{\operatorname{gr}}^{\operatorname{cp}}.
    \]
    It is not clear to the author how to get a (clean) natural map in either direction. The problem seems to be that completely positive asymptotic morphisms do not naturally arise in practical situations.

    Nevertheless, Meyer has proven using different methods in \cite[Section 7]{MEYER} the analogue of the first two assertions of \Cref{thm: final} for $\widehat{KK}_{\operatorname{Kas}}^{G,\operatorname{sep}}$, whence the analogue of Assertion 3 of \Cref{thm: final} for $\KKKASgr$ also follows from \Cref{thm: higson email to johannes}. 
\end{rem}
\printbibliography
\end{document}